\documentclass[reqno,oneside,12pt]{amsart}

\usepackage[T1]{fontenc}
\usepackage{times,mathptm}
\usepackage{amssymb,epsfig,verbatim,xypic}
\theoremstyle{plain}
\newtheorem{thm}{Theorem}[section]

\newtheorem{cor}[thm]{Corollary}
\newtheorem{pro}[thm]{Proposition}
\newtheorem{lem}[thm]{Lemma}

\theoremstyle{definition}

\newtheorem{eg}[thm]{Example}
\newtheorem{rem}[thm]{Remark}

\newenvironment{thm-A}
{{\vs \noindent \bf Theorem~A.$\,$}\it}{\vs}

\newenvironment{thm-*}
{{\vs \noindent \bf Theorem.$\,$}\it}{\vs}

\newenvironment{thm-B}
{{\vs \noindent \bf Theorem~B.$\,$}\it}{\vs}

\newenvironment{thm-C}
{{\vs \noindent \bf Theorem~C.$\,$}\it}{\vs}

\newenvironment{thm-D}
{{\vs \noindent \bf Theorem~ D.$\,$}\it}{\vs}

\newenvironment{thm-E}
{{\vs \noindent \bf Theorem~ E.$\,$}\it}{\vs}

\def\vs{\vspace{0.2cm}}

\def\C{\mathbf{C}}
\def\R{\mathbf{R}}
\def\Q{\mathbf{Q}}
\def\Z{\mathbf{Z}}
\def\N{\mathbf{N}}

\def\k{\mathbf{k}}
\def\i{\sf{i}}
\def\j{\sf{j}}

\def\p{{\bullet}}

\def\P{\mathbb{P}}

\def\Bub{{\mathcal{B}}}

\def\z{{\sf{Z}}}
\def\zz{{\mathcal{Z}}}

\def\H{\mathbb{H}}

\def\B{\mathcal{B}}

\def\Sym{\sf{Sym}}

\def\Aut{{\sf{Aut}}}

\def\Pic{{\sf{Pic}}}

\def\NS{{\sf{NS}}}

\def\Cr{{\sf{Cr}}}

\def\Cent{{\sf{Cent}}}
\def\End{{\sf{End}}}

\def\PGL{{\sf{PGL}}}
\def\GL{{\sf{GL}}}

\def\Isom{{\sf{Isom}}}

\def\kf{{\sf{\Omega}}}

\def\Bir{{\sf{Bir}}}
\def\kod{{\sf{Kod}}}

\def\da{\dasharrow}

\def\Ind{{\text{Ind}}}

\def\ax{{\sf{Ax}}}

\def\dist{{\sf{dist}}}

\def\deg{{\sf{deg}}}
\def\mcdeg{{\sf{mcdeg}}}

\def\Pis{{\sf{Pis}}}
\def\Sal{{\sf{Sal}}}
\addtocounter{section}{0}             
\numberwithin{equation}{section}       
\begin{document}
\setlength{\baselineskip}{0.491cm}        
%
%

\title[Dynamical degrees]
{Dynamical Degrees of Birational transformations of projective surfaces}
\thanks{The first author acknowledge support by the Swiss National Science Foundation Grant  "Birational Geometry" PP00P2\_128422 /1 and both authors acknowledge support by the French National Research Agency Grant "BirPol", ANR-11-JS01-004-01. }
\author{J\'er\'emy Blanc}
\author{Serge Cantat}

\address{IRMAR, UMR 6625 du CNRS\\
         Universit\'e de Rennes I\\
         35042 Rennes, France}
\email{cantat@univ-rennes1.fr}

\address{Mathematisches Institut\\
Universit\"at Basel\\
 Rheinsprung 21\\
  4051 Basel, Switzerland}
\email{Jeremy.Blanc@unibas.ch}

%
%
\begin{abstract}
The dynamical degree $\lambda(f)$ of a birational transformation $f$ measures the exponential growth rate of the degree
of the formulae that define the $n$-th iterate of $f$. We study the set of all dynamical degrees of all birational transformations
of projective surfaces, and the relationship between the value of $\lambda(f)$ and the structure of the conjugacy class of $f$.
For instance, the set of all dynamical degrees of birational transformations of the complex projective plane is a closed and well ordered
set of algebraic numbers. 
\end{abstract}
%
%

\maketitle

\setcounter{tocdepth}{1}
%
%

\section{Introduction}\label{part:intro}

Given a birational transformation $f\colon X\dasharrow X$ of a projective surface, defined over a field $\k$, its dynamical 
degree $\lambda(f)$ is a positive real number that measures the complexity of the dynamics of $f$. 
If $\k$ is the field of complex numbers, 
the neperian logarithm $\log(\lambda(f))$
 provides an upper bound for the topological entropy of $f\colon X(\C)\dasharrow X(\C)$ 
and  is equal to it under natural assumptions (see \cite{Bedford-Diller:2005, Dinh-Sibony:Annals}). 
Our goal is to study the structure of the set of all dynamical degrees $\lambda(f)$, when $f$ runs over the group of
all birational transformations $\Bir(X)$ and $X$ over the collection of all projective surfaces. 

The dynamical degree $\lambda(f)$ is invariant under conjugacy. An important feature of our results may be summarized 
by the following slogan: {\sl{Precise knowledge on $\lambda(f)$ provides useful information on the conjugacy class of $f$}}. 
In particular, we shall obtain effective, quantitative bounds for the solutions of certain equations in $\Bir(X)$, like the conjugacy 
problem asking for a solution $h$ of the equation $hfh^{-1}=g$. 

Another motivation of the present paper is to develop the ``dictionary'' between groups of birational transformations of projective surfaces and mapping class groups of higher genus, closed, orientable surfaces. The dynamical degree $\lambda(f)$ plays a role which is similar to the dilatation factor $\lambda(\varphi)$ of pseudo-Anosov mapping classes (see  \S~\ref{par:modular-group} below).

As we shall see, our main results should be compared to two theorems proved by W. Thurston. The first one describes explicitly the set of topological entropies of post-critically finite, continuous, multimodal transformations of the unit interval as the set of logarithms of ``weak Perron numbers''. The second describes the structure of the set of volumes of hyperbolic manifolds of dimension $3$; this set is a countable, non-discrete, and well ordered subset of the real line.

\subsection{Dynamical degrees, Pisot numbers, and Salem numbers}

\subsubsection{Dynamical degrees}
Let $X$ be a projective surface defined over an algebraically closed field $\k$. In what follows, 
$\NS(X)$ denotes the N\'eron-Severi group  of $X$. Given a ring ${\mathbf{A}}$, $\NS_{\mathbf{A}}(X)$ stands for $\NS(X)\otimes_\Z {\mathbf{A}}$;
hence, $\NS_\Z(X)$ coincides with $\NS(X)$.

Let $f$ be
a birational transformation of $X$ defined over $\k$. It determines an endomorphism
$f_*\colon \NS(X)\to \NS(X)$, and the {\bf{dynamical degree}} $\lambda(f)$ of $f$ 
is defined as the spectral radius of the sequence of endomorphisms $(f^n)_*$, as $n$ goes to $+\infty$. More
precisely, once a norm $\parallel \cdot \parallel$ has been chosen on the real vector space $\End(\NS_\R(X))$, 
one defines
\[
\lambda(f) = \lim_{n\to \infty} \parallel (f^n)_* \parallel ^{1/n} \; ; 
\]
this limit exists, and does not depend on the choice of the norm.
Moreover, for every ample divisor $D\subset X$
\[
\lambda(f)=\lim_{n\to \infty} \left({D\cdot (f^n)_* D} \right)^{1/n},
\]
where $C\cdot D$ denotes the intersection number between divisors or divisor classes. 
By definition, $f$ is {\bf{loxodromic}} if $\lambda(f)>1$.

The {\bf{dynamical spectrum}} of $X$ is defined as the set 
\[
\Lambda(X)=\{\lambda(f)\mid f\in \Bir(X)\}.
\]
If one wants to specify the field $\k$, one may denote the dynamical spectrum by $\Lambda(X,\k)$.

\begin{eg} The N\'eron-Severi group of $\P^2_\k$ coincides with the Picard group $\Pic(\P^2_\k)$, has rank
$1$, and is generated by the class $e_0$ of a line: 
\[
\NS(\P^2_\k)=\Pic(\P^2_\k)=\Z e_0.
\]
Fix a choice of homogeneous coordinates $[x:y:z]$ on the projective plane~$\P^2_\k$. 
Let $f$ be an element of $\Cr_2(\k)$. One can then find three homogeneous polynomials
$P$, $Q$, and $R$ in the variables $(x,y,z)$, of the same degree $d$, and without common
factor of positive degree, such that
\[
f([x:y:z])=[P(x,y,z):Q(x,y,z):R(x,y,z)].
\]
This degree $d$ does not depend on the choice of homogeneous coordinates; it is
denoted by $\deg(f)$ and called the degree of $f$. On $\Pic(\P^2_\k)$, $f$ acts by multiplication by
$\deg(f)$; thus, we have $\lambda(f)=\lim\deg(f^n)^{1/n}.$
For instance, the {\bf{standard quadratic involution}}
\[
\sigma([x:y:z])=[\frac{1}{x}:\frac{1}{y}:\frac{1}{z}]=[yz:zx:xy].
\]
satisfies $\deg(\sigma^n)=1$ or $2$, according to the parity of $n$; hence $\lambda(\sigma)=1$. 
\end{eg}

\subsubsection{Pisot and Salem numbers (see \cite{BDG:Book})}\label{par:Intro-Pisot-Salem}
A {\bf{Pisot number}} is an algebraic integer $\lambda \in \, \,  ]1,\infty[$ whose other Galois conjugates
lie in the open unit disk; the set of Pisot numbers includes all integers $d\geq 2$ as well as all reciprocal quadratic integers 
$\lambda>1$. A {\bf{Salem number}} is an algebraic integer $\lambda\in \, \, ]1,\infty[$ whose other Galois conjugates
are in the closed unit disk, with at least one on the boundary; hence,
the minimal polynomial of $\lambda$ has at least two complex conjugate roots on the unit circle, its roots are permuted by the
involution $x\mapsto 1/x$, and its degree
 is at least $4$. We denote by $\Pis$ the set of Pisot numbers and by $\Sal$ the set of Salem numbers. 

It is known that $\Pis$ is a closed subset of the real line. It is contained in the closure of $\Sal$, and its infimum 
is equal to $\lambda_P\simeq 1.324717$,  the unique root $\lambda_P>1$ 
of the cubic equation $ x^3=x+1;$
this Pisot number is known as the {\bf{plastic number}}, or {\bf{padovan number}}. 
The smallest accumulation point of $\Pis$ is the golden mean $\lambda_G=(1+\sqrt{5})/{2}$;
all Pisot numbers between $\lambda_P$ and $\lambda_G$ have been listed. 

Our present knowledge of Salem numbers is much weaker. Conjecturally, the infimum of $\Sal$ is larger than 
$1$, and should be equal to the {\bf{Lehmer number}}, i.e.~to the Salem number $\lambda_L\simeq 1.176280$ 
obtained as the unique root $>1$ of the irreducible polynomial 
$
x^{10} + x^9 - x^7 - x^6 - x^5 - x^4 - x^3 + x + 1.
$

\subsubsection{Dynamical degrees and algebraic stability}\label{par:dyna-deg-alg-stab}
One says that $f\in \Bir(X)$ is {\bf{algebraically stable}} when the endomorphism $f_*$ of the N\'eron-Severi group $\NS(X)$ satisfies 
\begin{equation}\label{eq:alg-stable}
(f^n)_*=(f_*)^n
\end{equation}
for all positive integers $n$. If $f$ is algebraically stable, then $f^{-1}$ is also algebraically stable and $\lambda(f)$ is the spectral radius of the endomorphism $f_*$ of $\NS(X)$;
in particular, $\lambda(f)$ is an algebraic integer. Diller and Favre proved in \cite{Diller-Favre:2001} that 
every birational transformation of a projective surface $X$ is conjugate by a birational morphism $\pi:Y\to X$
to an algebraically stable transformation $\pi^{-1}\circ f \circ \pi$. From this fact and the Hodge index theorem, 
they obtained the following result. 

\begin{thm}[Diller and Favre]
Let $\k$ be a field and let $f$ be a birational transformation of a projective
surface defined over $\k$. If $\lambda(f)$ is different from~$1$, then $\lambda(f)$ is a Salem or a Pisot number.
\end{thm}

In this article we initiate the study of the dynamical spectrum $\Lambda(X)$. 
By Diller-Favre Theorem, $\Lambda(X)$ splits in two parts, its Pisot part $\Lambda^{P}(X)$
and its Salem part $\Lambda^S(X)$. The problem is to describe which numbers can appear in each of
these sets, as well as the relationship between these two sets. 

\begin{eg}
When $f$ is an algebraically stable  transformation of $\P^2_\k$, one gets $\lambda(f)=\deg(f)$.
For instance, the automorphism $h$ of the affine plane defined by 
$h(X,Y) = (Y, X+ Y^d)$ extends to a birational map of the projective plane such that $\deg(h^n)=d^n$
for all $n\geq 0$. In particular, $\Lambda(\P^2_\k)$ contains all integers $d\geq 1$, 
for all fields $\k$.
\end{eg}

\begin{eg}\label{eg:monomial}
Consider the group $\GL_2(\Z)$ acting by (monomial) automorphisms of the 
multiplicative group $\k^*\times \k^*$: If 
\[
A=\left( \begin{array}{cc} a & b \\ c& d \end{array} \right)
\]
is an element of $\GL_2(\Z)$ and $(X,Y)$ denotes the coordinates on $\k^*\times \k^*$, the
automorphism associated to $A$ is defined by
$f_A(X,Y)=(X^aY^b, X^cY^d).$
This provides an embedding of $\GL_2(\Z)$ in the automorphism group
$\Aut(\k^* \times \k^*)$, and thus in $\Bir(\P^2_\k(\k))$. 

For every $A$ in $\GL_2(\Z)$, the dynamical degree of $f_A$ is equal
to the spectral radius of the matrix $A$, i.e.~to the modulus of its unique eigenvalue 
$\lambda$ with $\vert \lambda\vert \geq1$;
this implies that $f_A$ is not an algebraically stable transformation of 
$\P^2_\k$ as soon as $\lambda(f_A)>1$, because $\lambda(f_A)$ is not
an integer in that case. 

As a byproduct of this example, the dynamical spectrum of the plane contains all reciprocal 
quadratic integers, i.e.~all roots $\lambda >1$ of equations
$
x^2+1 = tx
$
with $t$ in $\Z$. 
\end{eg}

\subsection{Salem numbers and automorphisms}

The dynamical degree of an automorphism, if different from $1$,  is either a quadratic number 
or a Salem number (see \cite{Diller-Favre:2001}). Here we prove a converse statement. 

\begin{thm-A}
Let $\k$ be an algebraically closed field. Let $f$ be a birational transformation of
a projective surface $X$, defined over $\k$. If $\lambda(f)$ is a Salem number, 
there exists a projective surface $Y$ and a birational mapping $\varphi\colon Y\da X$ such that $\varphi^{-1}\circ f \circ \varphi$
is an automorphism of $Y$. 
\end{thm-A}

Thus, one can decide whether a birational transformation is conjugate to an automorphism by looking
at its dynamical degree, except when this degree is $1$ or a quadratic integer. For the quadratic 
case, Examples~\ref{eg:Kummer} and~\ref{eg:counterexample-quad} show that there are quadratic
integers which are simultaneously realized as dynamical degrees of automorphisms and 
of birational transformations that cannot be conjugate to an automorphism. See Remark~\ref{rem:deg1} for birational transformations with  dynamical degree equal to $1$.

Once Theorem~A is proved, three corollaries can be deduced from results of McMullen
and the second author (see \cite{McMullen:Blowups} and \cite{Cantat:Annals}). 
The first corollary (see \S~\ref{par:firstgap})
is a  {\bf{spectral gap property}} for dynamical degrees: {\sl{There is no dynamical degree in the interval $]1, \lambda_L[$}}.
The second corollary does not seem to be related to values of dynamical degrees,
but the simple proof given here makes use of the spectral gap. It asserts that {\sl{the centralizer, in the group $\Bir(X)$, of a loxodromic element $f$ is finite by cyclic}} (see \S~\ref{par:centralizer}).
The third consequence is an effective and explicit bound for the optimal degree of a conjugacy (see \S~\ref{par:conjugacy-degree-bound}):
\begin{cor}\label{coro:conjug-degree-bound}
Two loxodromic elements $f$, $g\in \Bir(\P^2_\k)$ of degree $\leq d$ are conjugate if and only if they are conjugate by  an 
element $h$ of degree $\leq (2d)^{57}$.
\end{cor}

\subsection{From projective surfaces to the projective plane}$\,$
Non rational surfaces are easily handled with.

\begin{thm-B} Let $\k$ be an algebraically closed field. Let $X$ be a projective surface defined over $\k$. 
If $X$  is not rational, then
\begin{enumerate}
\item $\Lambda(X)=\{1\}$ if $X$ is not birationally equivalent to an abelian surface, a
K3 surface, or an Enriques surface;

\item $\Lambda(X)\setminus \{1\}$ is made of quadratic integers and of Salem numbers of degree at most $6$  
(resp. $22$, resp. $10$) if $X$ is an abelian surface (resp. a K3 surface, resp. an Enriques surface).

\end{enumerate} 
The union of all dynamical spectra $\Lambda(X, \k)$, for all fields and all surfaces which are not
geometrically rational, is a closed discrete subset of the real line.
\end{thm-B}

\begin{rem}
When the characteristic of the field $\k$ vanishes, the degree bounds of Assertion (2) become $4$, $20$, and $10$ (in place
of $6$, $22$, and $10$).
\end{rem}

This result, proved in Section \ref{par:not-rational}, shows that the most interesting case is provided by rational surfaces. 
Thus, in the following statements, one can  assume that $X$ is birationally equivalent to the projective plane 
$\P^2_\k$; the dynamical  spectrum  is then equal to the set $\Lambda(\P^2_\k)$ of dynamical 
degrees of elements of the {\bf{Cremona group}}
$$
\Cr_2(\k)=\Bir(\P^2_\k).
$$

\subsection{Degrees and conjugacy classes}

\subsubsection{Minimal degree in the conjugacy class}

Given an element $f$ of $\Bir(\P^2_\k)$, define the {\bf{minimal degree}} of $f$ in its conjugacy class as
the positive integer
\[
\mcdeg(f)=\min \deg(g\circ f \circ g^{-1})
\]
where $g$ describes $\Bir(\P^2_\k)$ (thus, $\mcdeg(f)$ depends on the field and may decrease after a field 
extension). The function $\mcdeg$ is constant on conjugacy classes, 
and  
\[
\lambda(f)\leq \mcdeg(f)\leq \deg(f)
\]
for all birational transformations of the plane.
One of our main goals is to provide the following reverse inequality:

\begin{thm-C}
Let $\k$ be an algebraically closed field and let $f$ be a birational transformation of the plane $\P^2_\k$. 
\begin{enumerate}
\item If $\lambda(f)\geq 10^6$ then $\mcdeg(f)\leq 4700 \, \lambda(f)^5$.

\item If $\lambda(f)>1$, then $\mcdeg(f)\leq \cosh(18+345\log(\lambda(f)))\leq e^{18}\lambda(f)^{345}$. 
\end{enumerate}
\end{thm-C}

On the other hand, there are sequences of elements $f_n\in \Bir(\P^2_\k)$ such that 
$\mcdeg(f_n)$ goes to $+\infty$ with $n$ while $\lambda_1(f_n)=1$ for all $n$.
\subsubsection{Well ordered sets}
The set $\Lambda(\P^2_\k)$ is a subset of $\R_+$ and, as
such, is totally ordered. The following statement, which follows from Theorem~C, asserts that $\Lambda(\P^2_\k)$ is {\bf{well ordered}}: Every non-empty subset of $\Lambda(\P^2_\k)$ has a minimum; equivalently, it satisfies
the descending chain condition (if $(f_n)_{n\geq 0}$ is a sequence of birational transformations of
$\P^2_\k$ and $\lambda(f_{n+1})\le \lambda(f_n)$ for each $n$,  then  $\lambda(f_n)$ becomes eventually constant).
 
\begin{thm-D}
Let $\k$ be an algebraically closed field. 
The dynamical spectrum $\Lambda(\P^2_\k)\subset \R$ is well ordered, and it is closed  if $\k$ is uncountable. 
\end{thm-D}

In Theorem~\ref{thm:sal-pis}, we also show that {\sl{$\Lambda^P(\P^2_\k)$ is contained in the closure
of $\Lambda^S(\P^2_\k)$ if $\k$ is algebraically closed and of characteristic $0$}}.

From Theorem~B and Theorem~D, one obtains the existence of gaps in the dynamical spectrum of projective surfaces: There are small intervals of
real numbers that contain infinitely many Pisot and Salem numbers, but do not contain any dynamical degree. 

\begin{cor}
Let  $\Lambda$ be the set of all dynamical degrees of birational transformations 
of projective surfaces, defined over any field. Then, 
\begin{itemize}
\item[(1)] $\Lambda$ is a well ordered subset of $\R_+$;
\item[(2)] if $\lambda$ is an element of $\Lambda$, there is a real number $\epsilon>0$ such that $]\lambda, \lambda+\epsilon]$ 
does not intersect $\Lambda$;
\item[(3)] there is a non-empty interval $]\lambda_G, \lambda_G+\epsilon]$, on the right of 
the golden mean, that contains infinitely many Pisot and Salem numbers
but does not contain any dynamical degree. 
\end{itemize}
\end{cor}
In fact, gaps as in the third assertion of this corollary occur infinitely often, because there are infinitely many Pisot numbers that
are limits of Pisot numbers from the right.
\subsection{Organization of the paper}

Section~\ref{par:SAL-AUTO} provides a proof of Theorem~A and its first corollary, the absence of dynamical degree 
between $1$ and $\lambda_L\simeq 1.17628$.
Theorem~B is proved in Section~\ref{par:not-rational}; this may be skipped on a first reading. 
Section~\ref{par:BUBLES} introduces the bubble space and an infinite dimensional hyperbolic space on which $\Bir(X)$ acts by isometries; as a first application, we obtain two  
corollaries of Theorem~A. Section~\ref{par:WEYL} contains preliminary results on the infinite Weyl group $W_\infty$: This group is a Coxeter group on an infinite set of generators, and plays a crucial technical role in the study of the Cremona group $\Cr_2(\k)$. The proof of Theorem~C is quite difficult even if, in spirit, it is a variation on Noether-Castelnuovo proof of the fact that $\PGL_3(\k)$ and the standard quadratic involution $\sigma$ generate $\Bir(\P^2_\k)$. This proof occupies Section~\ref{par:PROOF-D}, and Section~\ref{par:ALG-FAM} shows how Theorem~D follows from Theorem~C. \tableofcontents

\subsection{Acknowledgements}
Thanks to Nicolas Bergeron, Antoine Chambert-Loir, Julie D\'eserti, H\'el\`ene Esnault, Yves de Cornulier, Vincent Guirardel, Mattias Jonsson, St\'ephane Lamy, Curtis T. McMullen, and  Juan Souto for interesting discussions on this topic. We are also grateful to the referee for his careful reading and his suggestions.

\section{Salem numbers and automorphisms}\label{par:SAL-AUTO}

This section is devoted to the proof of Theorem~A. On our way, we introduce basic
definitions that are used all along this article. 

\subsection{Indeterminacy points, homaloidal nets and base points}

Let $X$ be a projective surface defined over an algebraically closed field $\k$. 
Let $f$ be a birational transformation of $X$. We denote by $\Ind(f)$ the set of 
{\bf{indeterminacy points}} of $f$; by convention, it is a proper subset of $X$ and does not
include infinitely near points. 

The {\bf{base points}} of $f$ are defined as follows. Let $D$ be a very ample divisor on $X$ 
and $\vert D \vert$ be the complete linear system containing $D$. The image of $\vert D\vert$
by $f$ is a linear system on $X$ (which, in general, is not complete); when $f$ is an element
of the Cremona group and $D$ is a line in $\P^2_\k$, this linear system $f_*\vert D\vert$ is 
the homaloidal net of $f^{-1}$ (see \cite{Alberich:LNM}). The set of base points of $f^{-1}$ 
(resp. the base ideal of $f^{-1}$) is defined as the set (resp. the ideal) of base points of this linear
system: Base points may be infinitely near, and come with a multiplicity. The notion of base point 
does not depend on the choice of a very ample divisor, but the multiplicities of the base points depend
on this choice. 

This distinction between base points and indeterminacy points is just used to emphasize 
the arguments for which it is important to know whether the point is a proper point of $X$ or not.

\subsection{Algebraic stability and the intersection form}

 One says that $f$ is algebraically
stable if the sequence $(f^n)_*$ of endomorphisms of $\NS(X)$ satisfies $(f^n)_*=(f_*)^n$
for all integers $n$ (cf. \S \ref{par:dyna-deg-alg-stab}). As explained in \cite{Diller-Favre:2001}, $f$ is not algebraically stable
if, and only if there is an indeterminacy point $q$ of $f^{-1}$ and a non-negative integer
$k$ such that $f$ is well defined at $q$, $f(q)$, $\ldots$, $f^{k-1}(q)$, and $f^k(q)$ is an indeterminacy
point of $f$. Blowing-up $q$, $\ldots$, $f^k(q)$,  the number of such ``bad'' indeterminacy 
points decreases and, in a finite number of steps, one constructs a birational morphism $\pi\colon X'\to X$ such
that $\pi^{-1}\circ f \circ \pi$ is algebraically stable (see \cite{Diller-Favre:2001} for this proof). 

Lets us now assume that $f$ is algebraically stable. The dynamical degree $\lambda(f)$ 
is then equal to the spectral radius of $f_*\in \End(\NS(X))$ and also to the spectral 
radius of $f^*=(f^{-1})_*$ because these endomorphisms are adjoint for the intersection form:
\[ 
f_*C\cdot D= C \cdot f^* D
\]
for all pairs $(C,D)$ of divisor classes. 

Factorize $f$ as $f=\epsilon \circ \pi^{-1}$
where $\pi\colon Z\to X$ and $\epsilon\colon Z\to X$ are birational morphisms. Write $\pi$
as a composition $\pi_1\circ \ldots \circ \pi_m$ of (inverse of) point blow-ups, and denote by
$F_j\subset Z$ the total transform of the indeterminacy point of $\pi_j^{-1}$ under the map
$\pi_j\circ \ldots \circ \pi_m$. Then, denote by $E_j$ the direct image of $F_j$ by $\epsilon$, 
for $1\leq j\leq m$. Each $E_j$, if not zero, is an effective divisor. According to \cite{Diller-Favre:2001}, 
Theorem~3.3, one has 
\begin{equation}\label{eq:push-pull}
f_*f^*C= C + \sum_{j=1}^m (C\cdot E_j)E_j
\end{equation}
for all curves (resp. divisor classes) $C$ in $X$; this formula corresponds to the following fact: 
The preimage of $C$ goes through the base points $p_j$ of $f$ with 
multiplicity $(C\cdot E_j)$; thus, the total transform of $f^{-1}C$ by $f$ contains both $C$
and $\sum_j (C\cdot E_j) E_j$. Taking intersection, and using that $f^*$ and $f_*$ are adjoint 
endomorphisms of $\NS(X)$ for the intersection form, one gets
\begin{equation}\label{eq:push-pull-intersection}
f^*C\cdot f^*C= C\cdot C + \sum_{j=1}^m (E_j\cdot C)^2.
\end{equation}
In particular, $f^*$ increases self-intersections. This property and Hodge index
theorem, according to which the intersection form has signature $(1,\rho(X)-1)$, 
are responsible for $\lambda(f)$ being a Pisot or Salem number (see \cite{Diller-Favre:2001}, Theorem~5.1). 

\subsection{Eigenvectors and automorphisms}

Since $X$ has dimension $2$, one easily shows that $f^*$ and $f_*$ preserve the
pseudo-effective and  nef cones of $\NS_\R(X)$. 
Assume that the dynamical degree $\lambda(f)$ is larger than $1$. 
Perron-Frobenius theorem assures the existence of an eigenvector $\Theta^+_X(f)$ for $f^*$ in the nef cone 
of $\NS_\R(X)$
such that $f^*\Theta^+_X(f)=\lambda(f) \Theta^+_X(f)$; moreover, this vector is unique
up to scalar factor (see \cite{Diller-Favre:2001}).  

\begin{thm}[Diller-Favre]\label{thm:isotropic-automorphism}
Let $X$ be a projective surface, and $f$ be a birational transformation of $X$, both
defined over an algebraically closed field $\k$. Assume that the dynamical degree $\lambda(f)$ is larger than $1$.
Then
\begin{enumerate}
\item $\Theta^+_X(f) \cdot \Theta^+_X(f)=0$ if and only if
 $\Theta^+_X(f)\cdot E_j=0$ for all $E_j$; 
 \item if  $\Theta^+_X(f) \cdot \Theta^+_X(f)=0$, there exists a birational morphism $\eta\colon X\to Y$,
such that $\eta\circ f \circ \eta^{-1}$ is an automorphism of $Y$.
\end{enumerate}
\end{thm}

\begin{proof}[Sketch of the proof]
Equation \eqref{eq:push-pull-intersection}
and the eigenvector property 
\[
f^*\Theta^+_X(f)=\lambda(f)\Theta^+_X(f)
\]
imply that 
\[
(\lambda(f)^2-1) \Theta^+_X(f)\cdot \Theta^+_X(f) = \sum_{j=1}^m (E_j\cdot \Theta^+_X(f))^2.
\]
Hence, all divisors $E_j$
are orthogonal to $\Theta^+_X(f)$ if, and only if, $\Theta^+_X(f)$ is an isotropic vector. 

Let us now prove the second assertion. By the first assertion, every $E_j$ is orthogonal
to $\Theta^+_X(f)$; since the $E_j$ are effective and $\Theta^+_X(f)$ is nef, all irreducible
components of the $E_j$ are orthogonal to $\Theta^+_X(f)$; in other words, the $\Q$-vector subset of $\NS_\Q(X)$ generated by 
the irreducible components of the divisors $E_j$ is contained in the orthogonal complement $\Theta^+_X(f)^\perp$ 
of the isotropic vector $\Theta^+_X(f)^\perp$. 
On $\Theta^+_X(f)^\perp$, the intersection form is negative and its kernel is the line generated by  $\Theta^+_X(f)$.

From Equation \eqref{eq:push-pull}, one gets $f_*^{k}\Theta^+_X(f)=\lambda(f)^{-k}\Theta^+_X(f)$. Since
$\lambda(f)>1$ and $f_*$ preserves the lattice $\NS_\Z(X)$ one deduces that $\Theta^+_X(f)$ is irrational:
no scalar multiple of $\Theta^+_X(f)$ is contained in $\NS_\Z(X)$. 
Thus, the intersection form 
is negative definite on the $\Q$-vector space generated
by all classes of irreducible components of the divisors $E_j$. 

From Grauert-Mumford contraction theorem (see \cite{BPV:book}, thm$.$ 2.1 p$.$ 91), 
there is a birational morphism $\eta_0\colon X\to  Y_0$ which contracts simultaneously all these components. 
Let $f_0$ be the birational transformation $\eta_0\circ f \circ \eta_0^{-1}$.  
Since $\Theta^+_X(f)$ does not intersect the curves which are contracted by $\eta_0$, the class 
$(\eta_0)_*\Theta^+_X(f)\in \NS_\R(Y)$ is both isotropic and an eigenvector for $(f_0)_*$ with eigenvalue $\lambda(f)$. 
One can thus iterate this process until $f_0^{-1}$ does not contract any curve, i.e.~$f_0$ is an 
automorphism of $Y_0$. If $Y_0$ is singular, and $Y$ is a minimal desingularization of $Y_0$, 
 $f_0$ lifts to an automorphism $f_Y$ of $Y$; one can then show that there is an intermediate birational morphism
$\eta\colon X\to Y$ such that $\eta \circ f \circ \eta^{-1}=f_Y$. This concludes the proof. \end{proof}

\begin{eg}\label{eg:Kummer} Let $E$ be the elliptic curve associated to the lattice of Gaussian (resp. Eisenstein)
integers:
\[
E=\C/\Z[\i] \quad ({\text{resp.}} \C/ \Z[\j])
\]
where $\i^2=-1$ (resp. $\j^3=1$, $\j\neq 1$). Let $A$ be the abelian surface $E\times E$. The group 
$\GL_2(\Z[\i])$ (resp. $\GL_2(\Z[\j])$) acts by automorphisms on $A$, and commutes to 
$\nu(x,y)=(\i x , \i y)$ (resp. $\nu(x,y)=(\j x, \j y)$). As a consequence $\PGL_2(\Z[\i])$ (resp. $\PGL_2(\Z[\j])$) 
acts by automorphisms on the (singular) rational surface $X_0=A/\nu$, and on its minimal desingularisation $X$.
The surface $X$ being rational, this construction provides an embedding of $\PGL_2(\Z[\i])$ (resp. $\PGL_2(\Z[\j])$)
into the Cremona group. If $M$ is an element of the linear group $\GL_2(\Z[\i])$ (resp. $\GL_2(\Z[\j])$), the
associated birational transformation $g_M$ has dynamical degree
\[
\lambda(g_M)=\lambda(M)^2
\] 
where $\lambda(M)$ is the spectral radius of the matrix $M$.
\end{eg}

\begin{eg}\label{eg:counterexample-quad}
Start with the matrix $C$ defined by 
\[
C=\left( \begin{array}{cc} 1 & 1 \\ 1 & 0 \end{array}\right).
\]
Its spectral radius is the Golden mean $\lambda_G$. The square of $\lambda_G$ can be
realized as the dynamical degree of the monomial map $f_{C^2}$ associated to the second power $C^2$ of $C$, as described in Example~\ref{eg:monomial}. It is also realized as the dynamical degree of the transformation $g_C$ from Example \ref{eg:Kummer}. The birational transformation $f_{C^2}$ is not
conjugate to an automorphism of a rational surface $Y$, while $g_C$ is. 
\end{eg}

\begin{rem}\label{rem:deg1}
The previous two examples show that {\sl{Theorem~A does not
extend to quadratic integers}}. 

If $f$ is a birational transformation of a projective surface $X$ with $\lambda(f)=1$, then  $\parallel (f^n)_*\parallel$
is bounded, or it grows linearly with $n$, or it grows quadratically: In the first and third cases, $f$ is  conjugate to 
an automorphism of a projective surface $Y$ by some birational transformation $\varphi\colon Y \dasharrow X$; 
in the second case, $f$ is not conjugate to an automorphism (see Section~\ref{par:GDF} an \cite{Diller-Favre:2001}). Thus, again, the 
``degree growth'' determines whether $f$ is, or not, conjugate to an automorphism. 
\end{rem}

\subsection{Proof of Theorem~A}
Let us now prove Theorem~A. Assume $\lambda(f)$ is a Salem number. 
Let $\chi(t)\in \Z[t]$  be the minimal polynomial of $\lambda(f)$. By assumption, 
there exists a root $\alpha$ of $\chi$ with modulus $1$; one can thus fix an
automorphism $\sigma$ of the field of complex numbers such that $\sigma(\lambda(f))=\alpha$. 

By Diller-Favre Theorem, we may assume that $f$ is algebraically stable. The eigenvector $\Theta^+_X(f)$ 
corresponds to the eigenvalue $\lambda(f)$; as such, it may be taken in $\NS_L(X)$, where
$L$ is the splitting field of $\chi$. Our goal is to show that $\Theta^+_X(f)$ is orthogonal to 
all $E_j$, $1\leq j \leq m$; the conclusion
will follow from Theorem~\ref{thm:isotropic-automorphism}. 

The automorphism $\sigma$ of the field $\C$ acts on $\NS_\C(X)$, preserving $\NS(X)$ pointwise. 
Apply $\sigma$ to both sides of $f^*\Theta^+_X(f)=\lambda(f)\Theta^+_X(f)$; since $f^*$ is defined over $\Z$, 
one obtains 
\[
f^* \Psi = \alpha \Psi, \quad {\text{ with }} \quad \Psi= \sigma(\Theta^+_X(f)).
\]
Since the divisor classes of the $E_j$ are in $\NS(X)$, all of them are $\sigma$-invariant. Thus, 
applying $\sigma$ to  Equation \eqref{eq:push-pull} we get 
\[
f_*f^* \Psi= \Psi +\sum_{j=1}^m (\Psi\cdot E_j) E_j.
\]
Taking intersection with the complex conjugate $\bar{\Psi}$ of $\Psi$, and using $f^*\Psi=\alpha \Psi$, we get
\[
(\alpha \bar{\alpha})\Psi\cdot \bar{\Psi} = f^* \Psi\cdot f^*\bar{\Psi} = \Psi\cdot \bar{\Psi} + \sum_{j=1}^m \vert  E_j\cdot \Psi\vert ^2.
\]
Since $\alpha$ has modulus  $1$, all intersections $E_j\cdot \Psi$ vanish and, applying $\sigma$ again, 
we deduce that $\Theta^+_X(f)\cdot E_j=0$ for all $1\leq j\leq m$. This concludes the proof.

\subsection{Salem numbers in $\Lambda(\P^2_\k)$}\label{par:Salem-in-Cremona}

Let $f$ be an element of $\Cr_2(\k)$ such that $\lambda(f)$ is a Salem
number. According to Theorem~A, $f$ is conjugate to an automorphism $g$ 
of a smooth rational surface $X$; according to Kantor and Nagata \cite{Nagata:I, Nagata:II},
$X$ is a blow-up of $\P^2_\k$ with Picard number $\rho(X)\geq 11$. Thus, the
study of $\Lambda^S(\P^2_\k)$ reduces to the following question: Which Salem numbers 
can be realized as spectral radii of linear transformations
\[
g_*\in \End(\NS_\R(X))
\]
where $X$ describes the set of blow-ups of $\P^2_\k$ and $g$ runs over the
group $\Aut(X)$? Recent results  answer  this question.

Write $X$ as a blow-up of the plane at $n$ points $p_1$, $p_2$, ..., $p_n$;
some of them may be infinitely near points, and we choose indices in such a way that 
$j\geq i$ if $p_j$ is infinitely near $p_i$. Denote by $\pi:X\to \P^2_\k$  the birational 
morphism corresponding to this sequence of blow-ups. Let $e_i\in \NS(X)$ denote
the N\'eron-Severi class of the total transform of $p_i$ under $\pi$ (for $1\leq i\leq n$), and let $e_0\in \NS(X)$ 
be the class of the total transform of a line in $\P^2_\k$. Then 
\[
\NS(X)= \Pic(X)= \Z e_0\oplus \Z e_1 \oplus \ldots \oplus \Z e_n,
\]
and the basis $(e_0, e_1, \ldots, e_n)$ is orthogonal with respect to the intersection form on 
$\Pic(X)$. More precisely, we have
\[
e_0\cdot e_0=1, \, e_i\cdot e_i=-1 \, \text{ if } \, i\geq 1, \, \text{ and } \, e_i\cdot e_j=0 \, \text{ if } \,  i\neq j.
 \]
The canonical class of $X$ is 
\[
k_X=-3e_0+ e_1+e_2+\ldots + e_n.
\]
The automorphism group $\Aut(X)$ acts linearly on $\Pic(X)$, preserves $k_X$, and preserves
the intersection form. As a consequence, $\Aut(X)$ preserves the orthogonal complement $k_X^\perp$ 
of $k_X$ in $\Pic(X)$. The elements 
\begin{eqnarray*}
v_0 & = & e_0 - e_1 - e_2 \\
v_i & = & e_i -e_{i+1}, \, 1\leq i\leq n-1
\end{eqnarray*}
form a basis of $k_X^\perp$, with respect to which the intersection form is given by the Dynkin diagram 
$T_{2,3,n-3}$:

\begin{figure}[ht]
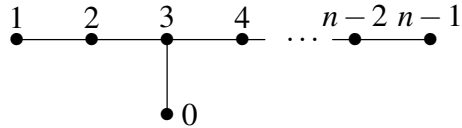

\xy
(-35,0)*{};
@={(0,0),(10,0),(20,0),(30,0),(45,0),(55,0),(20,-10)}@@{*{\bullet}};
(0,0)*{};(33,0)*{}**\dir{-};(42,0)*{};(55,0)*{}**\dir{-};
(20,0)*{};(20,-10)*{}**\dir{-};(38,0)*{\ldots};
(0,3)*{1}; (10,3)*{{2}}; (20,3)*{3};(30,3)*{4}; (45,3)*{{n-2}};(55,3)*{{n-1}};(23,-10)*{0};
\endxy
\caption{{\sf{Coxeter-Dynkin diagram of type $T_{2,3,n-3}$}}}\label{fig:Dynkin}
\end{figure}
In other words, 
\begin{eqnarray*}
v_k\cdot v_k & = & -2, \quad {\text{for all indices}}\, \,    k, \\
v_i \cdot v_j & = & 0, \quad  {\text{if the vertices}} \, \,    i \, \,   {\text{and}} \, \,    j \, \,   {\text{are not linked by an edge}}, \\
v_i\cdot v_j & = & 1, \quad  {\text{if the vertices}} \, \,   i \, \,   {\text{and}} \, \,   j \, \,   {\text{are the endpoints of an edge}}.
\end{eqnarray*}
The Weyl (or Coxeter) group $W_X$ of $X$ is the group of orthogonal transformations of $\Pic(X)$ generated 
by the involutions 
\[
s_i:u\mapsto u+ (u\cdot v_i)v_i, \quad 0\leq i \leq n-1.
\]
This group preserves the orthogonal decomposition $\Pic(X)=\Z k_X\oplus k_X^\perp$
and is isomorphic to the Coxeter group $W_n$ of the Dynkin diagram $T_{2,3,n-3}$. It turns out
that the definition of $W_X$ does not depend on the choice of the realization of $X$ as a
blow-up of the plane;  as an abstract group, $W_X$ depends only on the Picard number of $X$
(see \cite{DO}).

\begin{thm}[Nagata, McMullen, Uehara] $\, $
Let $\k$ be an algebraically closed field, and $n\geq 10$ be an integer. 

\begin{enumerate}
\item Let $X$ be a rational surface obtained from the projective plane $\P^2_\k$ by a sequence of blow-ups. 
 The image of $\Aut(X)$ in $\GL(\NS(X))$ is contained in the Weyl group $W_X$.
\item If ${\sf{char}}(\k)=0$ and if $\Phi$ is an element of $W_n$, there exists a rational surface $Y$ with Picard
number $n+1$ and an element $g$ of $\Aut(Y)$ such that the dynamical degree $\lambda(g)$
of $g$ is equal to the spectral radius $\lambda(\Phi)$ of $\Phi$.
\item There are Salem numbers which are not contained in $\Lambda(\P^2_\k)$ (resp. in $\Lambda(X)$
for any projective surface $X$). 
\end{enumerate}
\end{thm}

When ${\sf{char}}(\k)=0$, this theorem 
shows that the Salem part of $\Lambda(\P^2_\k)$ is described in purely algebraic terms: It coincides 
with the set of spectral radii $\lambda(\Phi)>1$, with $\Phi$ in some $W_n$, $n\geq 10$, and this set
does not exhaust all Salem numbers. 

\begin{rem}
Assertion (1) is due to Nagata (see \cite{Nagata:I, Nagata:II}). Assertion (2) is due to Uehara, based on previous works by McMullen and Bedford and
Kim (see \cite{Uehara}). When the characteristic $p$ of the field is positive,  Harbourne proves a similar result, but for $\Phi$ in a normal subgroup $W_n(p)$ 
of $W_n$ of finite index (the index goes to $+\infty$ with $n$, see Example 3.4 in \cite{Harbourne:1987}).
Assertion (3) makes use of Theorem~A to extend a former result of McMullen. More precisely, 
McMullen proves that there are Salem numbers between $\lambda_L$ and $\lambda_P$ which are not realized by eigenvalues of elements
in the Coxeter groups $W_n$ (see \cite{McMullen:Coxeter}), and deduce from this that there are Salem numbers which are not realized
by dynamical degrees of automorphisms of surfaces (see \cite{McMullen:Blowups}); Theorem~A implies that McMullen's result holds for dynamical degrees of birational transformations. \end{rem}

\subsection{Gaps in the dynamical spectrum}\label{par:firstgap}

As announced in the introduction, we can now prove the following corollary to Theorem~A. 

\begin{cor}\label{cor:thm-A-cor1} Let $\k$ be an algebraically closed field. 
\begin{enumerate}

\item If $f$ is a birational transformation of a projective surface $X$ defined over $\k$
and $\lambda(f)$ is in the interval $]1,\lambda_P[$, then $f$ is conjugate
to an automorphism of a projective surface $Y$ by a birational mapping $\phi\colon X\da Y$.

\item There is no dynamical degree in the interval $]1, \lambda_L[$. 

\item If ${\sf{char}}(\k)=0$, the minimum of the dynamical degree $\lambda(f)>1$ among all  birational 
transformations of projective surfaces defined over $\k$ (resp. of $\P^2_\k$) is equal to the Lehmer 
number $\lambda_L\simeq 1.176280$.
\end{enumerate}
\end{cor}

\begin{proof}
Let $f$ be a birational transformation of a projective surface $X$
defined over an algebraically closed field $\k$. 
Assume that the dynamical degree $\lambda(f)$ is a Salem number.
From Theorem~A, $f$ is conjugate to an automorphism of a smooth projective surface. Thus, Assertion (1)
follows from the fact that $\lambda(f)$ is a Salem number if $1< \lambda(f) < \lambda_P\simeq 1.324717$ (see \S~\ref{par:Intro-Pisot-Salem}). 

From Theorem~1.2 in \cite{McMullen:Blowups}, we deduce that $\lambda(f)\geq \lambda_L$, where $\lambda_L$ is the Lehmer number. Since all Pisot numbers are larger than $\lambda_L$, this proves assertion (2).

If ${\sf{char}}(\k)=0$, there is an automorphism $g$ of a rational surface $X$ such that $\lambda(g)=\lambda_L$ (see \cite{Bedford-Kim:2006, Bedford-Kim:2009} and \cite{McMullen:Blowups}); McMullen recently announced that such an example also exists on a projective K3 surface (see \cite{McMullen:Lehmer-K3}). In particular, 
the infimum of all dynamical degrees is a minimum, and is equal to the Lehmer number. This proves (3).
 \end{proof}

\section{Surfaces which are not rational}\label{par:not-rational}

In this section we prove Theorem~B and provide an example of a K3 surface with automorphisms
$f$ whose dynamical degrees $\lambda(f)$ are  Salem numbers of degree $22$. 

Let $X$ be a projective surface defined over an algebraically closed field $\k$. In order to prove Theorem~B, 
we consider the  Kodaira dimension of $X$ and refer to the classification of surfaces in Kodaira dimension $0$ and $-\infty$
(see \cite{BPV:book}).

\subsection{Ruled surfaces}

If the Kodaira dimension of $X$ is $-\infty$ but $X$ is not rational, then $X$
is ruled in a unique, $\Bir(X)$-invariant way. This implies that all elements of 
$\Bir(X)$ have dynamical degree $1$ (see \S~\ref{par:GDF} and Theorem~\ref{thm:Giz-class} below).

\subsection{Minimal models and automorphisms}
If the Kodaira dimension of $X$ is non negative, $X$ admits a unique minimal 
model $X'$. From now on, we replace $X$ by $X'$, so that we now have $\Bir(X)=\Aut(X)$. In particular, 
all elements of $\Lambda(X)\setminus \{ 1\}$ are Salem numbers, obtained
from eigenvalues of linear transformations of $\NS(X)$ (preserving the intersection 
form). 

\subsubsection{Positive Kodaira dimension} If the Kodaira dimension is equal to $2$, the automorphism group is finite, and
$\Lambda(X)$ reduces to $\{1\}$. If the Kodaira dimension of $X$ is equal to 
$1$, the Kodaira-Iitaka fibration provides an $\Aut(X)$-equivariant fibration $X\to B$ from $X$ to a curve $B$. 
The divisor class of the generic fiber of this fibration  is an isotropic vector in $\NS(X)$. This vector is $\Aut(X)$-invariant
and, consequently, all elements $f$ in $\Bir(X)$ are elliptic or parabolic isometries of $\NS(X)$ (see \S~\ref{par:Isom-Hinfini-TLength}). 
This implies that $\lambda(f)=1$ for all $f$ in $\Bir(X)$. 

\subsubsection{Vanishing Kodaira dimension} Let us now assume that ($X$ is minimal and) the Kodaira dimension of $X$ is equal to $0$. 
According to the classification of surfaces, $X$ is either
\begin{itemize}
\item[(i)]  an abelian surface;
\item[(ii)] a hyperelliptic surface, obtained as a quotient of an abelian surface
by a fixed point free group of automorphisms;
\item[(iii)] a K3 surface;
\item[(iv)] or an Enriques surface. 
\end{itemize}
Hyperelliptic surfaces don't have automorphisms with $\lambda(f)>1$, as shown in \cite{Cantat:cras}.
In cases (i), (iii), and (iv), $X$ has Picard number bounded from above by $4$, $22$ and $10$ respectively. 
This shows that $\lambda(f)$ is a Salem number of degree at most $22$. Moreover, the Picard number is
at most $20$ if the characteristic of $\k$ vanish, so that $\lambda(f)$ is an algebraic integer of degree 
at most $20$ in this case.

\begin{pro}
In characteristic $2$, there are examples of pairs $(X,f)$ where $X$ is a K3 surface, $f\colon X\to X$ is an automorphism, 
and $\lambda(f)$ is a Salem number of degree $22$.
\end{pro}

To construct such an example, we make use of one of the main results of~\cite{Dolgachev-Kondo:2003}:{\sl{
Let $\k$ be an algebraically closed field of characteristic $2$. There exists a K3 surface $X$, 
defined over $\k$, such that 
\begin{itemize}
\item[(i)] the Picard number of $X$ is equal to $22$;
\item[(ii)] the automorphism group of $X$ is infinite, and does not preserve any proper subspace of $\NS_\R(X)$.
\end{itemize}
}}

Let ${\sf{O}}_{\R}(\NS_\R(X))$ be the Lie group of orthogonal endomorphisms of the N\'eron-Severi space with respect to the intersection form. 
This group is an algebraic group, and we denote by ${\sf{O}}^0_{\R}(\NS_\R(X))$ its irreducible component that contains the identity. 
From the second property, we deduce that the image $\Aut(X)^\sharp$ of $\Aut(X)$ in $\GL_{\R}(\NS_\R(X))$ intersects
${\sf{O}}_{\R}^0(\NS_\R(X))$ on a Zariski dense subgroup; indeed, if $G\subset {\sf{O}}^0_{\R}(\NS_\R(X))$ is not Zariski
dense, then $G$ preserves a non-trivial, strict subspace of $\NS_\R(X)$ (see \cite{Benoist-delaHarpe} for instance). 

As $\Aut(X)^\sharp$ is Zariski dense, we can now prove that the characteristic polynomial of a ``general'' element of $\Aut(X)^\sharp$ is irreducible (over $\Z$), its degree is equal to $22$, and its larger root is a Salem number. The proof relies on the following remark: If $g^*$ is an element of 
$\Aut(X)^\sharp$, then $g^*$ preserves the integral structure of $\NS(X)$, and preserves the intersection form, the signature of which is equal to $(1,21)$; hence,
\begin{itemize}
\item if $g^*$ has no eigenvalue of modulus $>1$, the roots of $\chi_{g^*}$ are algebraic integers of modulus at most $1$ and, by Kronecker Lemma, are roots of $1$; thus $\chi_{g^*}$ splits as a product of cyclotomic polynomials;
\item if $g^*$ has an eigenvalue of modulus $>1$, it is unique and is either quadratic or a Salem number; hence, if $\chi_{g^*}(t)$ splits as a product of two non-constant polynomials $q(t)$ and $r(t)$ in $\Z[t]$, all the roots of $r$ or $q$ have modulus $1$ and $\chi_{g^*}$ is divisible by a cyclotomic polynomial. 
\end{itemize}
Thus, either there are elements $g$ with the required properties, or $\chi_{g^*}$ is divisible by a cyclotomic polynomial of degree at most
$22$ for every $g$ in $\Aut(X)$. Since their degree is bounded by $22$, there are only finitely many cyclotomic polynomials to consider. Let $V_{22}\subset \R[t]$ be the set of all monic polynomials of degree $22$. Given  a cyclotomic polynomial $r(t)$, the subset 
\[
V_{22}(r)=\{\chi(t)\; \vert \;  \, r(t) \;  {\text{divides}}\; \chi(t) \}
\]
is a proper algebraic subset of positive codimension; moreover, the image of ${\sf{O}}^0_{\R}(\NS_\R(X))$ in $V_{22}$ by the characteristic polynomial mapping is not contained in this set, because there are elements of ${\sf{O}}^0_{\R}(\NS_\R(X))$ without any eigenvalue being a root of unity (here we use that $22$ is even). Since $\Aut(X)^\sharp$ is Zariski dense in ${\sf{O}}^0_{\R}(\NS_\R(X))$, we conclude that there are elements $f^*$ of $\Aut(X)^\sharp$ such that $\chi_{f^*}(t)$ is not contained in any $V_{22}(r)$; the characteristic polynomial of such an element is irreducible (over $\Z$), its degree is equal to $22$, and its larger root is a Salem number. 

\begin{rem}
This argument has now been extended to other examples of K3 surfaces in positive characteristic by Esnault, Oguiso, and Yu (see \cite{Esnault-Oguiso-Yu}). 
\end{rem}

\subsection{Discrete spectrum} To conclude the proof of Theorem~B, we need to show  that the  
union of all dynamical spectra $\Lambda(X)$ where $X$ runs over the set
of non-rational projective surfaces defined over $\k$, and $\k$ runs over the 
set of all fields, is a discrete subset of the real line. This follows from the upper bound $22$ for the degrees
of Salem numbers in $\Lambda(X)$ and the following lemma.

\begin{lem}
Let $B$ be a positive number and $\Sal_B$ be the set of Salem numbers of degree at most $B$.
Then $\Sal_B$ is a closed discrete subset of the real line. 
\end{lem}

\begin{proof}
Let $\lambda$ be such a Salem number, contained in the interval $[a^{-1},a]$, with $a>1$. 
Its minimum polynomial $\chi(t)\in \Z[t]$ 
has integer coefficients, and all of them are symmetric 
polynomials in  $\lambda$, $\frac{1}{\lambda}$
and its conjugates of modulus $1$. Since all these numbers have modulus
at most $a$, all coefficients of $\chi$ are bounded by $C_B a^B$, where the constant
$C_B$ depends only on $B$. Since the coefficients of $\chi$ are integers,
there is a finite list of possible coefficients, a finite list of possible minimum 
polynomials $\chi$, and therefore a finite list of Salem numbers $\lambda\in [a^{-1},a]$
of degree $\leq B$. 

Thus, the intersection of $\Sal_B$ with any compact interval $[a^{-1},a]\subset \R^*_+$ is
finite, and $\Sal_B$ is discrete.
\end{proof}

\section{Blow-ups, bubbles, isometries}\label{par:BUBLES}

When $X$ is a projective surface, the  group $\Bir(X)$  acts faithfully by isometries on 
a hyperbolic space $\H_X$, the dimension of which is infinite when $X$ is ruled or rational. 
This construction is described in \cite{Cantat:Annals} and \cite{Cantat-Lamy:Acta} ; 
in this section, we summarize the main facts and apply them to control centralizers and conjugacy 
classes in $\Bir(X)$.
The reader may consult \cite{Burger-Iozzi-Monod:2005}, \cite{Cantat:Annals}, \cite{Cantat-Lamy:Acta}, and \cite{Favre:Bourbaki}
for the results which are summarized in Sections~\ref{SubSec:Bubble} and~\ref{par:Isom-Hinfini}.

\subsection{Bubbles and Picard-Manin space}\label{SubSec:Bubble}

Let $X$ be a projective surface, defined over an algebraically closed field $\k$. If
$\pi\colon Y\to X$ is a birational morphism, one obtains an embedding of 
N\'eron-Severi groups $\pi^*\colon \NS(X)\to \NS(Y)$. 
Given two birational 
morphisms $\pi_1\colon Y_1\to X$ and $\pi_2\colon Y_2\to X$, one says that $\pi_2$
is above $\pi_1$ (or covers $\pi_1$) if $\pi_1^{-1}\circ \pi_2$ is regular. 
Starting with two birational morphisms $\pi_1\colon Y_1\to X$ and $\pi_2\colon Y_2\to X$, one can always
find a third birational morphism $\pi_3\colon Y_3\to X$ which covers $\pi_1$
and $\pi_2$. It follows easily that the inductive limit of all groups
$\NS(Y_i)$, for all surfaces $Y_i$ above $X$, is well defined. This limit is the {\bf{Picard-Manin space}} $\z_X$
of $X$; the intersection form determines a scalar product on $\z_X$, which we denote
by $(v,w)\mapsto v\cdot w$.

The bubble space $\Bub(X)$ of $X$ is defined as follows. Consider all surfaces $Y$ above $X$, i.e.~all 
birational morphisms $\pi\colon Y\to X$. Given $p_1$ on $Y_1$ and $p_2$ on $Y_2$, identify
$p_1$ with $p_2$ if $\pi_1^{-1}\circ \pi_2$ is a local isomorphism in a neighborhood of $p_2$ and maps 
$p_2$ onto $p_1$. The bubble space $\Bub(X)$ is the union of all points of all surfaces above $X$ modulo
the equivalence relation generated by these identifications. 
If $p$ is a point of the bubble 
space, represented by  a point $p$ on a surface $Y\to X$, one denotes by $E(p)$ 
the exceptional divisor of the blow-up of $p$, and by $e(p)$ its divisor class, viewed as a point
in $\z_X$. These classes satisfy $e(p)\cdot e(p')=0$ if $p\neq p'$
and $e(p)\cdot e(p)=-1$. 

\begin{rem}[see \cite{Vaquie:2000}]
A point $p$ of $\Bub(X)$ can also be seen as a divisorial valuation whose center in $X$ is a closed point; the value of the valuation 
on a rational function is the order of vanishing (or pole) of this function along $E(p)$. 
\end{rem}

The N\'eron-Severi group $\NS(X)$ is naturally  embedded as a subgroup of the Picard-Manin space. This finite dimensional lattice
is orthogonal to  $e(p)$ for all $p$ in $\Bub(X)$, and the Picard-Manin space
coincides with the direct sum 
\[
\z_X=\NS(X)\oplus \bigoplus_p \Z e(p) 
\]
where $p$ runs over the bubble space. 
Thus, each element $v$ of the Picard-Manin space
can be written as a finite sum 
\[
v= v_X + \sum_{p\in \Bub(X)} a(p) e(p).
\]
The canonical form on $\z_X$ is a linear form $\kf\colon \z_X\to \Z$, 
which is defined by 
\[
\kf(v)=k_X\cdot v_X - \sum_p a(p),
\]
where $k_X$ is the canonical divisor of $X$.

There is a completion process, for which the completion $\zz_X$ of $\z_X\otimes_\Z \R$ 
is represented by square integrable sums:  
\[
\zz_X=\{ w + \sum_p a(p) e(p) \; \vert \; w\in \NS_\R(X), \; {\text{and}} \; \sum_p a(p)^2<\infty\}
\]
The intersection form extends as a scalar product with signature $(1,\infty)$ on this space, but the canonical
form $\kf$ doesn't. 

The {\bf{hyperbolic space}} $\H_X$ of $X$ is then defined by 
\[
\H_X=\{w\in \zz_X \; \vert \;  w\cdot w =1, \; {\text{and}} \; w\cdot a>0\;  {\text{for all ample classes}}\; a \in \NS(X)\}.
\]
This space $\H_X$ is an infinite dimensional analogue of the classical hyperbolic spaces
$\H_n$: The distance $\dist$ on $\H_X$  is defined by 
\[
\cosh(\dist(v,w))=v\cdot w
\]
for all pairs of elements of $\H_X$; it is complete. If $\H_X$ is cut with a  subspace of $\zz_X$ of dimension $n$, and 
the intersection is not empty, the result is a totally geodesic hyperbolic space of dimension $n-1$. In particular,
geodesics are intersections of $\H_X$ with planes. The projection of $\H_X$ in the projective space 
$\P(\zz_X)$ is one to one, and the boundary of its image is the projection of the cone of isotropic vectors
of $\zz_X$. Thus, we denote by $\partial \H_X$ the set
\[
\partial \H_X=\{\R_+ v \in \zz_X \; \vert \; v\cdot v = 0 , \; {\text{and}} \;  v\cdot a>0\; {\text{for all ample classes}}\; a \in \NS(X) \}.
\]

\begin{rem}\label{rem:gromov-hyp}$\;$

\noindent{\bf{1.--}} The hyperbolic space $\H_X$ is $\log(3)$-hyperbolic, in the sense of Gromov (see \cite{CDP:Book}).

\noindent{\bf{2.--}} Since $\H_X$ is Gromov hyperbolic, one can approximate configurations of points in $\H_X$ by 
configuration of points in metric trees (see \cite{Cantat-Lamy:Acta} for instance).

\noindent{\bf{3.--}} The set $\partial \H_X$ coincides with the Gromov boundary of the hyperbolic space $\H_X$ (note that
$\H_X$ is not locally compact). 
\end{rem}

\subsection{Isometries and dynamical degrees}\label{par:Isom-Hinfini}

The important fact is that $\Bir(X)$ acts faithfully on $\zz_X$ by continuous linear endomorphisms, preserving the intersection form, 
the effective cone, and the nef cone; it also preserves the subset $\z_X$ and canonical form $\kf\colon  \z_X\to \Z$. 
In particular, it preserves the hyperbolic space $\H_X$. 

\begin{rem}
Intuitively,  elements of $\Bir(X)$ behave like automorphisms
on $\zz_X$, because all points have been blown-up to define $\zz_X$, so that all indeterminacy points
have been resolved.
When the Kodaira dimension of $X$ is non-negative and $X$ is minimal, then $\Bir(X)$ coincides with $\Aut(X)$. 
The space $\H_X$ can be replaced by the subset of $\NS(X,\R)$ of
elements $v$ with $v\cdot v=1$ and $v\cdot a>0$ for all ample classes $a$; the action of $\Bir(X)=\Aut(X)$ is not
always faithful but the kernel coincides with the connected component $\Aut^0(X)$ up to finite index (see \cite{Lieberman:1978, Cantat:survey-auto}). 
\end{rem}

Let $f$ be an element of $\Bir(X)$. Denote by $f_\p$ its action on $\zz_X$:
\[
f_\p\colon \zz_X\to \zz_X
\]
is a linear isometry with respect to the intersection form. We also denote by 
$f_\p$ the isometry of $\H_X$ that is induced by this endomorphism of $\zz_X$. 

\subsubsection{Translation length and types of isometries}\label{par:Isom-Hinfini-TLength}
The translation length of an isometry $g$ of $\H_X$ is defined, as for all hyperbolic
spaces, by 
\[
L(g)=\inf \{\dist (v, g(v))\; \vert \; v \in \H_X\}.
\]
If this infimum is a minimum, 
either it is equal to $0$ and $g$ has a fixed point in $\H_X$, in which case $g$ is {\bf{elliptic}},
or it is positive and $g$ is {\bf{loxodromic}} (also called hyperbolic). If $g$ is loxodromic, the set of points
$x\in \H_X$ such that $\dist(x,g(x))$ is equal to the translation length of $g$ is a geodesic line $\ax(g)\subset \H_X$; its
boundary points  are represented by isotropic vectors $a(g)$ and $b(g)$ in $\zz_X$ 
such that 
\[
g(a(g))= e^{L(g)} a(g) \; {\text{ and }} \; g(b(g))=e^{-L(g)}b(g).
\] 
The axis of $g$ is the intersection of $\H_X$ with the plane containing $a(g)$ and $b(g)$.
Normalize the choice of $a(g)$ and $b(g)$ in such a way that $a(g)\cdot b(g)=1$.
Let $x$ be a point of $\H_X$, or a point of the isotropic cone of $\zz_X$ that 
intersects $a(g)$ and $b(g)$ positively; then, the orbit $g^n(x)$ 
converges to the boundary point $\R_+ a(g)$ when $n$ goes to $+\infty$, and 
to $\R_+ b(g)$  when $n$ goes to $-\infty$. More precisely, in $\zz_X$ we have
\[
\frac{1}{e^{n L(g)}} g^n(x)\to (x\cdot b(g))\,  a(g), \;\, {\text{and}} \; \, \frac{1}{e^{n L(g)}} g^{-n} (x) \to (x\cdot a(g))\,  b(g)
\]
as $n$ goes to $+\infty$.

When the infimum is not realized, $L(g)$ is equal to $0$, and $g$ is {\bf{parabolic}}: It fixes
a unique line in the isotropic cone of $\zz_X$; this line is fixed pointwise, and 
all orbits $g^n(x)$ in $\H_X$ accumulate to the corresponding boundary point when
$n$ goes to $\pm \infty$ (see \cite{Burger-Iozzi-Monod:2005} for examples of accumulation without convergence). 

\subsubsection{Types of birational transformations (see \cite{Cantat:Annals})}\label{par:GDF}
This classification of isometries into three types hold for all isometries of $\H_X$. For
isometries $f_\p$ induced by birational transformations of $X$, there is a dictionary between
this classification and the geometric properties of $f$. To state it, let us introduce
the following definitions: A birational transformation $f$ of a projective surface $X$ is
\begin{itemize}
\item[(i)] {\bf{virtually isotopic to the identity}} if there is a positive iterate $f^n$ of $f$ and 
a birational mapping $\phi\colon Z\da X$ such that $\phi^{-1}\circ f^n \circ \phi$ is an 
element of $\Aut(Z)^0$;
\item[(ii)] a {\bf{Halphen twist}} if $f$ preserves a one parameter family of genus one curves on $X$ but
$f$ is not virtually isotopic to the identity;
\item[(iii)] a {\bf{Jonqui\`eres twist}} if $f$ preserves a one parameter family of rational curves on $X$ but
$f$ is not virtually isotopic to the identity.
\end{itemize}

When $f$ is a Halphen (resp. a Jonqui\`eres twist) then, after  conjugacy by a birational 
mapping $\phi\colon Z\da X$, $f$ permutes the fibers of a genus $1$ (resp. a rational)
fibration $\pi\colon Z\to B$. Let $z$ be the divisor class of the generic fiber of this
fibration. Then $z$ is an isotropic vector in $\zz_X$ that is fixed by $f_\p$; in particular, 
$f_\p$ can not be loxodromic.

\begin{rem}
Let $f\colon X\dasharrow X$ be a Halphen twist, and let $\phi \colon Z\to X$ be a modification 
of $X$ on which the $f$-invariant family of genus one curves form a fibration $\pi\colon Z\to B$.
 Let $Z'$ be a relative minimal model of this genus one fibration ($Z'$ is obtained from $Z$ by blowing down 
 exceptional divisors of the first kind that are contained in fibers of $\pi$, and iteration of this process). Then, $f$ 
becomes an automorphism of $Z'$. On the other hand, Jonqui\`eres twists are not conjugate to
automorphisms of projective surfaces (see \cite{Blanc-Deserti}).
\end{rem}

\begin{thm}[Gizatullin, Cantat, Diller-Favre, see \cite{Diller-Favre:2001, Cantat:Annals}]\label{thm:Giz-class}
Let $\k$ be an algebraically closed field. Let $X$ be a projective surface defined over $\k$.
Let $f$ be birational transformation of $X$, let $f_\p$ be the isometry of $\H_X$ determined by $f$, 
and let $x$ be a point of $\H_X$. 
\begin{enumerate}
\item $f_\p$ is elliptic if and only if $f$ is virtually isotopic to the identity.
\item If $f_\p$ is parabolic, either $x\cdot f_\p^n(x)$ grows linearly with $n$, and $f$ is
a Jonqui\`eres twist, or $x\cdot f_\p^n(x)$ grows quadratically with $n$, and $f$ is
a Halphen twist.
\item $f_\p$ is loxodromic if and only if the dynamical degree $\lambda(f)$ is $>1$. 
\end{enumerate}
In all cases, the translation length $L(f_\p)$ is equal to the logarithm of $\lambda(f)$.
\end{thm}

\begin{eg} Let $f$ be a birational transformation of the plane $\P^2_\k$. 
Let $e_0$ be the class of a line, viewed as a point in $\H_{\P^2_\k}$. Then 
\[
f_\p(e_0)= \deg(f) e_0 -\sum a(p) e(p)
\]
where $\deg(f)$ is the degree of $f$ and $a(p)$ is the multiplicity of the homaloidal
net $f_*\vert {\mathcal{O}}(1)\vert $ at the point $p$ ($p$ may be "infinitely near"). 
Since $e_0$ does not intersect any of the $e(p)$, one gets 
\[
\cosh(\dist(e_0, f_\p(e_0)))=e_0\cdot f_\p(e_0)=\deg(f).
\]
This  establishes the link between $\deg(f^n)$ and $\dist(e_0, f^n_p(e_0))$ which leads
to the equality $L(f)=\log(\lambda(f))$ (see \S~\ref{par:Axis-Degree} for details and complements).
\end{eg}

\begin{eg} 
An element $f$ of the Cremona group is virtually isotopic to the identity if and only if $f$ has finite
order or $f$ is conjugate to an element of $\Aut(\P^2_\k)=\PGL_3(\k)$ (see \cite{Blanc-Deserti}).
\end{eg}

\subsection{Centralizers}\label{par:centralizer}

\begin{cor}   Let $f$ be a birational transformation of a projective surface $X$.  If $f$ is loxodromic, 
the infinite cyclic group generated by $f$ is a finite index subgroup of the centralizer of $f$ in $\Bir(X)$. 
\end{cor}

\begin{proof}
Let $f$ be a loxodromic birational transformation of the surface $X$. 
Then $f$ acts on the hyperbolic space $\H_X$ as a hyperbolic isometry, with an invariant
axis $\ax(f)$. The endpoints of $\ax(f)$ correspond to two eigenvectors $b(f)$
and $a(f)$ in the isotropic cone of $\zz_X$, with 
\[
f_\p ( b(f)) =\frac{1}{\lambda(f)}b(f), \;\, \text{ and } \; f_\p (a(f))=\lambda(f)a(f).
\]
These vectors are unique up to scalar multiplication. 

Let $\Cent(f)$ denote the centralizer of $f$ in the group $\Bir(X)$. It preserves the eigenlines
$\R b(f)$ and $\R a(f)$, acting on each of them by scalar multiplication. This
provides a morphism 
$\theta\colon \Cent(f)\to \R^*_+$
such that 
\[
g_\p(a(f))=\theta(g)a(f)
\]
for all $g$ in $\Cent(f)$. Moreover, $\theta(g)$ or its inverse coincides with the dynamical degree
of $g$ because, if $g$ is loxodromic, then $g_\p$ has exactly two fixed points on the boundary of $\H_X$. 

The image of $\theta$ is a subgroup of $\R^*_+$ which is contained in $\Lambda(X)\cup \{1\} \cup \Lambda(X)^{-1}$. 
From the spectral gap property, this image does not intersect the interval $]1,\lambda_L[$, and is consequently a
discrete subgroup of $\R^*_+$. Since all infinite discrete subgroups or $\R^*_+$ are cyclic, the image
$\theta(\Cent(f))$ is cyclic. 

Let $\Cent(f)^0$ be the kernel of $\theta$. All we need to prove is that $\Cent(f)^0$ 
is finite because, then, the exact sequence
\[
1\to \Cent(f)^0 \to \Cent(f) \stackrel{\theta}{\to}  \Z\to 0
\]
proves that $\Cent(f)$ is finite by cyclic. 

The group $\Cent(f)^0$ preserves $\ax(f)$ and fixes $a(f)$. It must
therefore fix $\ax(f)$ pointwise. Let $q$ be a point of $\ax(f)$ and let $\Delta$ be its distance to $e_0$ 
in $\H_X$ (where, as above, $e_0$ is the class of a multiple of an ample divisor $D$ on $X$ with $e_0^2=1$). 
Let $h$ be an element of $\Cent(f)^0$. Then $\dist(e_0,h_\p(e_0))\leq 2 \dist(e_0,q)=2\Delta$; hence
\[
h_\p(e_0)\cdot e_0\leq \cosh(2\Delta).
\]
This shows that the degree of $h\in \Cent(f)^0\subset \Bir(X)$ with respect to the polarization $D$ 
is uniformly bounded by some explicit constant $M=\cosh(2\Delta)$. 

Assume that the Kodaira dimension $\kod(X)$ is non-negative. Changing $X$ into its unique minimal 
model, we assume that $X$ is minimal; this implies $\Bir(X)=\Aut(X)$ because $\kod(X)\geq 0$. 
Thus, $\Aut(X)$ contains a loxodromic element (determined by $f$) and $X$ is either an abelian surface, a K3 surface, or
an Enriques surface. The group $\Cent(f)^0$ is, now, a group of automorphisms of $X$ with bounded
degree with respect to a fixed polarization $D$ on $X$. This implies that the intersection 
of $\Aut(X)^0$ with $\Cent(f)^0$ is a finite index subgroup of $\Cent(f)^0$. Moreover, the centralizer of 
$f$ in $\Aut(X)^0$ is an algebraic group. Thus, either $\Cent(f)^0$
is finite, or it contains a connected algebraic subgroup $G\subset \Aut(X)^0$ of positive dimension. In the latter case, $X$ 
is an abelian variety and $G$ acts by translations on $X$, because $\Aut(X)$ is discrete for K3 and Enriques surfaces. 
Let $G_1$ be a closed, one dimensional subgroup of $G$: Its orbits form a fibration of $X$ by elliptic curves.
Since $f$ commutes to $G_1$, it preserves this fibration of $X$. This contradicts the fact that $f$ is 
loxodromic, and proves that $\Cent(f)^0$ is finite.

Assume that $\kod(X)$ is negative. Since $\Bir(X)$ contains a loxodromic element $f$, the surface $X$ is rational, and
we can suppose that $X$ is the projective plane and $e_0$ is the class of a line in $\P^2_\k$.
The group $\Cent(f)^0$ is a subgroup of $\Bir(\P^2_\k)$ of bounded degree. From Corollaries 2.8 and 2.18 of \cite{Blanc-Furter}, we
deduce that its Zariski closure in $\Bir(\P^2_\k)$ is an algebraic subgroup of $\Bir(\P^2_\k)$. Denote by $G$ the connected
component of the identity in this group. If $\Cent(f)^0$ is infinite, the dimension of $G$ is positive, and a result of 
Enriques shows that $G$ is contained, after conjugation, in the group of automorphisms of a minimal, rational surface (see \cite{Blanc:2009, Enriques}).
As a consequence, $G$ contains a Zariski closed abelian subgroup $A$ whose orbits have dimension $1$ in $X$. Those orbits are organized
in a pencil of curves that is invariant under the action of $f$. This contradicts $\lambda(f)>1$ and shows that $\Cent(f)^0$ is finite. 
\end{proof}

\subsection{Conjugacy between loxodromic transformations}\label{par:conjugacy-degree-bound}

Assertion (1) of Corollary~\ref{cor:thm-A-cor1} can be rephrased as follows: {\sl{For all loxodromic elements $f$ in $\Bir(\P^2_\k)$ and all points $x$ in 
$\H_{\P^2_\k}$, 
\[
\dist(x,f_\p (x)) \geq \log(\lambda_L) 
\]
where $\lambda_L$ is the Lehmer number.}} 

\begin{lem}\label{lem:28}
For all loxodromic elements $f$ in $\Bir(\P^2_\k)$ and all points $x$ in 
$\H_{\P^2_\k}$, 
\[
\dist(x,\ax(f_\p)) \leq 28 \cdot \dist(x,f_\p(x)).
\]
\end{lem}

\begin{rem}
The constant $28$ is the smallest integer $m$ such that $m\log(\lambda_L) \geq 4\log(3)$, 
and the occurrence of $\log(3)$ comes from the fact that $(\H_{\P^2_\k}, \dist)$ is $\log(3)$-hyperbolic 
in the sense of Gromov (see Remark~\ref{rem:gromov-hyp}). 
\end{rem}

\begin{proof}
Let $y$ be the projection of the point $x$ on the axis of $f_\p$. Let $n$ be the least positive integer which satisfies 
\[
\dist(y,f^n_\p(y))\geq 8\log(3).
\]
Consider the geodesic quadrilateral with vertices $x$, $y$, $f^n_\p(y)$, and $f^n_\p(x)$. By hyperbolicity, 
the geodesic segment $[x,f^n_\p(x)]$ is contained in the $(2\log(3))$-neighborhood of the other three, and its length is at least $8\log(3)$; hence, its middle point $m$ is at most $(2\log(3))$-away from $[y,f^n_\p(y)]$. Let $m'$ be the projection of $m$ on the segment  $[y,f^n_\p(y)]$. Then the distance from $x$ to $m'$ is equal to the sum of the distances from $x$ to $y$ and from 
$y$ to $m'$, up to an error of $2\log(3)$. The same estimate in the triangle $(m',f^n_\p(y), f^n_\p(x))$ provides the inequality :
\[
\dist(x,f^n_\p(x))\geq \dist(x,y)+ \dist(y, f^n_\p(y))+\dist(f^n_\p(y),f^n_\p(x))-8\log(3).
\]
Since $f_\p$ is an isometry and $y$ is the projection of $x$ on its axis, the choice for $n$ implies
\[
n\cdot \dist(x,f_\p(x))\geq 2\dist(x,\ax(f_\p)).
\]
On the other hand, Corollary~\ref{cor:thm-A-cor1} shows that $n$ can be chosen to be the smallest integer above $8\log(3)/\log(\lambda_L)\simeq 54.13$, that is $n=55$.
\end{proof}

\begin{thm}\label{thm:conj}
Let $f$ and $g$ be two loxodromic elements of $\Bir(\P^2_\k)$. If $f$ is conjugate to $g$,  one can find an element $h$ of 
$\Bir(\P^2_\k)$ such that $f=hgh^{-1}$ and 
\[
\deg(h)\leq 2^{57}(\deg(f)\deg(g))^{29}
\]
\end{thm}

The following argument provides also a proof of Corollary~\ref{coro:conjug-degree-bound} (with $d\geq \deg(f)$, $\deg(g)$).

\begin{proof}
Let $x$ be the projection of $e_0$ on the axis of $f$, and $y$ be its projection on the axis of $g$. Let $h_0$ be an element of $\Bir(\P^2_\k)$ that conjugates $f$ to $g$; it maps $y$ onto a point $z_0:=(h_0)_\p(y)$ of $\ax(f_\p)$.  Let $k$ be an integer such that 
$\dist(f_p^k(z_0),x)\leq \log(\lambda(f))$. Such a $k$ exists because  $f_\p$ acts by translation of length $\log(\lambda(f))$ on its axis. Changing 
$h_0$ into $h=f^k\circ h_0$, we obtain a new conjugacy from $g$ to $f$ that maps $y$ onto a point $z=h(y)$ at distance at most
$\log(\lambda(f))$ from $x$.  Now, 
\[
\dist(e_0,h_\p(e_0))\leq \dist(e_0,x)+\dist(x,h_\p(y))+ \dist(h_\p(y),h_\p(e_0)).
\]
Since $h_\p$ is an isometry we get 
\[
\dist(e_0,h_\p(e_0))\leq \dist(e_0,\ax(f_\p))+\log(\lambda(f))+ \dist(e_0, \ax(g_\p)).
\]
The previous lemma can then be applied to $e_0$, and gives 
\[
\dist(e_0,h_\p(e_0))\leq \log(\lambda(f))+28\cdot(\dist(e_0,f_\p(e_0))+\dist(e_0, g_\p(e_0))).
\]
The result follows from $\log(\lambda(f_\p))\leq \dist(e_0,f_\p(e_0))$, $\cosh (\dist(e_0,f_\p(e_0)))= \deg(f)$
and easy estimates for the reciprocal function of $\cosh(\cdot )$. 
\end{proof}
Let us add the following complement to Theorem~\ref{thm:conj}, which  shows that the hypothesis on $f$ and $g$ 
(being loxodromic) in the theorem can be checked in a finite number of steps.
\begin{cor}
An element $f\in \Bir(\P^2_\k)$ is loxodromic if and only if $\deg(f^{400})\geq 3^{19} \deg(f^{200})$. 
\end{cor}

\begin{proof}
Let us write $g=f^{200}$, which is loxodromic if and only if $f$ is.

One direction has been proved by  Junyi Xie in \cite{Junyi-Xie:preprint}: If $\deg(g^2)\geq 3^{19} \deg(g)$, then $\dist(e_0,g_\p^2(e_0)) > \dist(e_0,g_\p(e_0)) + 18 \log(3)$, and this implies that $g_\p$ is a loxodromic isometry because $\H_X$ is $\log(3)$-hyperbolic.

In the other direction, if $f_\p$ is loxodromic, the translation length of $f_\p$ is at least $\log(\lambda_L)$. Hence, the translation length of $g$ is larger than $ 200\log(\lambda_L)$. Then, as in the proof of Lemma~\ref{lem:28}, 
\[
\dist(e_0,g_\p^2(e_0)) \geq \dist(e_0,g_\p(e_0))+ L(g)  - 8\log(3),
\]
and this implies that $\deg(g^2)\geq 3^{19} \deg(g)$ because $\lambda_L^{200}> 3^{8+19}$.
\end{proof}

\begin{eg} Let $m$ be a positive integer and $a$ be a non-zero element of the field $\k$. 
Let $(x,y)$ be affine coordinates of the plane. Consider the  transformation  
$f_a\colon (x,y)\mapsto (ax,y)$. This automorphism is conjugate to 
$g_{a,m}\colon (x,y)\mapsto (ax,a^my)$
by the monomial transformation $h(x,y)=(x, x^m \cdot y)$. The degree of $f_a$ and of $g_{a,m}$ is equal to $1$, but the degree of $h$ is 
$m+1$. If $a^\Z$ is a Zariski dense subgroup of $\k^*$, one easily shows that there is no conjugacy $h'$ of degree $<m+1$.
Thus, the degree of the conjugacy is not bounded by the degree of $f_a$ and $g_{a,m}$ if one does not assume $\lambda(f)>1$.\end{eg}

\section{The Weyl group $W_\infty$}\label{par:WEYL}

We now define, and study, a group of linear transformations of $\z_{\P^2_\k}$, with integer coefficients, preserving the 
intersection form. This group of isometries $W_\infty$ is a subgroup of $\Isom(\z_{\P^2_\k})$ that contains the image
of $\Bir(\P^2_\k)$. 

\subsection{Definition of $W_\infty$}

In what follows, $e_0\in \z_{\P^2_\k}$ is the class of a line. 
Let $p_1$, $p_2$, and $p_3$ be the three points of the plane defined by 
\[
p_1=[1:0:0], \quad p_2=[0:1:0], \quad p_3=[0:0:1].
\]

The ``infinite Weyl group'' $W_\infty$   is the group of $\Z$-linear automorphisms of $\z_{\P^2_\k}$ generated by:
\begin{enumerate}
\item the group $\Sym({\B(\P^2_\k)})$ of permutations of the set $\B(\P^2_\k)$, that acts on 
$\z_{\P^2_\k}$ by sending $e_0$ to itself and permuting the $e(p)$.
\item the involution $\sigma_0$ that sends $e_0$ onto $2e_0-e({p_1})-e({p_2})-e({p_3})$, sends $e({p_i})$ onto $e_0-e({p_1})-e({p_2})-e({p_3})+e({p_i})$ for $i=1,2,3$ and fixes $e(p)$ for all $p$ in $\B(\P^2_\k)\backslash \{p_1,p_2,p_3\}$.
\end{enumerate}

Let $p$ and $q$ be two elements of $\B(\P^2_\k)$. The element $e(p)-e(q)$ of $\z_{\P^2_\k}$ has self-intersection $-2$; as a consequence, the linear 
transformation 
\[
\tau_{p,q}\colon x\mapsto x + (x\cdot (e(p)-e(q))) (e(p)-e(q))
\]
is the orthogonal reflection that maps $e(p)-e(q)$ to its opposite. The group generated by all these reflections is the subgroup of elements of $\Sym(\B(\P^2_\k))$ with finite support. Similarly, $\sigma_0$ 
corresponds to the orthogonal reflection associated to $e_0-e({p_1})-e({p_2})-e({p_3})$. 
This explains why $W_\infty$ is considered as an infinite Weyl group (or Coxeter group).

By construction, $W_\infty$ preserves the intersection form and the canonical form $\kf$, and extends as a group of isometries of $\H_{\P^2_\k}$.

\begin{lem}\label{lem51} Let $\k$ be an algebraically closed field.
If  $f$ is an element of $\Bir(\P^2_\k)$, the linear transformation $f_\p\colon \z_{\P^2_\k}\to \z_{\P^2_\k}$ is an element of $W_\infty$.
\end{lem}

\begin{proof}
If $f$ has degree $1$, it is an element of the group $\Aut(\P^2_\k)$, acts by permutation on $\Bub(\P^2_\k)$ and fixes the class $e_0$.
In other words, the map $f\mapsto f_\p$ provides an embedding of $\Aut(\P^2_\k)$ into $ \Sym({\B(\P^2_\k)})\subset W_\infty$. 
If $f$ is the standard quadratic transformation 
\[
 [x:y:z]\dasharrow [yz:xz:xy]
\] it has three base-points, namely 
$p_1=[1:0:0]$, $p_2=[0:1:0]$, and $p_3=[0:0:1]$. 
Moreover, $f_\p$ acts as $\sigma_0$ on $e_0$ and the  $e(p_i)$ for $i=1,2,3$,  and transforms each $e(q)$, $q\in \Bub(\P^2_\k)\setminus \{e(p_1), e(p_2), e(p_3)\}$, to some $e(q')$ with $q'\in \Bub(\P^2_\k)\setminus \{e(p_1), e(p_2), e(p_3)\}$. This implies that $f_\p$ is the composition of $\sigma_0$ with an element of $\Sym({\B(\P^2_\k)})$, so that $f_\p$ is in $W_\infty$. The result follows from Noether-Castelnuovo theorem, which asserts that $\Bir(\P^2_\k)$ is generated by $\Aut(\P^2_\k)$ and the standard quadratic transformation, when $\k$ is algebraically closed (see \cite{KSC:Book} or \cite{Alberich:LNM}). 
\end{proof}

\begin{rem}$\;$

{\bf{  1.--}} The group $W_\infty$ is strictly larger than $\Cr_2(\k)$ because the elements of $\Sym({\Bub(\P^2_\k)})$
fix $e_0$ but most of them are not induced by projective linear transformations of the plane.

{\bf{2.--}} Lemma~\ref{lem51} is implicitly contained in Noether's original ``proof'' of Noether-Castelnuovo theorem.
\end{rem}

\subsection{Degrees, multiplicity, base-points}

Let $h$ be an element of $W_\infty$. We define the {\bf{degree}} $\deg(h)$ by $\deg(h)=e_0\cdot h(e_0)$; the degree is a positive integer because all elements of $W_\infty$ preserve $\H_{\P^2_\k}$.
Writing 
\[
h(e_0)=\deg(h)e_0-\sum_p a(p)e(p),
\]
where $p$ runs over $\Bub(\P^2_\k)$, we say that $p$ is a {\bf{base-point}} of $h^{-1}$ if $a(p)\neq 0$; the integer $a(p)$ is the {\bf{multiplicity}} of the base-point. Since $h$ is an isometry of $\H_{\P^2_\k}$, it is either elliptic, parabolic, or loxodromic; moreover, the {\bf{dynamical degree}}
\[
\lambda(h)=\lim_{k\to +\infty} (\deg(h^k)^{1/k})
\]
is well defined, and its logarithm is the translation length $L(h)$. 

It is not a priori clear that the multiplicities $a(p)$ are non-negative. 
For example, $\Lambda=3e_0+e({p_1})-\sum_{i=2}^7 e({p_i})$ satisfies $\Lambda^2=1$ and intersects the canonical form as $e_0$ does. To show that such an element cannot be sent onto $e_0$ by an element of $W_\infty$, we need the following lemma.

\begin{lem}\label{Lem:MultiplicitiesPositive}
Let $v\in\z_{\P^2_\k}$ be one of the following vectors:
\[
e_0, \;  \; e({q_1}),  \; \;  e_0-e({q_1}), \;  \; 3e_0-\sum_{i=1}^l e(q_i)
\]
for some distinct points $q_1,\dots,q_l\in \Bub(\P^2_\k)$. For any $h\in W_\infty$, the following holds:
\begin{enumerate}
\item
There exists $n\geq 1$ and $s_1,\dots,s_n\in \Sym(\Bub(\P^2_\k))$ satisfying
\begin{eqnarray*}
h(v)& =& s_n\sigma_0s_{n-1}\sigma_0\dots s_2\sigma_0s_1(v)\\
(s_i\sigma_0\dots\sigma_0s_1)(v)\cdot e_0 & > &(s_{i-1}\sigma_0\dots\sigma_0s_1)(v)\cdot e_0
\end{eqnarray*}
for all $i=2,\dots,n$.

\item Either $h(v)=e(q)$ for some $q\in \Bub(\P^2_\k)$, or there exists $k\geq 0$, non-negative integers  $d,a_1,\dots,a_k$, and a finite set of points $r_1,\dots,r_k\in \Bub(\P^2_\k)$, such that
\[
h(v)=de_0-\sum_{i=1}^k a_i e({r_i}).
\]
\end{enumerate}
\end{lem}

\begin{rem}\label{rem:quadW} Let us list the elements of degree $1$ or $2$ in $W_\infty$; this will be useful for the proof of Lemma~\ref{Lem:MultiplicitiesPositive}. 

\vspace{0.1cm}

$\bullet$ Let $h$ be an element of degree $1$ in $W_\infty$. This means that 
$
h(e_0)=e_0-\sum a(p) e(p)
$
where $a(p)\in \Z$ vanishes for all but a finite number of points $p\in \Bub(\P^2_\k)$.
Since the self-intersection
 is preserved under the action of $h$, the sum $\sum_p a(p)^2$ is equal to $0$; this   
 implies that $h(e_0)=e_0$. Let $p$ be an element of $\Bub(\P^2_\k)$. From 
 \[
 h(e(p))\cdot e_0 = h(e(p))\cdot h(e_0)= 0
 \]
we deduce that $h(e(p))=\sum b(q)e(q)$ with $b(q)\in \Z$, and then that $h(e(p))=\pm e(q)$ for some point $q$ in $\Bub(\P^2_\k)$, because the self-intersection of $h(e_p)$ is $-1$. Since the canonical linear form is preserved, one concludes that $h(e_p)=e_q$. 
In other words, $h$ is an element of $\Sym(\Bub(\P^2_\k))$.

\vspace{0.1cm}

$\bullet$ Say that $h\in W_\infty$ is {\bf{quadratic}} if its degree is equal to $2$. Write
$h(e_0)=2e_0-\sum a(p) e(p)$ with $a(p)\in \Z$. The invariance of the self-intersection
provides
 \[
 4-\sum_p a(p)^2 = 1\, ;
 \]
hence, there are exactly three base points, each with multiplicity $1$. Composing $h$ with an element of $\Sym(\Bub(\P^2_\k))$, 
one may assume that these three base-points coincide with the base-points $p_1$, $p_2$, and $p_3$ of $\sigma_0$. 
Then $\sigma_0 h$ has degree $1$. We conclude that quadratic elements are composition $s\sigma_0 s'$ with $s$ and $s'$ in $\Sym(\Bub(\P^2_\k))$.\end{rem}

\begin{proof} The proof of this lemma parallels classical facts from Coxeter group theory.

We first prove that $(1)$ implies $(2)$. 
The proof proceeds by induction on the minimum number $n\geq 1$ for which $h$ satisfies assertion $(1)$.
If $n=1$ then $h(v)=s_1(v)$ for some element 
$s_1\in \Sym(\Bub(\P^2_\k))$, and assertion $(2)$ follows. 
Let us now assume that $n\geq 2$ and apply the induction hypothesis to $s_{n-1}\sigma_0\dots s_2\sigma_0s_1$.
Let $w=s_{n-1}\sigma_0\dots s_2\sigma_0s_1(v)$.
If $w$ is equal to $e_q$ for some $q\in \Bub(\P^2_\k)$, we apply $s_n\sigma_0$: $h(v)=s_n\sigma_0(w)$ is either $e({q'})$ or $e_0-e({q'})-e({q''})$ for some $q',q''\in \Bub(\P^2_\k)$, and the first case is in fact impossible by $(1)$.
Otherwise, we write
$w=de_0-\sum_{i=1}^k c_ie({r_i})$
for some points $r_1,\dots,r_k\in \Bub(\P^2_\k)$ and non-negative integers $d,c_i$. Ordering the points and adding, if necessary, points with trivial coefficients $c_i=0$, we assume that $r_1$, $r_2$, and $r_3$ are the three base-points $p_1$, $p_2$ and $p_3$ of $\sigma_0$.
By definition of $\sigma_0$, we find 
\[
\sigma_0(w)=(d+d')e_0-\sum_{i=1}^3(c_i+d')e(r_i)-\sum_{i=4}^m c_ie(r_i),
\]
 where $d'=d-c_1-c_2-c_3$. Since 
 \[
 d+d'=s_n\sigma_0\dots s_2\sigma_0s_1(v)\cdot e_0>d,
 \]
  we obtain $d'>0$ and see that $c_i+d'$ is non-negative for $i=1,2,3$. This proves that $(1)\Rightarrow(2)$ by induction on $n$.

We now  prove assertion $(1)$. By definition of $W_\infty$, we can always write $h$ as a composition $s_m\sigma_0s_{m-1}\sigma_0\dots s_2\sigma_0s_1$ for some $s_1,\dots,s_m\in \Sym(\Bub(\P^2_\k))$. For $i=1,\dots,m$, write 
\[
w_i=s_i\sigma_0\dots\sigma_0s_1(v) \quad {\text{and}} \quad   d_i=w_i\cdot e_0.
\]
Our aim is to replace the sequence $(s_i)_{i=1}^m$, keeping $s_m\sigma_0s_{m-1}\sigma_0\dots s_2\sigma_0s_1(v)=h(v)$, in order to assure that the sequence $(d_i)_{i=1}^m$ increases strictly.

Let $s$ and $s'$ be elements of $\Sym({\Bub(\P^2_\k)})$. One easily checks, for all possibilities of $v$,  that   
\[
s'\sigma_0s(v)\cdot e_0> v \cdot e_0 \quad {\text{or}} \quad s'\sigma_0 s(v)=s''(v)
\]
for some $s''\in \Sym({\Bub(\P^2_\k)})$. We can then change the sequence $(s_i)_{i=1}^m$ to assure that $m=1$, in which case the result is obvious,  or $d_2>d_1$.

We set $S=\{i\in\{2,\dots,m-1\}\mid d_{i-1}< d_i\ge d_{i+1}\}$ and write $D=\max\{d_i | i\in S\}$ with the convention $D=0$ if $S=\emptyset$. We then denote by $l$ the number of elements $i\in \{2,\dots,m-1\}$ such that $d_i=D$ and prove assertion $(1)$ by induction on the pairs $(D,l)$, ordered lexicographically. 

If $D=0$ or $l=0$, then $S=\emptyset$ and we have $d_1<d_2<\dots<d_m$, which achieves the proof. 
We can thus assume that there is some $k\in S$ such that $d_k=D>0$. Let us recall that
\[
d_{k-1}<d_k\ge d_{k+1}.
\]

Since $k<m$, the induction hypothesis and the proof of $(1)\Rightarrow (2)$ yield the existence of points $p_1,\dots,p_r$ such that 
\[
w_{k}=s_{k}\sigma_0\dots s_2\sigma_0s_1(v)=d_ke_0-\sum_{i=1}^r a_ie(p_i)
\]
 for some non-negative integers $a_1,\dots,a_r$. We again assume that $p_1,p_2,p_3$ are the three base-points of $\sigma_0$. Since $w_{k+1}=s_{k+1}\sigma_0(w_k)$ and $w_{k-1}=\sigma_0(s_k)^{-1}(w_k)$, we find
\[
d_{k+1}=e_0\cdot \sigma_0(w_k)=\sigma_0(e_0)\cdot w_k=2d_k-\sum_{i=1}^3 e({p_i})\cdot w_k=2d_k-\sum_{i=1}^3 a_i,
\]
and
\[
d_{k-1}=e_0\cdot \sigma_0(s_k)^{-1}(w_k)=s_k\sigma_0(e_0)\cdot w_k=2d_k-\sum_{i=1}^3 e({q_i})\cdot w_k,
\]
where $q_i=s_k(p_i)$ for $i=1,2,3$.
This implies 
\[
\sum_{i=1}^3 e({q_i}) \cdot w_k > d_k\quad {\text{and}} \quad \sum_{i=1}^3 e({p_i}) \cdot w_k \ge d_k.
\]
Let  $t+2$ be the number of points of the set $\{p_1,p_2,p_3,q_1,q_2,q_3\}$. We can define $t$ sets of $3$ points 
$P_1,\dots,P_t\subset \Bub(\P^2_\k)$, such that $P_1=\{q_1,q_2,q_3\}$, $P_t=\{p_1,p_2,p_3\}$, and  
\begin{itemize}
\item $P_i\cap P_{i-1}$ contains $2$ points and 
\item $\sum_{x\in P_i} w_k \cdot e({x})>d_k$ for $i=1,\dots,t-1$.
\end{itemize}
For each $i=1,\dots,t$, we fix an element $s'_i\in \Sym({\Bub(\P^2_\k)})$ that sends $\{p_1,p_2,p_3\}$ onto $P_i$
; for $i=1$, we choose $s'_1$ to be $s_k$ and for $i=t$ we choose $s'_t={\rm{Id}}$. We then write $w'_i=\sigma_0 (s'_i)^{-1}(w_k)$ and $g_i=\sigma_0 (s'_{i+1})^{-1}s'_i\sigma_0$. To illustrate this, we make a picture in the case where $t=4$:

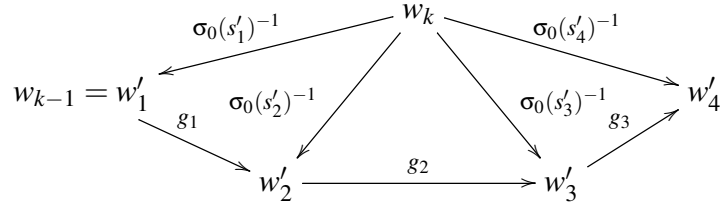
\begin{center}$\xymatrix@R=0.5cm@C=1.2cm{
&&w_k
\ar[rrd]^{\sigma_0 (s'_4)^{-1}}
\ar[lld]_{\sigma_0 (s'_1)^{-1}}
\ar[rdd]^>>>>>>{\sigma_0 (s'_3)^{-1}}
\ar[ldd]_>>>>>>{\sigma_0 (s'_2)^{-1}}
&&&\\
w_{k-1}=w'_1\ar[dr]^{g_1} && && w'_4 \\
& w'_2\ar[rr]^{g_2}&& w'_3\ar[ru]^{g_3} }$\end{center}

For $i=1,\dots,t-1$, the inequality $\sum_{x\in P_i} w_k \cdot e({x})>d_k$ yields $e_0\cdot w'_i<d_k$. Moreover, the fact that $P_i\cap P_{i-1}$ contains two points implies that $g_i$ is an element of $W_\infty$ of degree $2$; as such, it is equal to $a_i \sigma_0 b_i$ for some $a_i,b_i\in \Sym_{\Bub(\P^2_\k)}$ (see Remark~\ref{rem:quadW}).
 Because $s'_t$ preserves $P_t=\{p_1,p_2,p_3\}$,  $s_{k+1}\sigma_0s'_t \sigma_0$ is equal to some $a_t\in \Sym_{\Bub(\P^2_\k)}$. We can then write $s_{k+1}\sigma_0s_k\sigma_0\in W_\infty$, that sends $w_{k-1}$ onto $w_{k+1}$, as

$\begin{array}{rcl}
s_{k+1}\sigma_0s_k\sigma_0&=&(s_{k+1}\sigma_0s'_t \sigma_0)\sigma_0 (s'_t)^{-1}s'_1\sigma_0\\
&=&a_t(\sigma_0 (s'_{t})^{-1}s'_{t-1}\sigma_0)\dots (\sigma_0 (s'_{3})^{-1}s'_2\sigma_0)(\sigma_0 (s'_{2})^{-1}s'_1\sigma_0)\\
&=&a_t g_tg_{t-1}\dots g_1.\end{array}$

For $i=1,\dots,t-1$, $g_i\dots g_1(w_{k-1})=w'_i$, which has intersection with $e_0$ smaller than $d_k$.  The replacement above in the decomposition  (writing each $g_i$ as $a_i \sigma_0 b_i$ and rearranging the terms) either decreases $D$ or decreases $l$, without changing $D$. 
\end{proof}

\subsection{Noether inequality}
Let $h$ be an element of $W_\infty$ of degree $d$.
By Lemma~\ref{Lem:MultiplicitiesPositive}, there is a finite set of points  $q_i\in \B(\P^2_\k)$, and positive integers $a_i$, $i=1,\dots,k$, such that 
\[
h(e_0)=de_0-\sum_{i=1}^k a_i e({q_i}),
\] 
where the $a_i$  are positive integers.  Computing $h(e_0)^2=e_0^2=1$ and applying the canonical form $\kf$ to $h(e_0)$, we get the classical Noether equalities
\begin{equation}\label{eq:NoetherEquality}
\sum_{i=1}^k (a_i)^2=d^2-1,\quad \sum_{i=1}^k a_i=3d-3.
\end{equation}
This already implies that there are at least three points $q_i$ if $d\neq 1$.
\begin{lem}[Noether inequality]\label{Lem:EqualitiesNoether}
Let $h$ be an element of $W_\infty$ of degree $d\ge 2$, and let $a_1,\dots,a_k$ be the multiplicities of the base-points of $h$. 

\begin{enumerate}
\item The following equality is satisfied.
\[\begin{array}{rcl}
(d-1)(a_1+a_2+a_3-(d+1))&=&(a_1-a_3)(d-1-a_1)+\\
&&(a_2-a_3)(d-1-a_2)+\\
&&\sum_{i=4}^k a_i(a_3-a_i)
\end{array}
\]
\item
For any $i,j$ with $1\le i<j\le k$, we have $a_i+a_j\le d$.
\item
Ordering the $a_i$ such that $a_1\ge a_2\ge a_3\ge a_4\dots$ we have 
\[
a_1+a_2+a_3\ge d+1.
\]
\end{enumerate}
\end{lem}

\begin{proof}
To prove assertion $(1)$, multiply the second Noether relation by $a_3$ and subtract it from the first; then rearrange the terms.

Assertion $(2)$ is equivalent to $(e_0-e({q_i})-e({q_j}))\cdot (h^{-1}(e_0))\ge 0$. Take an element $s\in \Sym({\B(\P^2_\k)})$ that sends  $q_i$ and $q_j$ onto $p_1=[1:0:0]$ and $p_2=[0:1:0]$. This implies that $\sigma_0 s$ maps $e_0-e({q_i})-e({q_j})$ onto $e(p_3)$, where $p_3=[0:0:1]$. The inequality is now equivalent to $(\sigma_0 s h^{-1})(e_0)\cdot e(p_3)\ge 0$, and follows from Lemma~\ref{Lem:MultiplicitiesPositive}.

Assertion $(2)$ implies that $d-1-a_i\ge 0$ for all $i$, since the number of base-points is bigger than $2$.
Then, assertion $(3)$ follows from the first one, because the right-hand side of the equality is non-negative.
\end{proof}

\subsection{Jonqui\`eres elements}

An element of $W_\infty$ is called a {\bf{Jonqui\`eres element}} with respect to $e_0-e(p)$ (or to the point $p\in \Bub(\P^2_\k)$)
if $h(e_0-e(p))=e_0-e(p)$. Jonqui\`eres twists $f$ in $\Bir(\P^2_\k)$ are conjugate to Jonqui\`eres elements within $\Bir(\P^2_\k)$ if $\k$ is algebraically closed. 

\begin{lem} \label{Lem:Jonqtrans}
Let $h\in W_\infty$, and let $p_1,q_1$ be points of  $\B(\P^2_\k)$ such that $h(e_0-e({p_1}))=e_0-e({q_1})$.
Let $m$ be the degree of $h$. There exists two subsets of $2m-2$ points $\{p_2,\dots, p_{2m-1}\}$ and $\{q_2,\dots,q_{2m-1}\}$ in $\B(\P^2_\k)$, such that $q_1\not=q_i$ and $p_1\not=p_i$ for $i\ge 2$, and such that the following hold:
$$\begin{array}{rcl}
h(e_0)&=&m e_0-(m-1)e({q_1})-\sum_{i=2}^{2m-1} e({q_i});\\
h^{-1}(e_0)&=&m e_0-(m-1)e({p_1})-\sum_{i=2}^{2m-1} e({p_i});\\
h(e({p_1}))&=&(m-1) e_0-(m-2)e({q_1})-\sum_{i=2}^{2m-1} e({q_i});\\
h^{-1}(e({q_1}))&=&(m-1) e_0-(m-2)e({p_1})-\sum_{i=2}^{2m-1} e({p_i});\\
h(e(p_i))&=&e_0-e({q_1})-e({q_i})\;  \mbox{ for } i=2,\dots,2m-1;\\
h^{-1}(e(q_i))&=&e_0-e({p_1})-e({p_i}) \; \mbox{ for } i=2,\dots,2m-1.
\end{array}$$
\end{lem}

\begin{proof}
Write $
h(e_0)=me_0-\sum_{i=1}^k a_i e({q_i}),
$ 
where the $a_i$  are positive integers. Since $m-a_1=(e_0-e({q_1}))\cdot h(e_0)=(e_0-e({p_1}))\cdot e_0=1$, we obtain $a_1=m-1$. From Noether equalities, one obtains $\sum _{i=2}^k a_i=\sum_{i=2}^k (a_i)^2=2m-2$. In particular, all $a_i$ are equal to  $1$ and $k=2m-2$: 
\[
h(e_0)=me_0-(m-1)e({q_1})-\sum_{i=2}^{2m-2} e({q_i}). 
\] 
From $h(e({p_1}))=h(e_0)-(e_0-e({q_1}))$ we deduce $h(e(p_1))=(m-1)e_0-(m-2)e({q_1})-\sum_{i=2}^{2m-2} e({q_i})$.

Apply now Lemma~\ref{Lem:MultiplicitiesPositive} for the elements $e(q_i)$, $i=1,\dots,2m-2$.
One finds a  subset $\{p_1,\ldots, p_l\}$ of $\B(\P^2_\k)$ and non-negative integers $b_i,c_{i1},\dots,c_{il}$ such that  
\[
h^{-1}(e({q_i}))=b_i e_0-\sum_{j=1}^l c_{ij} e({p_j}).
\] 
Since $e_0\cdot h^{-1}(e({q_i}))=h(e_0)\cdot e({q_i})=1$ and $e({p_1})\cdot h^{-1}(e({q_i}))=h(e({p_1}))\cdot e({q_i})=1$, we get $b_i=c_{ i 1}=1$. From $(h^{-1}(e({q_i})))^2=-1$ follows that $h^{-1}(e({q_i}))=e_0-e({p_1})-e({p_i})$ for some point $p_i\in \B(\P^2_\k)$ distinct from $p_1$. Doing this for each $i$, this defines $2m-1$ points $p_2,\dots,p_{2m-1}$. Then
\[
h(e({p_i}))=h(e_0-e({p_1}))-h(e_0-e({p_1})-e({p_i}))=e_0-e({q_1})-e({q_i}).
\] 
It remains to observe that $mh(e_0)-(m-1)h(e({p_1}))-\sum_{i=2}^{2m-1} h(e({p_i}))=e_0$, so $h^{-1}(e_0)=m e_0-(m-1)e({p_1})-\sum_{i=2}^{2m-1} e({p_i})$; the value of $h^{-1}(e({q_1}))$ follows now directly from $h(e_0-e({p_1}))=e_0-e({q_1})$.
\end{proof}

\begin{lem}\label{Lem:linear-growth}
Let $h_1,h_2\in W_\infty$ be Jonqui\`eres elements with respect to the same point $p\in  \B(\P^2_\k)$. Then 
\[
\deg(h_1h_2)<\deg(h_1)+\deg(h_2).
\]
In particular, the sequence $\{\deg (h_1)^n\}_{n\in \N}$ grows at most linearly and $h_1$ is not loxodromic.
\end{lem}

\begin{rem} 
Let $m\geq 2$ 
 be an integer and $h$ be an isometry of a finite dimensional hyperbolic space $\H_m$; here, $\H_m$
is one of the two connected components of the affine quadric $x_0^2=x_1^2+\ldots + x_m^2$ in $\R^{m+1}$, and $h$ is 
the restriction of an element of ${\sf{O}}_{1,m}(\R)$ preserving $\H_m$. Assume that $h$ is parabolic ; this means
that the linear transformation $h$ is not contained in a compact subgroup of ${\sf{O}}_{1,m}(\R)$ but does not have 
any eigenvalue of modulus $>1$. Then, $\parallel h^n\parallel$ grows quadratically with $n$. In other words, 
given any base point $e_0$ in $\H_m$, the sequence of distances $\dist(e_0, h^n(e_0))$ grows like $\cosh(c n^2)$
for some positive constant $c$.

Lemma \ref{Lem:linear-growth} shows that the behavior of parabolic transformations may be different if $\H_n$
is replaced by its infinite dimensional sibling. For instance, consider the birational transformation $f\colon (x,y)\mapsto (xy,y)$. Then $f_\p$ determines a Jonqui\`eres element  of $W_\infty$ with $\deg(f_\p^n)=n$; equivalently, $\dist(e_0, f_p^n(e_0))$ grows like $\cosh(c n)$. 
\end{rem}

\begin{proof}
Note that $\deg (h_1h_2)=h_1h_2(e_0)\cdot e_0=h_2(e_0)\cdot (h_1)^{-1}(e_0)$. Let $q_2,\dots,q_k$ be the  base-points  of $h_1$ or $(h_2)^{-1}$ which are distinct from $p$. Because $(e_0-e({p}))\cdot (h_i)^{\pm 1}(e_0)=(e_0-e({p}))\cdot e_0=1$, we get
\[
\begin{array}{rcl}
(h_1)^{-1}(e_0)&=&d_1e_0-(d_1-1)e({p})-\sum_{i=2}^k a_i e(q_i),\\
h_2(e_0)&=&d_2e_0-(d_2-1)e({p})-\sum_{i=2}^k b_i e(q_i),\end{array}
\]
for some non-negative integers $d_1$, $d_2$, $a_i$, and $b_i$. Moreover, $d_1=\deg(h_1)$, $d_2=\deg(h_2)$. Hence,
$$\begin{array}{l}
\deg (h_1h_2)=d_1d_2-(d_1-1)(d_2-1)-\sum a_ib_i\leq \deg(h_1)+\deg(h_2)-1.
\end{array}$$
\end{proof}

\begin{pro}\label{Lem:JonqRac}
Let $h$ be an element of $W_\infty$ of degree $d\geq 3$. Let $q_1,\dots,q_k$ be the points which are base-points of either $h$ or $h^{-1}$. 
Let $a_i$ and $b_i$ be the multiplicities of the base points:  $a_i=e({q_i})\cdot h^{-1}(e_0)$  and $b_i=e({q_i})\cdot h(e_0)$. Then one of the following properties holds:
\begin{enumerate}
\item
$d<3 (d-a_i)(d-b_i)$;
\item
$a_i=d-1=b_i$ for at least one base-point $q_i$; in that case $h$ is a Jonqui\`eres element with respect to $e_0-e(q_i)$.
\end{enumerate}
In particular if $h$ is not a  Jonqui\`eres element of $W_\infty$, then
\[
d-( a_i+b_i)/2 > \sqrt{d/3}
\]
for all $i=1,\dots,k$.
\end{pro}

\begin{proof}
To simplify the notation, write $e_i=e(q_i)$ for $i=1,\dots,k$.
Suppose that one of the $a_i$ is equal to $d-1$, and order the points to assume $a_1=d-1$. 
Since $a_1=h(e_1)\cdot e_0$, we have $h(e_0-e_1)\cdot e_0=1$; this implies that 
$h(e_0-e_1)=e_0-e_j$ for some $j$ and we deduce $b_j=d-1$. If $j=1$, then $h$
is a Jonqui\`eres element with respect to $e_0-e_1$. Assume now that $j\neq 1$.
Lemma~\ref{Lem:Jonqtrans} implies that
the $a_i$ for $i\neq 1$ and the $b_i$ for $i\neq j$  are equal all to $0$ or $1$. In particular,  for each $i$, we  get
$(d-a_i)(d-b_i)\ge d-1>d/3$. This proves that either $(1)$ or $(2)$ is satisfied when some $a_i$ is equal to $d-1$.

Assume now that $a_j<d-1$ for all indices $j$. In particular $h(e_0-e_j)$ is distinct from $e_0-e_j$. One only needs to show that $d<3 (d-a_i)(d-b_i)$ for all~$i$. 
Reordering the points one may assume $i=1$.
 
We have 
\[
h(e_0-e_1)\cdot e_0=(e_0-e_1)\cdot h^{-1}(e_0)=d-a_1\ge 2
\]
and we can write 
\begin{center}$
h(e_0-e_1) = (d-a_1) e_0 - \sum_{i=1}^k r_i e_i
$\end{center}
for some coefficients $r_i\geq 0$. Moreover
\[
1\le h(e_0-e_1)\cdot (e_0-e_1)=d-a_1-r_1.
\]
because $e_0-e_1$ and $h(e_0-e_1)$ are two isotropic elements of $\z_{\P^2_\k}$ 
in the boundary of $\H_{\P^2_\k}$ that are not orthogonal (hence their intersection is a positive integer).
As $h$ preserves the canonical linear form, 
\[
\sum_{i=1}^k r_i = 3 (d-a_1)-2.
\]
From $h(e_0)\cdot h(e_0-e_1)=e_0\cdot (e_0-e_1)=1$ we get 
\begin{center}$
d(d-a_1)-b_1r_1=\sum_{i=2}^{k} b_i r_i+1
$\end{center}
where $b_i=e_i\cdot h^{-1}(e_0)\ge 0$ is the multiplicity of $p_i$ as a base-point of $h$. Applying Lemma~\ref{Lem:EqualitiesNoether} to $h$ we have $b_1+b_i\le d$ for $i=2,\dots,k$, so $b_i\le d-b_1$. Because $d(d-a_1)-b_1r_1=d(d-a_1-r_1)+(d-b_1)r_1\ge d$, we get
\[
d\le1+(d-b_1)\sum_{i=2}^k r_i\le 1+(d-b_1)\cdot (3(d-a_1)-2)<3(d-b_1)(d-a_1).
\]
This concludes the proof of the alternative. Then, assertion $(1)$ implies
\[
\begin{array}{rcl}
(d-( a_i+b_i)/2)&=&\frac{ (d-a_i)+ (d-b_i)}{2}\\
& \ge & \sqrt{(d-a_i)(d-b_i)}\\
&>&\sqrt{d/3}.\end{array}
\]
This concludes the proof of the proposition.\end{proof}

\subsection{Halphen elements}

Given nine distinct points $q_i$ in $\Bub(\P^2_\k)$, the class
\[
K=3e_0-\sum_{i=1}^9 e(q_i)
\]
is an isotropic vector of $\z_{\P^2_\k}$. An  element $h$ of $W_\infty$ is an {\bf{Halphen element}}
with respect to such a class $K$ if $h(K)=K$.

\begin{lem}[Growing of Halphen type maps]\label{lem:Halphengrowth}
If  $h_1$ and $h_2$  are Halphen elements of $W_\infty$  with respect to the same isotropic class
 $K$ then
 \[
 \sqrt{\deg(h_1h_2)}<\sqrt{\deg(h_1)}+\sqrt{\deg(h_2)}.
 \]
In particular, the sequence $\{\deg (h_1)^n\}_{n\in \N}$ grows at most quadratically and $h_1$ is not loxodromic.
\end{lem}

\begin{proof}
Note that $\deg (h_1h_2)=h_1h_2(e_0)\cdot e_0=h_2(e_0)\cdot (h_1)^{-1}(e_0)$.
For $i=1,2$, write $d_i=\deg (h_i)$ and define $v_i$ by 
\[
(h_i)^{-1}(e_0)=\frac{d_i}{3}K+v_i.
\]
This decomposition satisfies $e_0\cdot v_i=0$,
\[
3=K\cdot e_0=K\cdot (h_i)^{\pm 1}(e_0)=K\cdot v_i
\]
because $K\cdot K=0$, and 
\[
1=h_i(e_0)^2=(v_i)^2+2\frac{d_i}{3}(K\cdot v_i)=(v_i)^2+2d_i.
\]
Writing $v_1=\sum a(p) e({p})$ and $v_2=\sum b(p)e({p})$, one gets
\[
(v_1\cdot v_2)^2=(\sum a(p)b(p))^2\leq \sum a(p)^2\cdot \sum b(p)^2=(v_1)^2\cdot (v_2)^2.
\] 
Hence,
\begin{eqnarray*}
\deg (h_1h_2)&=&(\frac{d_1}{3}K+v_1)\cdot (\frac{d_1}{3}K+v_2)\\
&=&d_1+d_2+v_1\cdot v_2\\
&\leq&d_1+d_2+\sqrt{(2d_1-1)(2d_2-1)}\\
&<&(\sqrt{d_1}+\sqrt{d_2})^2.
\end{eqnarray*}
\end{proof}

\begin{lem}\label{Lem:HalphenIneq}
Let $h$ be an element of $W_\infty$,  let $p_1,\dots,p_m$ be the  base-points of  $h$ and $h^{-1}$. 
Let $a_k$ and $b_k$ be the multiplicities of the base points: $a_k=h(e_0)\cdot e(p_k)$, $b_k=h^{-1}(e_0)\cdot e(p_k)$. 
If 
\[
d/3 \geq  3+ (3+\sum_{j={10}}^m b_j)\left( \max_{i=1}^9\vert 3a_i-d\vert   + \sum_{j={10}}^{m} a_j\right).
\]
then $h$ is an Halphen element with respect to $K=3e_0-\sum_{i=1}^9 e(p_i)$.
\end{lem}
\begin{proof}
To simplify the notation, we write $e_k=e({p_k})$ for $k=1,\dots,m$. Thus, 
\begin{eqnarray*}
h(e_0) & = & de_0-\sum_{i=1}^9 a_i e_i -\sum_{j=10}^m a_j e(q_j) \\
h^{-1}(e_0) & = & de_0-\sum_{i=1}^9 b_i e_i -\sum_{j=10}^m b_j e(q_j) 
\end{eqnarray*}
where the multiplicities $a_k$ and $b_k$ are non-negative integers ($1\leq k\leq m$). 

Denote by $K\in \z_{\P^2_\k}$ the element $3e_0-\sum_{i=1}^9 e_i$, and  write
\[
h(K)=ne_0-\sum_{i=1}^m c_k e_k\; ;
\]
to obtain such a formula, we may have to allow new base points, and thus increase the number $m$. By Lemma~\ref{Lem:MultiplicitiesPositive}, 
$n$ is a positive integer, and the $c_k$ are non-negative integers.

The canonical form $\kf$ vanishes on $K$. The invariance of $\kf$ gives 
\[
3n=\sum_{k=1}^m c_k
\]
and $h(K)\neq e_0-e_k$ for all $k$ (because $\kf(e_0-e_k)=2\neq 0$). From 
Lemma~\ref{Lem:MultiplicitiesPositive} we get $K\cdot (e_0-e_k)= h(K)\cdot h(e_0-e_k)  > 0$ for all indices $k$:
\[
c_k\leq n-1, \quad \forall 1\leq k  \leq m.
\]
Since $h$ preserves the intersection form, Hodge index theorem implies that $h(K)=K$ if and only if $h(K)\cdot K=0$, if and only if $h(K)$ is proportional to $K$. We now assume that $h$ does not fix $K$; this implies that $h(K)\cdot K$ is positive, hence that
\[
1\leq 3n-\sum_{i=1}^9c_i \quad {\text{and}} \quad \sum_{i=1}^9 c_i\leq 3n-1
\]
Since $h(K)\cdot e_0 = K\cdot h^{-1}(e_0)$, Noether equalities imply
\[
 n = 3d-\sum_{i=1}^9 b_i = 3+\sum_{j=10}^m b_j \geq 3.
\] 
We now compute $h(e_0)\cdot h(K)$, and get 
\[
3 = dn -\sum_{i=1}^9 a_i c_i - \sum_{j=10}^m a_j c_j = \frac{d}{3}(3n-\sum_{i=1}^9c_i) - \sum_{i=1}^9(a_i-\frac{d}{3})c_i-\sum_{j=10}^ma_jc_j.
\]
Then
\begin{eqnarray*}
d/3 & \leq & 3+ \sum_{i=1}^9(a_i-\frac{d}{3})c_i + \sum_{j=10}^ma_jc_j\\
 & \leq & 3 +  \left( \sum_{i=1}^9 c_i\right) \max_{i=1}^9{\left\vert a_i-\frac{d}{3}\right\vert}+ \max_{j\geq 10}{\left(c_j\right)} \sum_{j=10}^m a_j\\
 & \leq & 3 + (3n-1)\max_{i=1}^9{\vert a_i-\frac{d}{3}\vert} + (n-1)  \sum_{j=10}^m a_j \\
 & \leq & \left( 3+\sum_{j=10}^m b_j \right) \left( 3(\max_{i=1}^9{\vert a_i-\frac{d}{3}\vert}) + \sum_{j=10}^m a_j \right) + R
\end{eqnarray*}
where 
\[
R= 3-\max_{i=1}^9{\vert a_i-\frac{d}{3}\vert} -   \sum_{j=10}^m a_j< 3.
\]
This concludes the proof.
 \end{proof}

\subsection{Base points of Jonqui\`eres transformations: From $W_\infty$ to $\Bir(\P^2_\k)$}

Elements of Jonqui\`eres type in $W_\infty$ are not all realized by birational transformations of the plane. 
The precise constraints that the base points must satisfied are listed in the following proposition. Both the
statement and its proof are necessary to obtain Theorem~C and~D.

\begin{pro}\label{Prop:ExistenceJonq}
Let $p_1,\dots,p_{2m-1}\in \B(\P^2_\k)$ be $2m-1$ distinct points. There exists a Jonqui\`eres element $f\in \Bir(\P^2_\k)$ whose base-points are $p_1,\dots,p_{2m-1}$, and such that 
\[
{f_\p}^{-1}(e_0)=me_0-(m-1)e({p_1})-\sum_{i=2}^{2m-1} e({p_i})
\]
if and only if the points $p_i$ can be ordered so as to satisfy the following properties: 
\begin{enumerate}
\item
$p_1$  is a proper point of $\P^2_\k$;
\item
for any $i\ge 2$,  either $p_i$ is  a proper point of $\P^2_\k$ or $p_i$ is in the first neighbourhood of $p_j$ for some $j<i$;
\item
for all $i>j\ge 2$, there is no line of $\P^2_\k$ which passes through $p_1$, $p_i$, $p_j$;
\item
for all triples $k>j>i\ge 2$, at least one of the two points $p_j,p_k$ does not belong, as proper or infinitely near point, to the exceptional divisor associated to $p_i$;
\item
the number of points in $\{p_2,\dots,p_{2m-1}\}$ that belong, as proper or infinitely near points, to the exceptional divisor associated to $p_1$, is at most $m-1$ (these are points "proximate" to $p_1$);
\item
for any $k\ge 1$, each curve of $\P^2_\k$ of degree $k$ with multiplicity $k-1$ at $p_1$ passes through at most $k+m-1$ of the points $\{p_2,\dots,p_{2m-1}\}$.
\end{enumerate}
 \end{pro}
 
For Property (1), the case $m=2$ is special: There are three base-points of multiplicity $1$, and at least one of them is a proper point of the plane; 
we chose such a point and call it $p_1$.

\begin{proof} This result is well known to specialists (see \cite{Gizatullin:1982, Alberich:LNM}), but it is hard to extract this precise statement from the literature.

We first verify that the six properties are necessary. The case $m=1$ corresponds to linear projective transformations
and is easily handled with. So, assume that $f$ is a Jonqui\`eres transformation 
with base points $p_i$, degree $m\geq 2$, and 
\[
{f_\p}^{-1}(e_0)=me_0-(m-1)e({p_1})-\sum_{i=2}^{2m-1} e({p_i}).
\]
Let $C$ be a curve of degree $k$ with multiplicity $k-1$ at $p_1$. Let $I$ be the set of indices $i$ with $p_i\in C$. 
The class $ke_0- (k-1) e(p_1) - \sum_{i\in I} p_i$ is effective, and the class $f_\p^{-1}(e_0)$ is numerically effective; their
intersection is equal to 
\[
mk-(m-1)(k-1)-\vert I\vert = m+k-1 -\vert I\vert.
\]
Since this number is non-negative, Property (6) is satisfied. Properties (3) and (4) are proved along the same lines. 
To prove (2), assume that $p_i$ is not a proper point of the plane, and is not in the first neighbourhood of any other 
base point $p_i$. Then we find a point $q$ which is not a base point such that $e(q)-e(p_i)$ is effective. Intersecting
with ${f_\p}^{-1}(e_0)$ one gets $-1\geq 0$, a contradiction. Property (1) and (5) are proved in a similar way.

We now prove that Properties (1) to (6) are sufficient to construct such a Jonqui\`eres transformation.  
Denote by $\pi\colon X_2\to \P^2_\k$ the blow-up of $p_1$. The surface $X_2$ is the Hirzebruch surface $\mathbb{F}_1$: it admits a morphism $\eta_2\colon X_2\to \P^1$, whose fibres correspond to the lines of $\P^2_\k$ through the point $p_1$.

By property (2), the point $p_2$ is a proper point of $X_2$. Consider the fiber 
\[
F_2=(\eta_2)^{-1}(\eta_2(p_2));
\]
since this curve is the strict transform of the line through $p_1$ and $p_2$, property (3) implies that $F_2$ does not contain any of the $p_i$, $i\geq 3$, as proper or infinitely near point. Denote by $X_2\dasharrow X_3$ the birational map which consists of the blow-up of $p_2$, followed by the contraction of the strict transform of $F_2$. By construction $X_3$ is a Hirzebruch surface $\mathbb{F}_0$ or $\mathbb{F}_2$, and $p_3$ is now a proper point of $X_3$ by property (2). The pencil of lines through $p_1$ correspond to the ruling $\eta_3\colon X_3\to \P^1$ and the assumptions (3) and (4) imply that the fiber $F_3$ of $\eta_3$ through $p_3$ does not contain any $p_i$ with $i\geq 4$. 

Iterating this process, one constructs a sequence of maps $X_3\dasharrow X_4\dasharrow \dots \dasharrow X_{2m-1}$. 
 For $j=2,\dots,2m-1$, the surface $X_j$ is a Hirzebruch surface and comes with a morphism $\eta_j\colon X_j\to \P^1$, the fibers of which correspond to the lines of $\P^2_\k$ through $p_1$.  Moreover, the point $p_j$ is a proper point of $X_j$, and no other point of the fibre $F_j=(\eta_j)^{-1}(\eta_j(p_j))$ is one of the $p_i$; this latter condition is given by $(3)$ and $(4)$.

By construction, $X_{2m-1}$ is isomorphic to $\mathbb{F}_{r}$, for some odd integer $r$. We claim that $r=1$; this amounts to show that no section of $\eta_{2m-1}$ has self-intersection $\le -3$. If this section corresponds to the curve of $X_2=\mathbb{F}_1$ contracted by $\pi$ onto $p_1$ (its exceptional curve), it implies that at least $m$ of the points $p_2,\dots,p_{2m-1}$ belong, as proper or infinitely near, to the section, contradicting  hypothesis (5).
We can thus assume that the hypothetic section of self-intersection $\le -3$ corresponds to a curve of $\P^2_\k$ of degree $k$ passing through $p_1$ with multiplicity $k-1$. Such a curve has self-intersection $2k-1$ on $X_2=\mathbb{F}_1$. It must pass  through $l$ of the points  $p_2,\dots,p_{2m-1}$ and it must have self-intersection $(2k-1)-l+(2m-2-l)=2(k+m-l)-3$ on $X_{2m-1}=\mathbb{F}_r$. Assumption (6) implies that $l< k+m$: this shows that the self-intersection is $\ge -1$.

Therefore, $X_{2m-1}$ is isomorphic to $\mathbb{F}_1$. Contracting the exceptional divisor, we  obtain a birational morphism $X_{2m-1}$ on a surface which is isomorphic to $\P^2_\k$; hence, we can identify this surface with the initial plane $\P^2_\k$ and assume that the section is contracted to the point $p_1$. The composition of the maps $\P^2_\k=X_1 \dasharrow X_2\dasharrow \dots \dasharrow X_{2m-1}\to \P^2_\k$ is a birational map that preserves the pencil of lines through $p_1$, and whose base-points are exactly $p_1,\dots,p_{2m-1}$. Classical Noether equalities imply that the degree of the map is $m$, the multiplicity of $p_1$ is $m-1$ and the other multiplicities are $1$. This achieves the proof.
\end{proof}

\section{Dynamical degrees and conjugacy classes}\label{par:PROOF-D}
 
Our goal is to prove Theorem~C from the introduction. Thus, given a birational transformation 
$f$ of the projective plane with large dynamical degree $\lambda(f)$, we want to conjugate
$f$ by an element $g$ of $\Bir(\P^2_\k)$ to obtain $\deg(gfg^{-1})\leq C^{st}\lambda(f)^5$
(where the constant $C^{st}$ does not depend on $f$ and will be bounded by $4700$). 

Given $f\in \Bir(\P^2_\k)$ with $\lambda(f)>1$, the main arguments
may be summarized as follows. The degree of $f$ is large, compared to $\lambda(f)$, if and only if the axis $\ax(f_\p)$ is far away from the
base point $e_0$ of the hyperbolic space $\H_{\P^2_\k}$. Thus, we want to conjugate $f$ by $g$ so  that
the axis $g_\p(\ax(f_p))$ of $gfg^{-1}$ becomes closer to $e_0$.
A similar problem occurs in the proof of Noether-Castelnuovo theorem. When one wants to prove that quadratic transformations of 
the plane generate $\Bir(\P^2_\k)$, one starts with an element $f$ in $\Bir(\P^2_\k)$ and then looks for a  quadratic map $h$
such that $\deg(hf)<\deg(f)$; on $\H_{\P^2_\k}$, the problem is to find a quadratic map such that $g_\p(f_\p(e_0))$ is closer to $e_0$ than $f_\p(e_0)$ is.
We follow the same strategy as in Noether's proof. In other words, we first work with elements of $W_\infty$, and produce elements
$h\in W_\infty$ such that $h(\ax(f_\p))$ is close to $e_0$; then, as Castelnuovo did to correct Noether's error, we have to modify $h$ slightly in order to realize it as $g_\p$ for some $g$ in $\Bir(\P^2_\k)$. Proposition~\ref{Prop:ExistenceJonq} is used for this purpose.

\begin{rem}
The proof makes use of ideas from hyperbolic geometry (on the metric space $(\H_{\P^2_\k}, \dist)$). The distance
is given by $\cosh(\dist(u,v))=u\cdot v$, and it becomes rapidly annoying to transfer inequalities from distances to intersection
numbers, and vice versa. This is the reason why computations are done with intersections, even if they correspond to simple inequalities between distances. \end{rem}
 
\subsection{Axis, degree, distance to $e_0$}\label{par:Axis-Degree}

\subsubsection{Isotropic eigenvectors of loxodromic elements}
 Let $h$ be a loxodromic element of $W_\infty$, and let $\lambda(h)$ be its dynamical degree. We refer to Section~\ref{par:Isom-Hinfini-TLength} for the basic properties of loxodromic isometries of hyperbolic spaces. 

Write 
\[
h(e_0)=de_0-\sum_i a_i e(p_i)
\]
where the $a_i\geq 0$ are the multiplicities of the base points $p_i\in \Bub(\P^2_\k)$. The positive integer $d$ 
is the degree of $h$: $d=h(e_0)\cdot e_0$. As explained in \S~\ref{par:Isom-Hinfini-TLength}, $h$ preserves two isotropic lines $\R v_+$ 
and $\R v_-$, where $v_+$ and $v_-$ are elements of $\zz_{\P^2_\k}$ and 
\[
h(v_+)= \lambda(h) v_+, \quad h(v_-)= v_-/\lambda(h).
\]
With the normalization $v_+\cdot e_0=v_-\cdot e_0=1$, the vectors $v_+$ and $v_-$ are uniquely defined. Moreover, 
one has
\[
v_{+} = \lim\limits_{n\to \infty} \frac{h^n(e_0)}{\lambda^n},\quad 
v_{-} = \lim\limits_{n\to \infty} \frac{h^{-n}(e_0)}{\lambda^n}.
\]
Thus, we can write 
\[
v_+=e_0-\sum_{i} \alpha_i e(p_i),\quad v_-=e_0-\sum_i \beta_i e(p_i)
\]
where the $p_i$ form a countable subset of $\Bub(\P^2_\k)$; the set $\{p_i\}$ is contained in the union of all base points 
of all iterates $h^n$ for $n\in \Z$. The canonical form $\kf$ is $h$-invariant, hence $\kf(v_+)=\kf(v_-)=0$; since $v_+$ and $v_-$ are isotropic, we obtain
\[
\sum_i \alpha_i^2=\sum_i\beta_i^2=1 \quad \text{and} \quad \sum \alpha_i=\sum \beta_i=3.
\]
Since $v_+$ and $v_-$ are limits of sequences $h^n(e_0)\lambda^{-\vert n\vert}$, Lemma~\ref{Lem:MultiplicitiesPositive} implies the following positivity statement. 

\begin{lem}\label{EPositif}
 Let $u\in\z_{\P^2_\k}$ be one of the vectors
\[
e_0, \; e({p_1}), \; e_0-e({p_1}), \; 3e_0-\sum_{i=1}^l e(p_i)
\] for some distinct points $p_1,\dots,p_l\in \B(\P^2_\k)$.  Then $u \cdot v\ge 0$ for all $v\in \ax(h)$, 
 where $\ax(h)$ is the intersection of the plane generated by $v_{+}$ and $v_{-}$ with  $\H_{\P^2_\k}$.

If $h=f_\p$ for some element $f\in \Bir(\P^2_\k)$ and $C\in \zz_{\P^2_\k}$ is an effective divisor, then $C\cdot v\ge 0$ for every $v\in \ax(h)$.
\end{lem}

\subsubsection{Axis, and translation length}

The proof of the following lemma is straightforward (see \cite{Cantat-Lamy:Acta, Burger-Iozzi-Monod:2005}).
 
\begin{lem}\label{Lem:LoxoAxDeltaE}
 Let $h$ be a loxodromic element of $W_\infty$  of degree $d$ and dynamical degree $\lambda$.
Denote by $v_{+}$ and $v_{-}$ the eigenvectors of $h$ in $\zz_{\P^2_\k}$ for the eigenvalues $\lambda$ and $\lambda^{-1}$ such that  $v_{+}\cdot e_0=v_{-}\cdot e_0=1$.
Then 
\begin{itemize}
\item[(i)] $d=h(e_0)\cdot e_0=e_0\cdot  h^{-1}(e_0)$ and this degree is equal to $\cosh(\dist(h(e_0),e_0))$ and to $\cosh(\dist(e_0,  h^{-1}(e_0)))$;
\item[(ii)] $\log(\lambda(h))$ is the translation length of $h$, i.e.~the minimum of
$\dist(x,h(x))$ for  $x$ in $\H_{\P^2_\k}$. 
\item[(iii)]
The set of points of $\H_{\P^2_\k}$ that realize the translation length is the {\bf{axis}} of $h$; it coincides with the geodesic line 
\[
\ax(h)=\left.\left\{ \frac{1}{\sqrt{2v_+\cdot v_-}} \left(t v_+ + \frac{v_-}{t}\right)\ \right|\ t\in \R_{>0}\right\}\subset \H_{\P^2_\k}.
\]
\item[(iv)]
The distance $\delta$ from $e_0$ to the axis $\ax(h)$ satisfies
\[
\cosh(\delta)=\sqrt{\frac{2}{v_{+}\cdot v_{-}}}.
\]
It is realized by the projection of 
$e_0$ on the axis, i.e.~by the point
\begin{center}$
E= \sqrt{\frac{2}{v_{+}\cdot v_{-}}}  \frac{v_{+}+v_{-}}{2}.
$\end{center}
\end{itemize}
\end{lem} 

\subsubsection{Approximation of the points of the axis}

Denote by $\left\Vert\cdot \right\Vert$  the Euclidean norm on $\zz_{\P^2_\k}$, defined by 
\[
\left\Vert u \right\Vert=a_0^2+\sum_{p\in \B(\P^2_\k)} a_p^2.
\]
for $u=a_0e_0 + \sum a_p e(p)$. 
If the degree of $h$ is large, the euclidean norm of $d^{-1}h(e_0)-v_+$ must be small: 

\begin{lem}[Approximation of the  axis]\label{Lem:Distvv}
Let $h\in W_\infty$ be a loxodromic element of degree $d$ and dynamical degree $\lambda$. Then
\[
\begin{array}{rclrcl}
\left\Vert \frac{1}{d}h^{-1}(e_0)-v_{-}\right\Vert&<&\sqrt{\frac{2}{\lambda d}} & ; \quad
\left\Vert \frac{1}{d}h(e_0)-v_{-}\right\Vert&<&\sqrt{\frac{2\lambda}{ d}}\vspace{0.1 cm}\\ 

\left\Vert \frac{1}{d}h^{-1}(e_0)-v_{+}\right\Vert&<&\sqrt{\frac{2\lambda}{ d}}& ; \quad
 \left\Vert \frac{1}{d}h(e_0)-v_{+}\right\Vert&<&\sqrt{\frac{2}{\lambda d}}.
\end{array}
\]
Moreover, 
\[
\left\Vert \left( \frac{{h(e_0)}+{h^{-1}(e_0)}}{2d}\right)-\left(\frac{  v_{+}+ v_{-}}{2}\right)\right\Vert<\sqrt{\frac{2}{\lambda d}}, \]
and 
\[\frac{(\lambda-\frac{1}{\lambda})^2}{2d^2}<v_{+}\cdot v_{-}<\frac{1}{d}\left(\frac{1}{\lambda}+\lambda+2\right).
\]
\end{lem}

In particular, Lemma~\ref{Lem:LoxoAxDeltaE} implies
\begin{equation}\label{Eq:Ineq-Axis-Deg}
\sqrt{\frac{2d}{\frac{1}{\lambda}+\lambda+2}}< \cosh(\dist(e_0,\ax(h)))<\frac{2d}{\lambda-\frac{1}{\lambda}}.
\end{equation}

\begin{rem}
Note that the middle points $(v_++v_-)/2$ or $(h(e_0)+h^{-1}(e_0))/2$ are not contained in $\H_{\P^2_\k}$. 
From a geometric point of view, it would be better to scale them (by the square root of their self-intersection), but the formulas would be 
difficult to read.
\end{rem}

\begin{proof}
Let us derive the top four inequalities. 
From $b_i=e_i\cdot h^{-1}(e_0)=h(e_i)\cdot e_0$ we get
\[
\lambda^{-1}=h(v_-)\cdot e_0=v_-\cdot h^{-1}(e_0)=d-\sum_i b_i\beta_i.
\]
With Noether equality $\sum (b_i)^2=d^2-1$ and the relation  $\sum (\beta_i)^2=1$ we deduce that
\begin{eqnarray*}
\sum_i \left(\frac{b_i}{d}-\beta_i\right)^2 & = & \frac{\sum (b_i)^2}{d^2}-\frac{2\sum b_i\beta_i}{d}+\sum (\beta_i)^2 \\
 & = & \frac{2}{\lambda d}-\frac{1}{d^2}<\frac{2}{\lambda d}.
\end{eqnarray*}
This means that $\left\Vert \frac{h^{-1}(e_0)}{d}-v_{-}\right\Vert<\sqrt{\frac{2}{\lambda d}}$. If one replaces $v_{-}$ by $v_{+}$,  then $\lambda^{-1}$ is changed into $\lambda$; replacing $h$ with $h^{-1}$ yields the three other inequalities.

The fifth inequality follows by the triangular inequality. We now estimate the intersection product $v_+\cdot v_-$. 
On one hand,  
\[
\Vert  v_+-v_-\Vert^2 = 2v_+\cdot v_-=  2-2\sum \alpha_i\beta_i
\]
because $v_+$ and $v_-$ are isotropic. On the other hand, the above inequalities yield 
\[
\Vert v_+-v_-\Vert^2 <\left(\sqrt{\frac{2}{\lambda d}}+\sqrt{\frac{2\lambda}{ d}}\right)^2=\frac{2}{d}\left(\frac{1}{\lambda}+\lambda+2\right). 
\]
Thus, altogether, one gets
\[
v_{+}\cdot v_{-}=1-\sum \alpha_i\beta_i<\frac{1}{d}\left(\frac{1}{\lambda}+\lambda+2\right).
\]

In the other direction, note that $\frac{1}{\lambda}=d-\sum_i a_i\alpha_i$ and $\lambda=d-\sum_i a_i \beta_i$, and deduce
$\lambda-\frac{1}{\lambda}=\sum a_i (\alpha_i-\beta_i)$. Cauchy-Schwarz inequality yields 
\[
(\lambda-\frac{1}{\lambda})^2\le \sum (a_i)^2\cdot \sum (\alpha_i-\beta_i)^2=(d^2-1)\cdot (2 v_{+}\cdot v_{-}),
\]
so that
\[
v_{+}\cdot v_{-}\geq   \frac{(\lambda-\frac{1}{\lambda})^2}{2(d^2-1)}.
\]
This concludes the proof.
\end{proof}

\subsection{Decreasing the distance from $e_0$ to the axis}
We keep the same notation, $h$ being a loxodromic element of degree $d$, dynamical degree $\lambda$, ... In particular, 
$E$ is the projection of $e_0$ to the axis $\ax(h)$,  $a_i=e({p_i})\cdot h(e_0)$ and $b_i=e({p_i})\cdot h^{-1}(e_0)$. The middle
point of $h(e_0)$ and $h^{-1}(e_0)$ is  
\[
de_0-\sum_i c_i e(p_i), \quad \text{with} \quad c_i=\frac{a_i+b_i}{2}.
\]
\subsubsection{Strategy}
The following lemma provides a strategy to decrease the distance between $\ax(h)$ and $e_0$ by conjugacy with 
a quadratic element $g$ of $W_\infty$. Similarly, if $\gamma_i=(\alpha_i+\beta_i)/2$ then
\[
\frac{1}{2} ( v_{+}+ v_{-})= e_0-\sum \gamma_i e({p_i}).
\]

\begin{lem}\label{Lem:Decreasing}
Let $p_1,p_2,p_3$ be three distinct points of $\B(\P^2_\k)$.
If 
\[
\sum_{i=1}^3c_i\ge d+\frac{5}{2}\sqrt{\frac{d}{\lambda}},
\]
 then
\begin{enumerate}
\item
$(e({p_1})+e({p_2})+e({p_3})-e_0)\cdot E>\frac{\sqrt{2}(\frac{5}{2}-\sqrt{6})}{\sqrt{\lambda\left(\frac{1}{\lambda}+\lambda+2\right)}}=\frac{5-2\sqrt{6}}{\sqrt{2}(\lambda+1)}$
\item
$\cosh(\dist(e_0,g(E)))<\cosh(\dist(e_0,E))-\frac{5-2\sqrt{6}}{\sqrt{2}(\lambda+1)}$,\\
for any quadratic element $g\in W_\infty$ with base-points at $p_1,p_2,p_3$.
\end{enumerate}
\end{lem}

\begin{proof}
If necessary, we enlarge the set of base points $\{p_i\}\subset \B(\P^2_\k)$ to include the three points $p_1$, $p_2$, and $p_3$ (allowing 
multiplicities equal to $0$). From Lemma~\ref{Lem:Distvv} we know that
\[
\left\Vert \left( \frac{{h^{-1}(e_0)}+{ h(e_0)}}{2d}\right)-\left(\frac{v_{+}+ v_{-}}{2}\right)\right\Vert<\sqrt{\frac{2}{\lambda d}},
\]
which may be written as 
\[
\sum \left(\frac{c_i}{d}-\gamma_i\right)^2<\frac{2}{\lambda d}.
\] 
Apply Cauchy-Schwarz inequality for the scalar product between the vector $(1,1,1)\in \R^3$ and the vector $((c_i/d)-\gamma_i)_{i=1}^3$, to get 
\[
\sum_{i=1}^3 \lvert\frac{c_i}{d}-\gamma_i\rvert<\sqrt{\frac{6}{\lambda d}}.
\]
By assumption, we have $(\sum_{i=1}^3\frac{c_i}{d})-1\ge \frac{5}{2\sqrt{\lambda d}}$; hence
\[
\gamma_1+\gamma_2+\gamma_3-1>\frac{\frac{5}{2}-\sqrt{6}}{\sqrt{\lambda d}}.
\]
Since $E=\sqrt{\frac{2}{v_+\cdot v_-}} \frac{  v_+ +   v_-}{2}$ and $v_{+}\cdot v_{-}<\frac{1}{d}\left(\frac{1}{\lambda}+\lambda+2\right)$ (Lemma~\ref{Lem:Distvv}), we obtain
\[
(e({p_1})+e({p_2})+e({p_3})-e_0)\cdot E >\sqrt{\frac{2}{{\frac{1}{d}\left(\frac{1}{\lambda}+\lambda+2\right)}}}\left(\frac{\frac{5}{2}-\sqrt{6}}{\sqrt{\lambda d}}\right)=\frac{\sqrt{2}(\frac{5}{2}-\sqrt{6})}{\sqrt{\lambda\left(\frac{1}{\lambda}+\lambda+2\right)}}
\]
If $g$ is a quadratic element of $W_\infty$ with base-points $p_1,p_2,p_3$,  then $g^{-1}(e_0)=2e_0-e({p_1})-e({p_2})-e({p_3})$. Consequently, 
\begin{eqnarray*}
\cosh(\dist(e_0,E))-\cosh(\dist(e_0,g(E))) & = & e_0\cdot E-e_0\cdot g(E) \\
& = & e_0\cdot E-g^{-1}(e_0)\cdot E\\
& = & (e({p_1})+e({p_2})+e({p_3})-e_0)\cdot E
\end{eqnarray*}
and the conclusion follows from the previous inequality. 
\end{proof}

\subsubsection{Noether inequality for the axis of $h$}
\begin{lem}\label{Lem:Noetherci}
The coefficients $c_i=(a_i+b_i)/2$ of $(h(e_0)+h^{-1}(e_0))/2$ satisfy
\[
\begin{array}{rcl}
\sum_{i=1}^k c_i&=&3d-3;\\
\sum_{i=1}^k (c_i)^2&>& (d^2-1)-\frac{d}{2}\left(\lambda^{-1}+\lambda+2\right);
\end{array}
\]

and 
\[
\begin{array}{rcl}
(d-1)(c_1+c_2+c_3-(d+1))&>&(c_1-c_3)((d-1)-c_1)+\\
&&(c_2-c_3)((d-1)-c_2)+\\
&&\sum_{i=4}^k c_i(c_3-c_i)\\
&&-\frac{d}{2}\left(\lambda^{-1}+\lambda+2\right)\end{array}
\]
\end{lem}
\begin{proof}
The first equality directly follows from Lemma~\ref{Lem:EqualitiesNoether}, which asserts that $\sum_{i=1}^k a_i=\sum_{i=1}^k b_i=3d-3$. By Lemma~\ref{Lem:Distvv}, we have 
$$\sum_{i=1}^k \left(\frac{a_i}{d}-\frac{b_i}{d}\right)^2< \left(\sqrt{\frac{2}{\lambda d}}+\sqrt{\frac{2\lambda}{ d}}\right)^2=\frac{2}{d}\left(\frac{1}{\lambda}+\lambda+2\right).$$
Since $\sum (a_i)^2=\sum (b_i)^2=d^2-1$, we get $d^2-1-\sum a_i b_i<d\left(\frac{1}{\lambda}+\lambda+2\right)$, hence
\[
\begin{array}{rcl}
\sum ( a_i+ b_i)^2&=&2(d^2-1)+2\sum a_i b_i\\
&>&4(d^2-1)-2d\left(\frac{1}{\lambda}+\lambda+2\right).\end{array}
\]
Dividing by $4$, we obtain the first inequality. Then, subtract $c_3\sum c_i=c_3\cdot  3(d-1)$ to obtain
successively
\[
\sum_{i=1}^k (c_i)^2-c_3\sum_{i=1}^k c_i> (d-1)((d+1)-3c_3)-\frac{d}{2}\left(\lambda^{-1}+\lambda+2\right)
\]
and
\[
 (d-1)(3c_3-(d+1))+\sum_{i=1}^3 c_i(c_i-c_3)>\sum_{i=4}^k c_i(c_3-c_i)-\frac{d}{2}\left(\lambda^{-1}+\lambda+2\right).
 \]
The   inequality follows by rearranging the terms as in the proof of Noether's inequality.
\end{proof}

\subsubsection{Decreasing the distance to the axis by conjugacy in $W_\infty$}

\begin{pro}\label{Pro:Lambda106}
Let $h\in W_\infty$ be of degree $d$ and dynamical degree $\lambda>10^6$. 
Let $p_1,\dots,p_k$ be the base-points of $h$ and $h^{-1}$, and   $c_i=(a_i+b_i)/2$ be the average of their multiplicities. Order the $c_i$ in such a way that $c_1\geq c_2\ge c_3\geq \dots \geq c_k$. If $d>24\lambda^3$, then 
\[
c_1+c_2+c_3\geq d+\frac{5}{2}\sqrt{\frac{d}{\lambda}}.
\]
\end{pro}
\begin{rem}
Together with Lemma~\ref{Lem:Decreasing}, Assertion (2), this proposition provides a way to conjugate a loxodromic element $h$ of $W_\infty$
by a quadratic involution $g\in W_\infty$ so as to decrease the distance from $e_0$ to the axis.\end{rem}

\begin{proof}
We use the inequality of
 Lemma~\ref{Lem:Noetherci} and observe that $(c_2-c_3)(d-1-c_2)$ can be removed, as it is non-negative. Indeed, by hypothesis we have $c_2\ge c_3$, and Noether equalities (\ref{eq:NoetherEquality}) imply that $a_2\le d-1$ and $b_2\le d-1$. This yields 
\begin{equation}\begin{array}{l}\label{TheEq}
\sum_{i=1}^3c_i- d>1+\frac{(c_1-c_3)(d-1-c_1)}{d-1}+\frac{\sum_{i=4}^k c_i(c_3-c_i)}{d-1}-\frac{d(\lambda^{-1}+\lambda+2)}{2(d-1)}.\end{array}\end{equation}
As a consequence, the result follows from 
\[
(c_1-c_3)(d-1-c_1)+\sum_{i=4}^k c_i(c_3-c_i)>\frac{1}{2}d(\lambda^{-1}+\lambda+2)+ (d-1)\frac{5}{2}\sqrt{\frac{d}{\lambda}}.
\]
The hypothesis on $\lambda$ implies $\frac{5}{2}\sqrt{\frac{d}{\lambda}}<{\sqrt{d}}/{400}$;  thus, it suffices to prove that 
\begin{equation}\label{TheEqToProve}
(c_1-c_3)((d-1)-c_1)+\sum_{i=4}^k c_i(c_3-c_i)>\frac{d\sqrt{d}}{200}.
\end{equation}
Note that $2(d-1- c_1)=((d-1)-a_1)+((d-1)-b_1)$  is non-negative because $d-1\ge a_1$ and $d-1\ge b_1$ (this follows, for instance, from Lemma~\ref{Lem:EqualitiesNoether}). Since $c_1\ge c_3\ge c_i$ for $i\ge 4$, every term of the left sum is non-negative.

We do a case by case study, and show that Inequality~(\ref{TheEqToProve}) holds in each case. 

\smallskip

\noindent{\bf{Step 1.--}}  Assume, first, that $c_3\ge ((d+\sqrt{d}/400)-c_1)/2$. In this situation,  the result directly follows: 
\[
c_1+c_2+c_3 \geq c_1+2c_3 \geq d+\sqrt{d}/400 \geq d+\frac{5}{2}\sqrt{\frac{d}{\lambda}}
\]
because $\lambda>10^6$. 

\smallskip

\noindent{\bf{Step 2.--}}  Hence, we  assume $c_3< ((d+\sqrt{d}/400)-c_1)/2$ in what follows. This yields 
\[
(c_1-c_3)(d-1-c_1)>\frac{1}{2}(3c_1-(d+\sqrt{d}/400)) (d-1-c_1).
\]
The right-hand side is a quadratic polynomial in the variable $c_1$; it vanishes at $d/3+\sqrt{d}/1200$ and $d-1$, and is positive between
these two roots. If $c_1\ge d/3+\sqrt{d}/100$ Proposition~\ref{Lem:JonqRac} implies that  
\[
c_1\in [d/3+\sqrt{d}/100,d-\sqrt{d/3}].
\]  
Both extremities of this interval are between the two roots of the above quadratic polynomial; thus, the infimum of this polynomial function on this
interval is equal to its value at $d/3+\sqrt{d}/100$ or at  $d-\sqrt{d/3}$. One easily estimates these two values from below; the first one is
\[
\frac{11}{800}\sqrt{d}\cdot (2/3d-\sqrt{d}/100-1)> \frac{d\sqrt{d}}{200}
\]
because $d> 24 \cdot 10^{18}$, and the second one is
\[(2d-(\sqrt{3} +1/400)\sqrt{d})/2\cdot (\sqrt{d}/\sqrt{3}-1)>\frac{d}{2} \cdot \frac{\sqrt{d}}{2}>\frac{d\sqrt{d}}{200}
\]
for the same reason. This implies Inequality~(\ref{TheEqToProve}).

\smallskip

\noindent{\bf{Step 3.--}} We can then assume that
\begin{equation}\label{Eq:c1d3}
c_1< d/3+\sqrt{d}/100.
\end{equation}
In particular, we have $d-1-c_1\ge 2d/3-\sqrt{d}/100-1>(0.6)\cdot d$. If $c_1-c_3\ge \sqrt{d}/100$, we obtain Inequality~(\ref{TheEqToProve}); hence, we may add the assumption 
\begin{equation}\label{Eq:c1d3bis}
c_3>c_1-\sqrt{d}/100.
\end{equation}
Lemma~\ref{Lem:Distvv} provides the inequality 
\[
\frac{1}{d^2}\sum_{i=1}^k (a_i-b_i)^2<\left(\sqrt{\frac{2}{\lambda d}}+\sqrt{\frac{2\lambda}{ d}}\right)^2=\frac{2}{d}\left(\frac{1}{\lambda}+\lambda+2\right).
\]
In particular $
(a_r-b_r)^2<2d\left(\frac{1}{\lambda}+\lambda+2\right)<{2.01}\cdot{\lambda d} 
$
for all indices $r$. Choosing $r$ such that $a_r\ge a_j$ for all $j$, we know from Noether inequality that  $a_r\geq d/3$ (see Lemma~\ref{Lem:EqualitiesNoether}); this leads to $b_r>d/3-1.42\sqrt{\lambda d}$, $c_r=(a_r+b_r)/2>d/3-0.71\sqrt{d\lambda}$, and $c_3>c_1-\sqrt{d}/100\ge c_r-\sqrt{d}/100>d/3-(3/4)\cdot\sqrt{d\lambda}.$ Hence,
\begin{equation}\label{Eq:c3d334}
c_3>d/3-(3/4)\cdot\sqrt{d\lambda}>3d/10,
\end{equation}
where the last inequality follows from $d>24\lambda^3>10^{12}\lambda$.

 For each $i\geq 4$, define $\epsilon_i=\min \{c_3-c_i,c_i\}$, and note that $c_i\cdot (c_3-c_i)\ge\epsilon_i\cdot c_3/2$. 
 This gives 
 \[
 \sum_{i=4}^k c_i\cdot (c_3-c_i)\ge (\sum_{i=4}^k \epsilon_i)\cdot c_3/2.
 \] 
  If $\sum_{i=4}^k \epsilon_i>\sqrt{d}/15$, Inequality~(\ref{TheEqToProve}) follows from $c_3>3d/10$ (Inequality~\eqref{Eq:c3d334}).

\smallskip

\noindent{\bf{Step 4.--}}  We can now add the inequality  
\begin{equation}\label{Eq:epsi-esti}
\sum_{i=4}^k \epsilon_i\le\frac{1}{15}\sqrt{d}
\end{equation}
to our assumptions. Our goal  is to derive a contradiction from these assumptions.
 
 Denote by $l$ the largest index such that $c_l\ge c_3/2$. 
For $i=4,\dots,l$, the inequality $c_i\ge c_3/2$ corresponds to $c_i\le c_3-c_i$, hence  $\epsilon_i=c_3-c_i$. 
This yields, together with Inequality~(\ref{Eq:epsi-esti}), the following estimates for $\sum_{i=4}^l c_i$:
$$(l-3)c_3-\sqrt{d}/15<(l-3)c_3-\sum_{i=4}^l \epsilon_i=\sum_{i=4}^l c_i\le (l-3)c_3.$$
Moreover, $c_1<c_3+\sqrt{d}/100$ (Inequality~\ref{Eq:c1d3bis}), so
$3c_3\le c_1+c_2+c_3 < 3c_3+\sqrt{d}/50$. Adding the two estimates yields
 \[
 lc_3-\sqrt{d}/15<\sum_{i=1}^l c_i<lc_3+\sqrt{d}/50.
 \]
 Because $\sum_{i=1}^l c_i\le \sum_{i=1}^k c_i=3d-3$ (Lemma~\ref{Lem:Noetherci}) and $c_3>3d/10$ (Equation~\ref{Eq:c3d334}), we have  $l(3d/10)-\sqrt{d}/15< 3d-3$, that gives $l<10$, hence $l\le 9$.

 From $\sum_{i=l+1}^k c_i=\sum_{i=l+1}^k \epsilon_i<\sqrt{d}/15$ (Inequality~\eqref{Eq:epsi-esti}), one gets $3d-3=\sum_{i=1}^k c_i<lc_3+\sqrt{d}/50+\sqrt{d}/15$. Together with $c_3\le c_1<d/3+\sqrt{d}/100$ (Inequality~\eqref{Eq:c1d3}), we get $3d-3<l(d/3+\sqrt{d}/100)+\sqrt{d}/50+\sqrt{d}/15$, so $l\ge 9$.  
 Since $l=9$, Inequality \eqref{Eq:epsi-esti} yields
\[
\sum_{i=10}^k c_i=\sum_{i=10}^k \epsilon_i< \sqrt{d}/15.
\]

In other words, there is a concentration of the  multiplicities on the $9$ points $p_1,\dots,p_9$:  $h$ behaves like a Halphen element of $W_\infty$.

\smallskip

\noindent{\bf{Step 5.--}} To derive a contradiction, we apply 
Lemma~\ref{Lem:HalphenIneq}:
\begin{equation}\label{Eq:LemLoxoHalp}
d/3<  (3+\sum_{i={10}}^k a_i)\left( (3\max\{b_i\}_{i=1}^9-d)  + \sum_{j={10}}^{k} b_j\right)+3,
\end{equation}
because $h$ is loxodromic. 

Let us estimate $\sum_{i=10}^k c_i$ from below.
For $i=1,\dots,9$, write $\mu_i=d/3-c_i$ and observe that Inequality~\eqref{Eq:c1d3} implies 
\[
-\sqrt{d}/100<\mu_1\le \mu_2\le\dots\le\mu_9.
\] 
Lemma~\ref{Lem:Noetherci} yields $3d-3=\sum_{i=1}^k c_i=3d-\sum_{i=1}^9 \mu_i+\sum_{i=10}^k c_i$, so that
\[
\sum_{i=1}^9 \mu_i=3+\sum_{i=10}^k c_i<3+\sqrt{d}/15<\sqrt{d}/14,
\]
and
\[
\mu_9<\sqrt{d}/14-\sum_{i=1}^8\mu_i<\sqrt{d}/14+8\sqrt{d}/100<\sqrt{d}/6.
\] 
In particular, $\sum_{i=1}^9 (\mu_i)^2<9(\sqrt{d}/6)^2=d/4$. 

We also have $\sum_{i=10}^k (c_i)^2\le (\sum_{i=10}^k c_i)^2<d/225$.
We compute 
\begin{eqnarray*}
\sum_{i=1}^k (c_i)^2 & = & \sum_{i=1}^9 (d/3-\mu_i)^2+\sum_{i=10}^k (c_i)^2 \\
& = & d^2-\frac{2}{3}d \sum_{i=1}^9 \mu_i +\sum_{i=1}^9 (\mu_i)^2+\sum_{i=10}^k (c_i)^2\\
& < & d^2-\frac{2}{3}d(3+\sum_{i=10}^k c_i)+\frac{1}{4}d+\frac{1}{225}d\\
& < & d^2-\frac{2}{3} d(\sum_{i=10}^k c_i)-\frac{3}{2}d.
\end{eqnarray*}

On the other hand,  Lemma~\ref{Lem:Noetherci} yields   
\[
\sum_{i=1}^k (c_i)^2>d^2-1-\frac{d}{2}\left(\lambda^{-1}+\lambda+2\right)>d^2-(0.501) \cdot (\lambda d),
\]  
and we obtain 
\[
d^2-(0.501)\cdot (\lambda d)<d^2-\frac{2}{3} d(\sum_{i=10}^k c_i)-\frac{3}{2}d,
\] 
hence
\[
\sum_{i=10}^k c_i<(0.501)\frac{3}{2}\lambda-\frac{9}{4}< 0.7516\lambda .
\]
 In particular, either $\sum_{i=10}^k a_i<0.7516\lambda$ or $\sum_{i=10}^k b_i<0.7516\lambda$. We assume the first (otherwise we apply Lemma~\ref{Lem:HalphenIneq} to $h^{-1}$ instead of $h$).

For each $i$, recall that $(a_i-b_i)^2<2d\left(\frac{1}{\lambda}+\lambda+2\right)<2.01d\lambda$ and, consequently, $\vert a_i-b_i\vert <1.42\sqrt{\lambda d}$, hence 
\[
b_i<c_i+0.71\sqrt{\lambda d}< d/3+\sqrt{d}/100+0.71\sqrt{\lambda d}<d/3+0.72\sqrt{\lambda d}.
\] 
This yields $(3\max\{b_i\}_{i=1}^9-d)<2.16\sqrt{d\lambda}$.
Equation~\ref{Eq:LemLoxoHalp} implies
$$\begin{array}{rcl}
d&<&  3(\sum_{i={10}}^k a_i)\left( (3\max\{b_i\}_{i=1}^9-d)  + \sum_{j={10}}^{k} b_j\right)\\
&<&3\cdot (0.7516)\cdot \lambda (2.16\sqrt{d\lambda}+2\cdot 0.7516\lambda)\\
&<&4.88\sqrt{d}\lambda^{3/2}.\end{array}$$
In particular, $\sqrt{d}<4.88\lambda^{3/2}$, contradicting the hypothesis $d>24\lambda^3$.
\end{proof}

\subsection{From $W_\infty$ to the Cremona group}
We can now prove Theorem~C, Assertion (1), which we rephrase as follows.
\begin{thm}\label{thm:thmD-rephrased}
Let $f\in \Bir(\P^2_\k)$ be a loxodromic element of dynamical degree $\lambda>10^6$. 
There exists a birational map $g\in \Bir(\P^2_\k)$ such that $\deg ( gfg^{-1} ) <4700\lambda^5$.
\end{thm}

To prove Theorem~\ref{thm:thmD-rephrased}, we denote by $\ax(f_\p)$ the axis of $f$ and by $E$ 
 the projection of the point $e_0$ onto $\ax(f_\p)$ (see Lemma~\ref{Lem:LoxoAxDeltaE}). We fix a point $p_1\in \B(\P^2_\k)$ such that $e(p_1)\cdot E\ge e(q)\cdot E$ for each $q\in \B(\P^2_\k)$. We can choose $p_1$ so that it is a proper point of the plane.

\subsubsection{The involutions $\sigma_\Omega$}\label{par:SigmaOmega}

For each $q\in \B(\P^2_\k)\backslash \{p_1\}$, the vector
\[
w_q:=\frac{e_0-e({p_1})}{2}-e({q})\in \z_{\P^2_\k}\otimes \Q 
\]
has self-intersection $-1$; we denote by $\nu_q$ the orthogonal reflection of $\z_{\P^2_\k}\otimes \Q$ with respect to the hyperplane orthogonal to $w_q$. The $\Q$-linear automorphism $\nu_q$ is given by 
\[
u\mapsto u+2(u\cdot w_q)w_q=u-2\frac{(u\cdot w_q)}{(w_q\cdot w_q)} w_q.
\] 
The transformations $\nu_q$, for $q$ in $\B(\P^2_\k)\backslash \{p_1\}$, constitute a family of commuting involutions 
because $w_q$ is orthogonal to $w_{q'}$ if $q\neq q'$.

For any finite set $\Omega \subset \B(\P^2_\k)\backslash \{p_1\}$ consisting of an even number $2m-2\ge 0$ of points, we denote by $\sigma_\Omega$ the composition of all $\nu_q$ for $q$ in $\Omega$. By induction, the transformation $\sigma_\Omega$ is the automorphism of $\z_{\P^2_\k}$ given by
\[
\begin{array}{rcl}
\sigma_\Omega(e_0)&=&m e_0-(m-1)e({p_1})-\sum_{q\in \Omega} e({q}) \\
& = & e_0+\sum_{q\in \Omega} (\frac{e_0-e({p_1})}{2}-e({q}));\\
\sigma_\Omega(e({p_1}))&=&(m-1) e_0-(m-2)e({p_1})-\sum_{q\in \Omega} e({q})\\
& = &e({p_1})+\sum_{q\in \Omega} (\frac{e_0-e({p_1})}{2}-e({q}));\\
\sigma_\Omega(e_q)&=&e_0-e({p_1})-e({q})\mbox{ if }q\in \Omega;\\
\sigma_\Omega(e_q)&=&e_q\mbox{ if }q\in \B(\P^2_\k)\backslash(\{p_1\}\cup \Omega).\end{array}
\]
The following statement is easily proved. It implies that the $\sigma_\Omega$ form a subgroup of $W_\infty$. 
\begin{lem} For all pairs $\Omega$, $\Omega'$ of subsets of $\B(\P^2_\k)\backslash \{p_1\}$ with an even number of elements, and 
for all pairs $(q,q')$ of distinct points in $\B(\P^2_\k)\backslash \{p_1\}$, we have:
\begin{itemize}
\item[(i)]  $\sigma_\Omega$ of $\z_{\P^2_\k}$ is  the product of all $\nu_q$, $q\in \Omega$, and, as such, is an involution;
\item[(ii)] $\sigma_{\Omega}\cdot \sigma_{\Omega'}=\sigma_{\Omega\cup \Omega'\backslash \Omega\cap \Omega'}$;
\item[(iii)] $\sigma_{\{q,q'\}}$ is conjugate to the composition of $\sigma_0\in W_\infty$ with a transposition $\eta\in \Sym({\B(\P^2_\k)})$.  
\item[(iv)] $\sigma_\Omega$ is an element of $W_\infty$.
\end{itemize}
 \end{lem}

\subsubsection{Minimal sets}
 
To simplify the notation, we define
\[
\delta=\frac{\sqrt{2}(\frac{5}{2}-\sqrt{6})}{\sqrt{\lambda\left(\frac{1}{\lambda}+\lambda+2\right)}}= \frac{5-2\sqrt{6}}{\sqrt{2}(\lambda+1)},
\]
which is the value given by Lemma~\ref{Lem:Decreasing}.
 
We denote by $\mathcal{S}$ the collection  of all finite subsets $\Omega\subset \B(\P^2_\k)\backslash \{p_1\}$ consisting of an even number of points, and satisfying the following properties:
\begin{enumerate}
\item
Each point $q\in \Omega$ is either a proper point of $\P^2_\k$ or is in the first neighbourhood of a point $q'\in \Omega$.
\item
There is no line of $\P^2_\k$ passing through $p_1$ and through two distinct points of $\Omega$ (taking here also infinitely near points).
\item
For any three distinct points $q_i,q_j,q_k\in \Omega$, the points $q_j,q_k$ cannot simultaneously belong, as proper or infinitely near points, to the exceptional curve obtained by blowing-up $q_i$ (this means that $q_j,q_k$ are not both "proximate" to $q_i$).
\item
Either $\Omega=\emptyset$ or 
\[
\cosh\left( \dist(e_0,\sigma_{\Omega}\ax(f_\p)) \right)<\cosh\left( \dist(e_0,\ax(f_\p)) \right)-\delta.
\]
(note that $\sigma_\Omega(\ax(f_\p))=\ax(\sigma_{\Omega}f_\p\sigma_{\Omega}^{-1}))$)
\end{enumerate}

\smallskip

We put a partial order on $\mathcal{S}$, defined by $\Omega'< \Omega$ if and only if
\[
\cosh\left( \dist(e_0,\sigma_{\Omega'}\ax(f_\p)) \right)<\cosh\left( \dist(e_0,\sigma_{\Omega}\ax(f_\p)) \right) -\delta.
\]

The set $\mathcal{S}$ is of course not empty, since it contains the set $\Omega=\emptyset$. The definition of the order implies that there is no infinite decreasing sequence $\Omega_1>\Omega_2>\dots$ in $\mathcal{S}$, so $\mathcal{S}$ contains minimal elements. We now prove the following assertion:

\smallskip

\begin{center}\begin{tabular}{lp{11.5cm}}
$(\star)$ & {\sl{Let $\Omega$ be a minimal element of $\mathcal{S}$, and let $E$ be the projection of $e_0$ onto $\sigma_{\Omega}\ax(f_\p)$. Either 
\[
\deg(\sigma_{\Omega}f_\p\sigma_{\Omega}^{-1})\le 24\lambda^3
\] 
or there exists $q\in \B(\P^2_\k)$ which satisfies $e(q)\cdot E>e({p_1})\cdot E$}}.\end{tabular}\end{center}

\smallskip

To prove $(\star)$, we take  an element $\Omega$ of $\mathcal{S}$, we assume that $\deg(\sigma_{\Omega}f_\p\sigma_{\Omega})> 24\lambda^3$ and that $e({p_1})\cdot E\ge e({q})\cdot E$ for all $q\in \B(\P^2_\k)$, and we show that $\Omega$ is not minimal in $\mathcal{S}$. 

\smallskip

$\bullet$ Proposition~\ref{Pro:Lambda106} and Lemma~\ref{Lem:Decreasing} provide three distinct points $q_1$, $q_2$, and $q_3$ such that $(e({q_1})+e({q_2})+e({q_3})-e_0)\cdot E>\delta$. Because $e({p_1})\cdot E\ge e({q_i})\cdot E$ for all indices $i$, we can assume that $q_1=p_1$. Since 
\[
(\sigma_{\{q_2,q_3\}})^{-1}(e_0)=2e_0-e({p_1})-e({q_2})-e({q_3}),
\] 
we obtain $E\cdot e_0-\sigma_{\{q_2,q_3\}}(E)\cdot e_0=E\cdot e_0-E\cdot (\sigma_{\{q_2,q_3\}})^{-1}(e_0)>\delta$; this implies that 
\[
e_0\cdot \sigma_{\{q_2,q_3\}}(E)<\cosh(\dist(e_0,\sigma_{\Omega}\ax(f_\p)))-\delta.
\]
As a consequence, if we define $\Omega'=\Omega \cup \{q_2,q_3\} \backslash (\Omega \cap \{q_2,q_3\})$, then $\sigma_{\Omega'}=\sigma_{\{q_2,q_3\}}\cdot \sigma_\Omega$, and the point  $E'=(\sigma_\Omega)^{-1}(E)\in \ax(f_\p)$ satisfies  $\sigma_{\{q_2,q_3\}}(E)=\sigma_{\Omega'}(E')$; thus, the inequalities
\begin{eqnarray*}
\cosh( \dist (e_0,\sigma_{\Omega'}\ax(f_\p) )) & \leq &  e_0\cdot \sigma_{\Omega'}(E')\\
& < & \cosh(\dist(e_0,\sigma_{\Omega}\ax(f_\p)))-\delta
\end{eqnarray*}
imply that $\Omega$ is not a minimal element of $S$, or that $\Omega'$  does not satisfy one of the assertions $(1)$, $(2)$, $(3)$ in the definition of $\mathcal{S}$.


\smallskip

$\bullet$   We now replace $\Omega'$ by a new set $\Omega''$ such that $\sigma_{\Omega''}(e_0)\cdot E'$ does not increase and $\Omega''$ satisfies the defining properties of $\mathcal{S}$.  From \S~\ref{par:SigmaOmega} we know that
\[
e_0\cdot \sigma_{\Omega'}(E')=\sigma_{\Omega'}(e_0)\cdot E'=(e_0+\sum_{q\in \Omega'} ((e_0-e({p_1}))/2-e({q})))\cdot E'.
\]
If there are two distinct points $q,q'\in \Omega'$ such that 
\[
(e_0-e({p_1})-e({q})-e({q'}))\cdot E'\ge 0,
\] 
we can replace $\Omega'$ with $\Omega\backslash \{q,q'\}$, and this does not increase $\sigma_{\Omega'}(e_0)\cdot E'$. We can thus assume that $(e_0-e({p_1})-e({q})-e({q'}))\cdot E'< 0$ for all pairs of distinct points $(q,q')$ of $\Omega'$. 

If $q\in \Omega$ is in the first neighborhood of a point $q'$, the divisor $e({q})-e({q'})$ is effective, and intersects $E'$ non-negatively, because $E'$ is on the axis $\ax(f_\p)$ of $f\in \Cr_2(\k)$  (Lemma~\ref{EPositif}). If $q'$ does not belong to $\Omega'$, we can thus replace $\Omega'$ with $\Omega' \cup \{q'\}\backslash \{q\}$; again, this does not increase $\sigma_{\Omega'}(e_0)\cdot E'$. 

These replacements down, we get a new set $\Omega''$.  Let us show that $\Omega''$ belongs to $\mathcal{S}$. Property (4) is obviously satisfied. The fact that $(e_0-e({p_1})-e({q})-e({q'}))\cdot E'< 0$ for all pairs of distinct points $(q,q')$ in $\Omega''$ implies that $p_1$, $q$, $q'$ are not collinear (Assertion $(2)$).  Similarly,  the second family of modifications of $\Omega'$ shows that $\Omega''$ satisfies Assertion $(1)$. It remains to show that Assertion $(3)$ holds for $\Omega''$. Let $q_i$, $q_j$, and $q_k$  be three distinct points of $\Omega''$ such that $q_j$ and $q_k$ are proximate to $q_i$; then,  the divisor $e({q_i})-e({q_j})-e({q_k})$ is effective and intersects thus $E'$ non-negatively (Lemma~\ref{EPositif}). This yields 
\[
0>(e_0-e({p_1})-e({q_j})-e({q_k}))\cdot E'\ge (e_0-e({p_1})-e({q_i}))\cdot E',
\] 
which is impossible. Indeed, this implies that $ (e_0-e({p_1})-e({q}))\cdot E'<0$, where $q\not=p_1$ is a point which is either a proper point of $\P^2_\k$ or in the first neighbourhood of $p_1$ (choose either $q=q_i$ or $q$ such that $q_i$ is infinitely near to $q$), and this is impossible because $e_0-e({p_1})-e({q})$ is effective, as it corresponds to a line of $\P^2_\k$. This concludes the proof of $(\star)$.

\subsubsection{Strategy}
To prove Theorem~\ref{thm:thmD-rephrased}, we provide an algorithm which runs as follows. Start with $f$ and choose a minimal configuration $\Omega\in \mathcal{S}$. If there is an element $g$ in $\Bir(\P^2_\k)$ such that $g_\p(e_0)=\sigma_\Omega(e_0)$, we prove that the distance from $e_0$ to the axis is decreased by a multiplicative factor that depends only on $\delta$; we can thus replace $f$ by $gfg^{-1}$. If this is not the case, it is proved that a conjugate of $f$ satisfies $\deg(gfg^{-1})<4700 \lambda^5$, and the algorithm stops. 

\subsubsection{Algorithm: first case}\label{par:algo-1}
We  now take an element $\Omega\subset \mathcal{S}$, which is minimal in $\mathcal{S}$. By Property $(\star)$, the set $\Omega$ is not empty
(otherwise the Theorem~is proved, with $g$ equal to the identity), hence we have
 \[
 \cosh(\dist(e_0,\ax(\sigma_{\Omega}f_\p\sigma_{\Omega}))<\cosh(\dist(e_0,\ax(f_\p)))-\delta.
 \]
We write now explicitly $\Omega=\{p_2,\dots,p_{2m-1}\}$, we denote by  $E$ the projection of $e_0$ on $\sigma_{\Omega}\ax(f_\p )$, and we denote by $E'\in \ax(f_\p)$ the element $(\sigma_{\Omega})^{-1}(E)$; in general, this point differs from the projection of $e_0$ onto $\ax(f_\p)$.

Suppose that there exists  an element $g \in \Bir(\P^2_\k)$ with 
$(g_{\p})^{-1}(e_0)=\sigma_\Omega^{-1}(e_0)=\sigma_\Omega(e_0)$.
The point $g_\p(E')\in \ax ((g fg^{-1})_\p)$ satisfies
\[
e_0\cdot g_\p(E')= (g_{\p})^{-1}(e_0)\cdot E'=\sigma_\Omega(e_0)\cdot \sigma_\Omega(E)=e_0\cdot E,
\]
and this implies
\begin{equation}\label{eq:decrease-by-delta}
\cosh(\dist(e_0,\ax((g fg^{-1})_\p)))<\cosh(\dist(e_0,\ax(f_\p)))-\delta.
\end{equation}
We can thus replace $f$ with $g fg^{-1}$ and repeat the process (see below \S~\ref{par:thmD-conclusion}).

\subsubsection{Algorithm: second case}\label{par:algo-2}
Suppose now that such  a birational transformation $g$ does not exist.  Denote by $p_i$ the elements of $\Omega$ (including the
point $p_1$).
  
  \smallskip
  
$\bullet${\bf{An inequality.--}} Recall that $p_1$ is a proper point of $\P^2_\k$. Since $\Omega$ is in the family $\mathcal{S}$, properties $(1)$ to $(4)$ of Proposition~\ref{Prop:ExistenceJonq} are fulfilled. Thus, if there is no birational transformation $g$ such that 
\[
g^{-1}(e_0)=\sigma^{-1}(e_0)= m e_0-(m-1)e({p_1})-\sum_{i} e({p_i}),
\]
one of the two assumptions $(5)$--$(6)$ of Proposition~\ref{Prop:ExistenceJonq} is not satisfied. We now study these two possibilities. 

Write 
\[
E=\alpha_0 e_0-\sum \alpha_i e({p_i}),
\]
where $p_1,\dots,p_{2m-1}$ are as above and the remaining points $p_{k}$, $k\geq 2m$, are elements of $\B(\P^2_\k)$;  the $\alpha_i$'s are real non-negative numbers (apply Lemma~\ref{EPositif}).

If Assumption $(5)$ is not fulfilled, the number of points in $\{p_2,\dots,p_{2m-1}\}$ which belong, as proper or infinitely near points, to the exceptional divisor associated to $p_1$ is equal to $m+l$ with $0\le l\le m-2$; we write these points as $p_{i_1},\dots,p_{i_{m+l}}$. The divisor $e({p_1})-\sum_{j=1}^{m+l} e({p_{i_j})}$ is thus effective; hence, it intersects $E'$ non-negatively.

Applying $\sigma_\Omega$, we see that $E=\sigma_\Omega(E_f)$ intersects non-negatively 
\begin{eqnarray*}
\sigma_\Omega(e({p_1})-\sum_{j=1}^{m+l} e({p_{i_j}})) & = & (m-1)e_0-(m-2)e({p_1})- \\
& & \sum_{i=2}^{2m-1} e({p_i})-\sum_{j=1}^{m+l}(e_0-e({p_1})-e({p_{i_j}})).
\end{eqnarray*}
This gives  $(-1-l)\alpha_0+(l+2)\alpha_1\ge 0$, i.e.
\begin{equation}\label{EqFulfill4}
\frac{\alpha_1}{\alpha_0}\ge \frac{l+1}{l+2}.
\end{equation}

   If Assumption $(6)$ of Proposition~\ref{Prop:ExistenceJonq} is not satisfied, we obtain the existence of a curve of degree $k\ge 1$ with multiplicity $k-1$ at $p_1$ which passes  through $k+m+l$ of the points $\{p_2,\dots,p_{2m-1}\}$, for some $l\ge 0$; note that this curve is unique because it corresponds to the exceptional section of the Hirzebruch surface obtained by blowing-up $p_1$ and performing elementary links at $p_2,\dots,p_{2m-2}$. As before, this implies that 
   \[
   ke_0-(k-1)e({p_1})-\sum_{j=1}^{k+m+l} e({p_{i_j}})=k(e_0-e({p_1}))+e({p_1})-\sum_{j=1}^{k+m+l} e({p_{i_j}})
   \]
    intersects non-negatively $E'$. Applying $\sigma_\Omega$, we see that $E$ intersects non-nega\-ti\-vely
 \[
 k(e_0-e({p_1}))+(m-1)e_0-(m-2)e({p_1})-\sum_{i=2}^{2m-1} e({p_i})-\sum_{j=1}^{k+m+l}(e_0-e({p_1})-e({p_{i_j}})),
 \]
 and this leads again to  $(-1-l)\alpha_0+(l+2)\alpha_1\ge 0$ and to Equation~(\ref{EqFulfill4}).

  \smallskip

$\bullet${\bf{Upper bound on the degree.-- }} Coming back to the proof of Proposition~\ref{Prop:ExistenceJonq} we know that, in both cases, the problem is that the Hirzebruch surface obtained after blowing-up $p_1$ and performing elementary links at $p_2,\dots,p_{2m-2}$ is equal to a new Hirzebruch surface $\mathbb{F}_{1+2l}$ which does
 not coincide with $\mathbb{F}_1$; so the map $\sigma_\Omega$ does not correspond to a geometric Jonqui\`eres map. 
 
 To recover a well defined birational transformation of the plane, we choose  $2l$ general points of $\P^2_\k$,  that we call $p_{2m},\dots,p_{2m+2l-1}$. Then, we obtain the existence of a Jonqui\`eres transformation $g \in \Bir(\P^2_\k)$ such that 
 \[
 (g_{\p})^{-1}(e_0)=(m+l)e_0-(m+l-1)e({p_1})-\sum_{i=2}^{2m+2l-1} e({p_i}).
 \]
 
 We now estimate the degree of $g  f g^{-1}$ from above. The bound comes from the computation of  the intersection of $g_{\p}(E')\in \ax((g  f g^{-1})_\p)$ with $e_0$, a number which is smaller than, or equal to $\cosh(\dist(e_0,\ax((g  f g^{-1})_\p)))$.
 First, 
 \[
 g_{\p}^{-1}(e_0)=\sigma_\Omega(e_0)+l(e_0-e({p_1}))-\sum_{i=2m}^{2m+2l-1} e({p_i}),
 \]
  and therefore
 \[
 \sigma_\Omega(g_{\p}^{-1}(e_0))=e_0+l(e_0-e({p_1}))-\sum_{i=2m}^{2m+2l-1} e({p_i}).
 \] 
 In particular, we obtain
 \[
 e_0\cdot g_{\p}(E')=\sigma_\Omega(g_{\p}^{-1}(e_0))\cdot E\le\alpha_0+l(\alpha_0-\alpha_1).
 \]
 Inequality~\eqref{EqFulfill4} provides the estimate $\alpha_1\geq \alpha_0\cdot  \frac{l+1}{l+2}$, and therefore
\[
e_0\cdot g_{\p}(E')\le\alpha_0\cdot (1+l-l\cdot \frac{l+1}{l+2})=2\alpha_0 \cdot \frac{l+1}{l+2}<2\alpha_0=2e_0\cdot E.
 \] 
Since $E$ is the projection of $e_0$ on  $\sigma_\Omega\ax( f)$, one gets 
\[
\cosh(\dist(e_0,\ax((g fg^{-1})_\p)))<2\alpha_0=2\cosh(\dist(e_0,\sigma_\Omega\ax( f))).
\]

\begin{rem}
We have $\alpha_i+\alpha_j\le \alpha_0$ for all indices $i\not=j$. Indeed, one can choose $p_i$, $p_j$ proper points and obtain that $e_0-e(p_i)-e(p_j)$ is effective, hence has non-negative intersection with $E$. This also follows from Lemma~\ref{EPositif}, as $e_0-e(p_i)-e(p_j)$ is in the same $W_\infty$-orbit as $e(p_i)$.
\end{rem}

From Equation~\eqref{EqFulfill4}, we get $\alpha_1\ge \frac{\alpha_0}{2}$; hence the remark shows that  $\alpha_1\ge \alpha_i$ for all $i\ge 2$. 
By $(\star)$, this implies that $\deg(\sigma_\Omega f\sigma_\Omega)\leq 24\lambda^3$.

According to the Inequality~\eqref{Eq:Ineq-Axis-Deg} we have
\[
\alpha_0=\cosh(\dist(e_0,\sigma_\Omega\ax( f)))<\frac{2\deg(\sigma_\Omega f_\p\sigma_\Omega)}{\lambda-\frac{1}{\lambda}}
\]
 and 
\[
\sqrt{\frac{2\deg(g fg^{-1})}{\frac{1}{\lambda}+\lambda+2}}< \cosh(\dist(e_0,\ax((g fg^{-1})_\p))).
\] 
As a consequence,
\[
\sqrt{\frac{2\deg(g fg^{-1})}{\frac{1}{\lambda}+\lambda+2}}<\frac{4\deg(\sigma_\Omega h\sigma_\Omega)}{\lambda-\frac{1}{\lambda}}
\] 
and 
\[
\deg(g fg^{-1})<\frac{8\deg(\sigma_\Omega h\sigma_\Omega)^2}{(\lambda-\frac{1}{\lambda})^2}(\frac{1}{\lambda}+\lambda+2)<4700\lambda^5.
\]
Thus, we may stop the algorithm, since $\deg(g fg^{-1})<4700\lambda^5$.

\subsubsection{Conclusion}\label{par:thmD-conclusion}
Thus, 
\begin{itemize}
\item either \S~\ref{par:algo-1} apply, which means that we can find an element $g$ in the Cremona group such that the hyperbolic cosine of the distance from $e_0$ to the axis of $gfg^{-1}$ decrease by $\delta$. We can then repeat the process for $gfg^{-1}$ as long as $\deg(gfg^{-1})>24\lambda^3$;
\item or the process stops, which means that \S~\ref{par:algo-1} does not apply, and then \S~\ref{par:algo-2} shows that there exists an element $g$ in $\Bir(\P^2_\k)$ with $\deg(gfg^{-1})< 4700 \lambda^5$. 
\end{itemize}
To sum up, 
the theorem~\ref{thm:thmD-rephrased} is proved in at most $\cosh(\dist(e_0,\ax(f)))/\delta$ steps.

\subsection{Proof of Theorem~C}

\begin{figure}[h]
\input{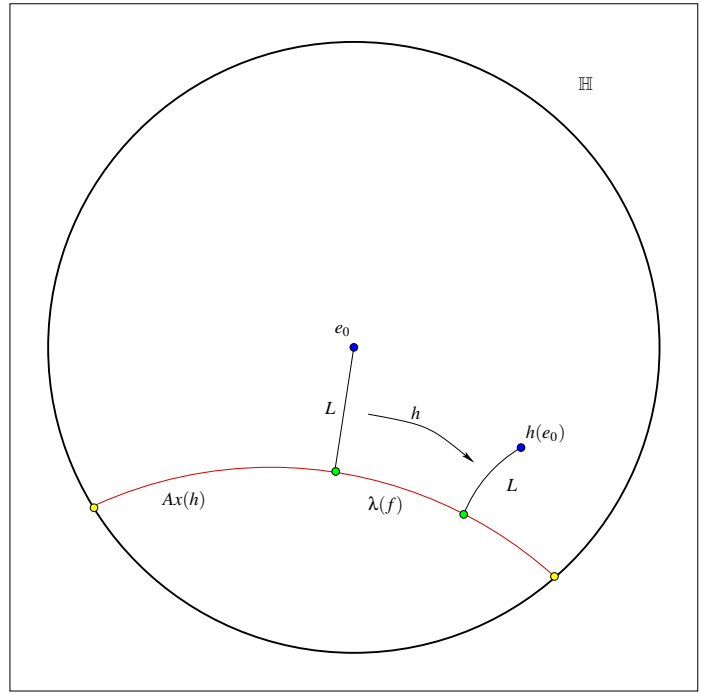}
\caption{{\sf{Axis and distances.--}} The blue points are the base point $e_0$ and its image by $h=gfg^{-1}$; the green points
are the projections of these blue points onto the axis of $h$.}
\end{figure}

To conclude the proof of Theorem~C, we need to prove the second assertion (using the first one, provided by Theorem~\ref{thm:thmD-rephrased}). Let $f$ be a loxodromic element of 
$\Bir(\P^2_\k)$. By the spectral gap property (see \S~\ref{par:firstgap}), $\lambda(f)\geq \lambda_l\simeq 1.176280$; hence, $\lambda(f^{86})>10^6$ and the first assertion of Theorem~C provides an element
$g$ in $\Bir(\P^2_\k)$ such that 
\[
\deg(g f^{86} g^{-1}) \leq 4700\,  \lambda(f^{86})^{5}.
\]
Let $h$ be equal to $gfg^{-1}$; we have $\lambda(h)=\lambda(f)$. Both $h$ and $h^{86}$ have the same axis, and we denote by $L$ the
distance from $e_0$ to it (see Figure~2). By definition, the distance from $e_0$ to $(h^{86})_\p(e_0)$ is at most equal to $\log(2 \deg(h^{86}))$ and,  by hyperbolicity of $\H_{\P^2_\k}$, it is bounded from below by 
$86 \log(\lambda(f)) + 2 L - 8\log(3)$. From the last inequality, we get 
\[
\log(2\deg(h^{86})) \leq \log(9400) + 5\cdot 86 \log(\lambda(f));
\]
hence,
\[
2L\leq 8\log(3)+\log(9400) + 4\cdot 86 \log(\lambda(f))\leq 18 + 344\log(\lambda(f)).
\]

With this upper bound in hands, we estimate 
\begin{eqnarray*}
\dist(e_0, (h)_\p e_0) & \leq & 2 L + \log(\lambda(f)) \\
&\leq & 18 + 345 \log(\lambda(f))
\end{eqnarray*}
and we obtain the inequality $\mcdeg(f)\leq \cosh(18+345\log(\lambda(f)))$. This concludes the proof of Theorem~C.

\section{Algebraic families of birational transformations and decreasing sequences of dynamical degrees}\label{par:ALG-FAM}

In this section, we prove the main corollaries of Theorem~C. This includes Theorem~D, which states that $\Lambda(\P^2_\k)$ is well ordered and that is is closed if $\k$ is uncountable and algebraically closed.

\subsection{The dynamical spectrum is well ordered}

\subsubsection{Well ordered subsets of $\R$}

Let $\Lambda$ be a subset of the real line $\R$. By definition, $\Lambda$ is well ordered if every subset $\Lambda_0$ of $\Lambda$ 
has a minimum $\min(\Lambda_0)\in \Lambda_0$; equivalently, $\Lambda$ is well ordered if it satisfies the descending chain condition: Every decreasing sequence $(\lambda_n)$ of elements of $\Lambda$ becomes eventually constant.

\begin{eg}
Consider the set of volumes of all compact riemannian manifolds of dimension $3$ with a metric of constant curvature $-1$. According to Jorgensen and Thurston, 
this set is infinite, well ordered, and contains accumulation points.
\end{eg}

\subsubsection{Dynamical degrees are well ordered}

\begin{thm}\label{thm:Well-Ordered}
Let $\k$ be a field. 
The subset $\Lambda\subset \R$ of dynamical degrees of all birational transformations 
of projective surfaces defined over $\k$ is well ordered.
\end{thm}

\begin{proof}
We may assume that $\k$ is algebraically closed.
Let $\lambda_n$, $n\geq 1$,  be a sequence of dynamical degrees. Suppose that $\lambda_{n+1}< \lambda_n$ for all indices $n$. Our goal is to prove that the number of terms in this sequence is bounded. 

Each element $\lambda_n$ of this sequence is the dynamical degree of some birational transformation $f_n\colon X_n\dasharrow X_n$. 
The subsequence of dynamical degrees $\lambda_n$ for which $X_n$ is not geometrically rational is bounded, because the set of dynamical degrees
of birational transformations of irrational surfaces is discrete. Thus, in what follows, we assume that $X_n$ is equal to the projective plane
$\P^2_\k$ for all $n\geq 1$. 

If $\lambda_{n_0}=1$ for some index $n_0$, the sequence contains  $n_0$ terms, because all dynamical degrees are larger than or equal to $1$.
Thus, we assume that $\lambda_n>1$ for all $n$. Let $\lambda_\infty$ be the limit of the sequence $(\lambda_n)$;  by Corollary~\ref{cor:thm-A-cor1}, 
\[
 \lambda_\infty+1\geq \lambda_n\geq\lambda_\infty\geq  \lambda_L\simeq 1.17628
\]
if $n$ is large enough ($n\geq n_0$).
Theorem~C  provides conjugates $g_n$ of $f_n$, such that 
\[
\deg(g_n)\leq \cosh(18+ 345\log(\lambda_\infty+1)).
\]
Hence, the degree of $g_n$ is bounded and, extracting a subsequence, we may assume that the degree of $g_n$ does not depend on $n$: There exists a degree $d$ such that $g_n$ is contained in the algebraic variety $\Bir_d(\P^2_\k)$ of elements of $\Bir(\P^2_\k)$ of degree $d$, for $n\geq n_0$.
 Junyi Xie  proved in \cite{Junyi-Xie:preprint} that the dynamical degree 
\[
\lambda\colon \Bir_d(\P^2_\k)\to [1,+\infty[
\]
is lower semi-continous for the Zariski topology. In other words, the level subsets 
\[
L(\beta)=\{h\in \Bir_d(\P^2_\k)\vert \quad \lambda(h)\leq \beta\}
\]
are Zariski closed (for all $\beta \geq 1$). As the sequence of dynamical degrees $(\lambda_n)$ is strictly decreasing,
we deduce that the sequence of Zariski closed sets $L(\lambda_n)\subset \Bir_d(\P^2_\k)$ decreases strictly. 
Since the Zariski topology is Noetherian, the sequence $(\lambda_n)$ contains only finitely many terms. \end{proof}

\subsection{Small Pisot numbers and spectral gaps}

Theorem~\ref{thm:Well-Ordered} implies that there are gaps in the dynamical spectrum on the right of every dynamical degree; the first gap occurs after the Lehmer number $\lambda_L$, the first gap on the right of a Pisot number occurs on the right of the plastic number $\lambda_P$. If one restricts the study to dynamical degrees in the Pisot family,  the following properties are corollaries of our previous results, known facts on Pisot numbers (see \cite{BDG:Book}), and a systematic study of quadratic birational transformations of the plane (see \cite{Diller:2011} and \cite{Blanc-Cantat:ToCome} for assertion (b)):
{\sl{
\begin{itemize}
\item[(a)] The Golden mean $\lambda_G$ is the smallest accumulation point in the set of Pisot numbers; it is an accumulation point from below, and from above. 

\item[(b)] All Pisot numbers below the Golden mean are realized as dynamical degrees of quadratic birational
transformations of the plane. 

\item[(c)] There is an $\epsilon>0$ such that $]\lambda_G, \lambda_G+\epsilon]$ does not contain any dynamical 
degree; hence, the infimum of the set 
\[
\{\lambda \in \Pis \mid \lambda \; {\text{is not a dynamical degree}}\}
\]
is equal to $\lambda_G$.
\end{itemize}
}}

\subsection{The dynamical spectrum is closed}

We now prove that {\sl{the dynamical spectrum $\Lambda(\P^2_\k)$ is closed if $\k$ is uncountable and algebraically closed}}.

Let $(f_i)_{i\geq 0}$ be a sequence of birational transformations of the projective plane such that $\lambda(f_i)$ converges
towards a real number $\lambda_\infty$. Our goal is to construct an element $h$ in $\Bir(\P^2_\k)$ such that $\lambda(h)=\lambda_\infty$.
Thus, we may (and do) assume that the numbers $\lambda_i:=\lambda(f_i)$ form a strictly increasing sequence converging to  $\lambda_\infty>1$.
Theorem~C applies: Changing each $f_i$ into a conjugate element of $\Bir(\P^2_\k)$, we assume that 
\[
\deg(f_i)\leq D
\]
where $D$ depends on the supremum of the $\lambda_i$ but does not depend on $i$. Replacing $(f_i)$ by a subsequence, we assume that the $f_i$ 
have the same degree $d$ (with $2\leq d\leq D$). 

The set $\Bir_d(\P^2_\k)$ of all birational transformations of degree $d$ is naturally endowed with the structure of an algebraic variety (see \cite{Blanc-Furter}); we denote by $F$ the Zariski closure of the set $\{ f_i\}_{i\geq 0}$ in $\Bir_d(\P^2_\k)$. The dimension of $F$ is positive because
the $\lambda_i$ are pairwise distinct; extracting a new subsequence, we assume that $F$ is irreducible. 

For each positive integer $n$, consider the set 
\[
X_n=\{g \in \Bir_d(\P^2_\k) \mid \deg(g^n)\leq \max_i (\deg(f_i^n)) \}.
\]
By \cite{Blanc-Furter},  this set is Zariski closed\footnote{With the notation of \cite{Blanc-Furter}, consider the set $\tilde{H_d}$ of triples $(p,q,r)$ of homogeneous polynomials of degree $d$ such that $f:[x:y:z]\mapsto [p:q:r]$ is a birational map; let $H_d$ be the quotient of $\tilde{H_d}$
by the equivalence relation for which two triples are equivalent if they are multiple of each other by a non-zero constant. 
Then $(p,q,r)\mapsto f$ is a map from $H_d$ to the set of birational transformations of degree $\leq d$. Denote by $H_{d,d}$ the subset of $H_d$ made of triples of $(p,q,r)$ that give rise to a birational map of degree $d$ exactly; this set is a Zariski open subset of $H_d$, and the projection $\pi_d\colon H_{d,d}\to \Bir_d(\P^2_\k)$ is an isomorphism. The map $H_{d}\to H_{d^n }$ that applies $f$ to $f^n$ is a morphism. Since $\Bir_{\leq \ell}(\P^2_\k)$
is closed in $\Bir(\P^2_k)$ (see Corollary 2.8 of \cite{Blanc-Furter}), one deduces that $X_n$ is closed.}; 
 hence $F$ is contained in $X_n$ for all $n\geq 1$. 
Similarly, the set
\[
Y_n= \{ g \in F \mid \deg(g^n) \leq \max_i ( \deg(f_i^n)) -1\}
\]
is a Zariski closed subset of $F$, and is a strict subset of $F$ because its complement contains at least one $f_j$. 
Since $\k$ is not countable, there is at least one element $h$ in 
$F\setminus \cup_{n\geq 1} Y_n$. This birational transformation satisfies 
\[
\deg(h^n) = \max_i (\deg(f_i^n))
\]
for all $n\geq 1$. This implies that $\deg(f_i^n)\leq \deg(h^n)$ for all indices $i$ and $n$; in particular, 
$\lambda(f_i)\leq \lambda(h)$ for all $i$, and 
\[
1< \lambda_\infty \leq \lambda(h).
\]
Thus, $h$ is a {\sl{loxodromic}} transformation of degree $d$. 

\begin{lem} Let $\k$ be a field and $d\geq 2$ be an integer. There exists a constant $\Delta(d)$ such that, 
for all  loxodromic element $g \in \Bir(\P^2_\k)$ of degree $d$, 
\begin{eqnarray}
\dist(e_0, \ax(g)) & \leq &  \Delta(d)/2 \\
\vert  \log(\deg(g^m))-m\log(\lambda(g))\vert & \leq & m\log(\lambda(g^m)) + \Delta(d)
\end{eqnarray}
for all $m\geq 1$.
\end{lem}

\begin{proof}
From the spectral gap property $\lambda(g)\geq \lambda_L=1.176280$. Thus, hyperbolic geometry implies
that $\dist(e_0, g_\p(e_0))$ goes to infinity with $\dist(e_0,\ax(g))$; more precisely, there is a uniform constant 
$\epsilon >0$ such that 
\[
\dist(e_0,  g_\p(e_0)) \geq 2\dist(e_0,\ax(g)) +\log(\lambda(g))-\epsilon.
\] 
Since $\dist(e_0, g_\p e_0)$ is bounded from above by $\log(2\deg(g))$, the first upper bound follows. 

The triangular inequality implies 
\[
\dist(e_0, (g^m)_\p(e_0)) \leq 2 \dist(e_0, \ax(g_\p)) + m\log(\lambda(g))
\]
and hyperbolicity implies
\[
\dist(e_0, (g^m)_\p(e_0)) \geq m\log(\lambda(g)) + 2 \dist(e_0, \ax(g_\p)) - \delta
\]
where $\delta$ is a uniform constant ($\delta<100$). The result follows.
\end{proof}

Apply this lemma to $h$ and to the $f_i$: 
\[
m \log(\lambda(h)) - \Delta(d) \leq \log(\deg(h^m)) \leq m \log(\lambda(h)) + \Delta(d) ,
\]
\[
m \log(\lambda(f_i)) - \Delta(d) \leq \log(\deg(f_i^m)) \leq m \log(\lambda(f_i)) + \Delta(d) .
\]
Let $\epsilon$ be a positive real number, and $m$ be a positive integer such that $\epsilon \geq 2\Delta(d)/m$.
Then, there exists $i$ such that $\deg(f_i^m)=\deg(h^m)$, and we get 
\[
m\log(\lambda(h))-\Delta(d) \leq \log(\deg(h^m)) = \log(\deg(f_i^m)) \leq m\log(\lambda(f_i)) + \Delta(d). 
\]
Hence, $\log(\lambda(h))\leq \log(\lambda(f_\infty))+ \epsilon$. Since this inequality holds for all $\epsilon>0$, 
we obtain $\lambda(h)\leq \lambda_\infty$, and thus $\lambda(h)= \lambda_\infty$, as desired. 

\subsection{Increasing approximation by Salem dynamical degrees}\label{par:Approx-Salem}
The set $\Pis$ is contained in the closure of the 
set $\Sal$. In this section, we show that the same property holds for dynamical degrees:

\begin{thm}\label{thm:sal-pis}
Let $\k$ be an algebraically closed field of characteristic $0$. 
Let $\beta$ be an element of $\Lambda^P(\P^2_\k)$. There exists a strictly increasing sequence $(\alpha_n)_{n\geq 0}$ 
of elements of $\Lambda^S(\P^2_\k)$  that converges towards $\beta$.
\end{thm}

\begin{cor} Let $\k$ be an algebraically closed field of characteristic $0$ and let $X$ be a projective surface defined over $\k$.
The dynamical spectrum $\Lambda(X)$ is contained in the closure of the set of dynamical degrees $\lambda(f)$ of automorphisms
of surfaces which are birationally equivalent to $X$.
\end{cor}

\begin{proof}[Proof of the corollary] If $X$ is rational, this is a consequence of Theorem~\ref{thm:sal-pis} and Theorem~A. If $X$ is not rational, all dynamical degrees are realized by dynamical degrees of automorphisms of surfaces which are birationally equivalent to $X$ (see Section~\ref{par:not-rational}). \end{proof}

The proof of Theorem~\ref{thm:sal-pis}  is given in $\S\ref{thm:sal-pisproofgeneral}$ when $\beta$ is not a reciprocal quadratic integer. The case of 
reciprocal quadratic integers is dealt with in  $\S\ref{thm:sal-pisproofquad}$.  

\subsubsection{Pisot numbers which  are not reciprocal quadratic integers}
\label{thm:sal-pisproofgeneral}

Let $\beta\in \Lambda^P(\P^2_\k)$ be a Pisot number which is not a reciprocal quadratic integer; thus, $\beta$
is the dynamical degree of a birational transformation of $\P^2_\k$, but is not the dynamical degree of an
automorphism of rational projective surface.

Choose $f\in \Bir(\P^2_\k)$ with $\lambda(f)=\beta$, denote by $d$ the degree of $f$, by $p_i$ and $q_i$ the base points of $f$ and $f^{-1}$ ($1\leq i \leq m$), and write
\begin{eqnarray*}
f_\p (e_0) & = & d e_0 - \sum_{i=1}^m a_i e(q_i) \\
f_\p (e(p_i)) & = & d_i e_0 - \sum_{i=1}^m c_{i,j} e(q_j).
\end{eqnarray*}
The multiplicities $a_i$ are positive integers; the $c_{i,j}$ are non-negative integers.  

Say that $q_i$ has an {\bf{infinite length}} if $f^l(q_i)$ is not a base point of $f$ for all $l\geq 0$, and say that $q_i$ has a {\bf{finite length}} (equal to $\ell_i$) if $f^l(q_i)$ is not a base point of $f$ for $0\leq l \leq \ell_i-1$ but $f^{\ell_i}(q_i)$ is one of the base points $p_j$ of $f$. 
If all the base points $q_i$, $1\leq i\leq m$, have a finite length, one can blow up the points $f^l(q_i)$, $1\leq l \leq \ell_i$ to get a new surface
on which $f$ is an automorphism. Since $f$ is not conjugate to an automorphism, at least one of the base points $q_i$ has an infinite length. 

Order the base points $q_i$ in such a way that $q_1$, ..., $q_n$ have infinite length and $q_{n+j}$  has finite length $\ell_{n+j}$ for 
$j=1, \ldots, m-n$. Then, number the $p_j$ in such a way that $p_{n+j}=f^{\ell_j}(q_{n+j})$ for all $j\geq 1$. We shall now construct
a sequence of birational transformations $f_k$ such that 
\begin{itemize}
\item each $f_k$ is conjugate to an automorphism ;
\item $\lambda(f_k)$ converges to $\lambda(f)=\beta$ as $k$ goes to $+\infty$.
\end{itemize}
The idea is to transform the points $q_i$ into base points of finite length $\ell_i$ for $i\leq n$, but with length $\ell_i=k$ going to $+\infty$.

For this purpose, define 
\begin{eqnarray*} 
A & = & \{e_0\}\cup \bigcup\limits_{i=n+1}^m \left(\bigcup\limits_{j=0}^{l_i-1}\left\{ f^j_\p(e(q_i))\right\}\right)\\
 B_j&=& \bigcup\limits_{i=1}^n\left\{ f^j_\p(e(q_i))\right\} \mbox{ for any } j\ge 0\\
 C&=&\bigcup\limits_{i=1}^n \{e(p_i)\}.
 \end{eqnarray*}
The elements of these three sets are linearly independent in $\z_{\P^2_\k}$. In particular, $A$ is a basis of the sub-module 
$V_A$ spanned by $A$ in $\z_{\P^2_\k}$. Similarly, $B_j$ (resp. $C$) is a basis of the sub-module $V_{B_j}$ spanned by $B_j$ (resp. $C$)
for all $j\geq 0$. The map $f_\p$ restricts to an isomorphism between $V_C\oplus V_A$ and $V_A\oplus V_{B_0}$, 
 and also to an isomorphism 
\[
V_C\oplus V_A\oplus V_{B_0}\oplus\dots \oplus V_{B_k}\stackrel{f_{*}}{\longrightarrow} V_A\oplus V_{B_0}\oplus\dots \oplus V_{B_{k+1}}.
\]
Writing $V_k=V_C\oplus V_A\oplus V_{B_0}\oplus\dots \oplus V_{B_k}$, we define a linear transformation 
$F_k\in \Aut(V_k)$ by 
\[
F_k=\pi_k\circ f_\p,
\]
where $\pi_k\colon V_{k+1}\to V_k$ is the isomorphism defined by $\pi_k(f^{k+1}_\p(e(q_i)))=e(p_i)$ for $i=1,\dots,n$ and 
$\pi_k(x)=x$ for $x\in V_A\oplus V_{B_0}\oplus\dots \oplus V_{B_k}=V_k\cap V_{k+1}$. 
 
Since $f_\p$ preserves the intersection form of $\z_{\P^2_\k}$, $F_k$ preserves the intersection form of $V_k$. This latter space is of Minkowski type: An orthonormal basis is given by $e_0$ (of self-intersection $+1$) and by the other elements of $A,B_j,C$ (each of self-intersection $-1$). 
Since $f_\p$ satisfies the Noether equalities~\eqref{eq:NoetherEquality}, so does~$F_k$:
\[
 \sum\limits_{i=1}^m a_i^2 = d^2-1;\quad \; 
 \sum\limits_{i=1}^m a_i =  3d-3.
 \]

For a fixed integer $k\geq 1$, denote by $r+1$ the dimension of $V_k$. Then,
denote by $W_r$ the subgroup of $W_\infty$ generated by 
\begin{itemize}
\item the finite group of permutations of the set $e(q)$, where $e(q)$ runs over the elements of $A\setminus\{e_0\}$, of $B_j$ ($j\leq k$), and of $C$;
\item the involution $\sigma_0$ (with base points $p_1$, $p_2$, $p_3$ chosen among the base points $p_j$);
\item the involutions $\tau_{p,q}$ for $e(p)$ and $e(q)$ in the sets   $A\setminus\{e_0\}$, $B_j$ ($j\leq k$) and~$C$.
\end{itemize}
This group is isomorphic to the Coxeter group of the Dynkin diagram $T_{2,3,r-3}$ introduced in Section~\ref{par:Salem-in-Cremona}.
Since $F_k$ satisfies the Noether equalities, $F_k$ is an element of the Coxeter group $W_r$ (this is a version of Nagata's theorem mentioned in
\S~\ref{par:Salem-in-Cremona}; see \cite{DO} and the proof of Lemma~\ref{Lem:MultiplicitiesPositive}). 

By Uehara's theorem (see \cite{Uehara}), there is an element $f_k\in \Bir(\P^2_\k)$ for which $\lambda(f_k)$ is equal to the
spectral radius of the linear transformation $F_k$; moreover, $f_k$ is conjugate to an automorphism of a projective rational surface $X_k$. 

\begin{lem}\label{lem:cvlambda}
If $\beta$ is not a reciprocal quadratic integer, the sequence $(\lambda(f_k))$ converges towards $\beta$ as $k$ goes to $+\infty$, and it contains a sub-sequence that increases strictly towards $\beta$.
\end{lem}

This lemma concludes the proof of Theorem~\ref{thm:sal-pis} when $\beta$ is not a reciprocal quadratic integer. Indeed, $\lambda(f_k)$ is
not a quadratic integer if $k$ is large, because the set of reciprocal quadratic integers is discrete; hence, $(\lambda(f_k))$ contains
a strictly increasing sequence of Salem numbers that converges towards $\beta$.  

\begin{proof}[Proof of Lemma~\ref{lem:cvlambda}]

Let $\lambda_k=\lambda(f_k)$.
This number is the largest real eigenvalue of $F_k\in \Aut(V_k\otimes \R)$.

The map $f_\p$ preserves $V_C\oplus V_A\bigoplus\limits_{i=1}^{\infty} V_{B_i}$ and its matrix can be written as a matrix by blocks as follows:
 \[
 \left(\begin{array}{ccccc}
  0& 0&  0 & &\\
  M& N & 0 & &\\
  P& Q& 0& &\\
  0  &0 &I&\ddots&\\
  & & &\ddots\end{array}\right)
  \]
  
Let $v$ be an eigenvector of $f_\p$ with eigenvalue $\lambda(f)$; such a vector exists (and is isotropic) in $\zz_{P^2_\k}$. Decompose $v$ as
$  v=0 + v_A + v_{B_1} + v_{B_2} + \ldots$ with respect to the direct sum $V_C\oplus V_A\bigoplus\limits_{i=1}^{\infty} V_{B_i}$. We obtain the
system of equations
\[
\left(\begin{array}{c} 0 \\ \lambda(f)\cdot v_A \\ \lambda(f)\cdot v_{B_1}\\\lambda(f)\cdot v_{B_2}\\ \dots \end{array}\right)=\lambda(f)\cdot v=f_{*}(v)=\left(\begin{array}{c} 0 \\ Nv_A \\ Qv_A\\ v_{B_1}\\ \dots \end{array}\right),
\]
from which we deduce that $v_A$ is an eigenvector of $N\in \GL(V_A)$, and $v_{B_{i}}=Qv_A/\lambda(f)^i$ for $i\ge 1$. Moreover, 
$v_A\neq 0$ because $v$ intersects $e_0$ positively. Thus, $\beta$ is a root of the characteristic polynomial $\det(tI-N)$.

 The matrix of $F_k$, acting on $V_k=V_C\oplus V_A\oplus V_{B_1}\oplus\dots \oplus V_{B_k}$, is
\[
M_{F_k}\left(\begin{array}{cccccc}
  0& 0&  & &&I\\
  M& N & & &\\
  P& Q& 0& &\\
  &&I&\ddots&\\
  & &&\ddots& 0\\
  &&&&I & 0\end{array}\right).
\]
Its characteristic polynomial $det(xI-M_{F_k})$ is equal to
 \[\det \left(\begin{array}{cccccc}
  xI& 0&  & &&-I\\
  -M& xI-N & & &&0\\
  -P& -Q& xI& x^2I&\dots &x^kI\\
  &&-I&0&&0\\
  & &&\ddots&\ddots&\vdots \\
  &&&&-I & 0\end{array}\right) 
\]
\[=\det\left(\begin{array}{cccccc}
  xI& 0& -I\\
  -M& xI-N &0\\
  -P& -Q& x^kI\\
 \end{array}\right)=\det\left(\begin{array}{cccccc}
 0& 0& -I\\
  -M& xI-N &0\\
  -P+x^{k+1}I& -Q& x^kI\\
 \end{array}\right)
 \]
Let $P(s,t)$ be the polynomial function in two variables that is defined by 
\[
P(s,t)= \det\left(\left(\begin{array}{cccccc}
 0& 0& -I\\
  -M& tI-N &0\\
  -sP+tI& -sQ& I\\
 \end{array}\right)\right)
\] 
The characteristic polynomial of $F_k$ is equal to $x^kP(1/x^k,x).$ Hence, 
\[
P(\lambda(f_k)^{-k}, \lambda(f_k))=0.
\]
Moreover, $P(0,t)=t^l\det(tI-N)$ for some integer $l$. Hence, the biggest real root of the polynomial $P(0,t)$ is $\beta$, and this root is simple.

We choose real numbers $\beta^{-},\beta^{+}$ with $1<\beta^{-}<\beta<\beta^{+}$ and define $\delta_k$ to be
$$\delta_k=\max_{t\in [\beta^{-},\beta^{+}]} |P(0,t)-P(1/t^k,t)|.$$
By construction, $\lim_{k\to \infty} \delta_k=0$. Hence, for large $k$ the rational function $P(1/t^k,t)$ has a real root $\beta_k$ near $\beta$, and $\lim_{k\to \infty} \beta_k=\beta$. Since $F_k$ is an element of the Coxeter group $W_r$, it has at most one real root bigger than $1$. Hence, $\beta_k=\lambda(f_k)$.

Thus, $\lambda(f_k)$ converges towards $\beta$. Since $\beta$ is a Pisot number and is not a reciprocal quadratic integer, it is not equal to the
dynamical degree of an automorphism. Thus, $\lambda(f_k)\neq \beta$ for all $k$, and one can extracts a sub-sequence from $(\lambda(f_k))$ whose members are pairwise distinct. Theorem~D implies that the sequence is strictly increasing. \end{proof}

\subsubsection{Reciprocal quadratic integers}
\label{thm:sal-pisproofquad}
It remains to prove Theorem~\ref{thm:sal-pis} for reciprocal quadratic integers. 

We fix integers $m,k\ge 2$ and choose a set $\Delta\subset \B(\P^2)$ of $2m-1+(m-2)k$ distinct points that we denote by
$$\begin{array}{c}
\Delta=\{q_i\}_{i=1}^{2m-1}\cup  
\{a_{i,j}\}_{i=1,\dots,m-2,j=1,\dots,k}.\end{array}$$
We choose $2m-1$ from these points, that we write $p_1,\dots,p_{2m-1}$. These are \begin{center}$p_i=a_{i,k}$ for $i=1,\dots,m-2$, $p_i=q_i$ for $i=m-1,\dots,2m-1.$\end{center}
Then we construct an element $h\in W_\infty$ defined by
$$\begin{array}{rcl}
h(e_0)&=&m e_0-(m-1)e({q_1})-\sum_{i=2}^{2m-1} e({q_i});\\
h(e({p_1}))&=&(m-1) e_0-(m-2)e({q_1})-\sum_{i=2}^{2m-1} e({q_i});\\
h(e(p_i))&=&e_0-e({q_1})-e({q_i})\;  \mbox{ for } i=2,\dots,2m-1;\\
h(e(q_i))&=&e(a_{i,1})\;  \mbox{ for } i=1,\dots,m-1;\\
h(e(a_{i,j}))&=&e(a_{i,j+1})\;  \mbox{ for } i=1,\dots,m-1,j=1,\dots,k-1.\\
h(e(r))&=&e(r)\; \mbox{ for } r\in \B(\P^2)\setminus\Delta.
\end{array}$$
Note that $h$ preserves the $\Z$-module $W$ generated by $e_0$ and the $\{e(r)\} _{r\in \Delta}$. It corresponds then to an element of the Coxeter group associated to these points. By Uhehara (see \cite{Uehara}), there is an element $f_{m,k}\in \Bir(\P^2_\k)$ for which $\lambda(f_{m,k})$ is equal to the
spectral radius $\lambda_{m,k}$ of the linear transformation $h$; moreover, $f_{m,k}$ is conjugate to an automorphism of a projective rational surface. Hence, Theorem~\ref{thm:sal-pis} follows from the following lemma in the case of reciprocal quadratic integers.

\begin{lem}
For integers $m\ge 2$, the sequence $(\lambda_{m,k})_{k\geq 1}$ converges towards
the largest root $\lambda_{m,\infty}$ of $P_m(x)=x^2-(m+1)x+1$, and 
 $\lambda_{m,k}$ is a Salem number if $k$ is large enough.
\end{lem}
\begin{proof}
Denote by $W'\subset W$ the sub-$\Z$-module whose basis is
\[
\begin{array}{l}e(p_1),\; \sum\limits_{i=2}^{m-2} e(p_i),\; e_0,\; \sum\limits_{i=m-1}^{2m-1} e(p_i)=\sum\limits_{i=m-1}^{2m-1} e(q_i),\; e(q_1),\; \sum\limits_{i=2}^{m-2} e(q_i),\\
e(a_{1,1}),\; \sum\limits_{i=2}^{m-2} e(a_{i,1}),\; \dots,\; 
e(a_{1,k-1}),\; \sum\limits_{i=2}^{m-2} e(a_{i,k-1}).\end{array}
\]
Then, $W'$ is invariant by $h$, and the matrix of $h$ relative to the above basis is 
$$M_h=\left(\begin{array}{rrrr|rrr|rr}
0 &0& 0&0& 0 & \dots & 0& 1 & 0\\
0 &0& 0&0& 0 & \dots & 0& 0 & 1\\
\hline
m-1&m-3& m& m+1& 0 &  \dots & 0& 0 & 0\\
-1& 0& -1& -1& 0 &  \dots & 0& 0 & 0\\
-(m-2)& -(m-3)& -(m-1)& -(m+1)& 0 &  \dots & 0& 0 & 0\\
-1& -1& -1& 0& 0 &  \dots & 0& 0 & 0\\
\hline
0& 0& 0& 0 & 1&&0&0&0\\
\vdots& \vdots& \vdots& \vdots & & \ddots&  &\vdots&\vdots\\
0& 0& 0& 0 & &  & 1 & 0 & 0\\
\end{array}\right)$$
A computation similar to the one done in the proof of Lemma~\ref{lem:cvlambda} shows that its characteristic polynomial $\det(xI-M_{h})$ is equal to
\[\det \left(\begin{array}{rrrrrr}
x &0& 0&0&  -1 & 0\\
0&x &0& 0&0  & -1\\
-(m-1)&-(m-3)& x-m& -(m+1)&  0& 0 \\
1& 0& 1& x+1&  0& 0\\
m-2& m-3& m-1& m+1&  x^k& 0 \\
1& 1& 1& 0& 0 & x^k
\end{array}\right)\]
and is therefore equal to 
\[
x^{2k+2}(x^2-x(m-1)+1)+x^{k+1}((m-1)x^2-4x+(m-1))+(x^2-(m-1)x+1)
\]
Fixing $m$, we see that the sequence $(\lambda_{m,k})_{k}$ converges towards $\lambda_{m,\infty}$, and that $\lambda_{m,k}\neq \lambda_{m,\infty}$ for $k$ large. Each $\lambda_{m,k}$ being the spectral radius of an element in a Coxeter group $W_{r_k}$, it is either equal to $1$, to a quadratic integer or a Salem number. The set of quadratic integers being discrete, and $\lambda_{m,k}$ being different from $\lambda_{m,\infty}$, the  $\lambda_{m,k}$ are all Salem numbers for $k$ large enough.
\end{proof}

\section{Appendix: Modular groups, dilatations, and volumes}\label{par:modular-group}

\subsection{}

The Cremona group $\Bir(\P^2_\k)$ acts faithfully on the hyperbolic space $\H_{\P^2_\k}$; this space contains all classes
\[
\frac{1}{\sqrt{C\cdot C}}[C]
\]
where $C$ is a curve with positive self-intersection on some rational surface. Similarly, the modular group (or mapping class group)
${\rm{Mod}}(g)$ of the closed, connected, and orientable surface $\Sigma_g$ of genus $g\geq 2$ acts by isometries on
several metric spaces, for instance on the Teichm\"uller space, endowed with its Teichm\"uller metric. \footnote{It also acts on the complex of curves of the surface, a metric space which is Gromov hyperbolic (see \cite{Masur-Minsky:1999}).}

The comparison of those two isometric actions provides a fruitful analogy between $\Bir(\P^2_\k)$ and ${\rm{Mod}}(g)$ for $g\geq 2$ (see \cite{Cantat:Annals, Cantat:ECM}). 
In this analogy,  loxodromic elements $f\in \Bir(\P^2_\k)$ correspond to pseudo-Anosov classes $\varphi\in {\rm{Mod}}(g)$. The
dynamical degree $\lambda(f)$ may be compared to the dilatation factor $\lambda(\varphi)$  of $\varphi$;
both $\lambda(f)$ and $\lambda(\varphi)$ are algebraic numbers: The degree of $\lambda(f)$ is bounded from above by the 
Picard number of a surface on which $f$ is conjugate to an algebraically stable transformation, while the degree of $\lambda(\varphi)$
is at most  $6g-6$. 

Theorem~A may be compared to Franks and Rykken result, according to which a pseudo-Anosov homeomorphism 
$\Phi\colon \Sigma_g\to \Sigma_g$ with a quadratic dilatation factor and with orientable stable and unstable foliations is semi-conjugate, via
a ramified cover, to a linear automorphism of a torus (see \cite{Franks-Rykken:1999}). As for birational transformations, the infimum of $\lambda(\varphi)$ when $\varphi$ describes the set of pseudo-Anosov classes that are composition of   Dehn-multitwists is the Lehmer number (see \cite{Leininger:2004}).

\subsection{} 

Another measure of the complexity of a pseudo-Anosov isotopy class $\varphi$ is obtained as follows. According to Thurston and Mostow, the three-dimensional manifold 
\[
M_\varphi=\left(\Sigma_g\times [0,1]\right)/(x,0)=(\Phi(x),1)
\]
(where $\Phi$ is a diffeomorphism of $\Sigma_g$ in the isotopy class $\varphi$) admits a unique hyperbolic metric (a riemannian metric of constant curvature $-1$). The volume of $M_\varphi$ with respect to this riemannian metric is a positive real number ${\rm{vol}}(\varphi)$; this volume is, up to a bounded multiplicative error, the translation length of $\varphi$ on the Teichm\"uller space with respect to the Weil-Petersson metric (see \cite{Brock:2003}). Jorgensen and Thurston proved that the set of all volumes ${\rm{vol}}(M)$ of all compact hyperbolic three-manifolds is infinite countable, contains accumulation points, and is well ordered (see \cite{BP:Book}). Thus, the set $\{{\rm{vol}}(\varphi)\}$ where $\varphi$ describes the set of pseudo-Anosov classes of some higher genus surface is well ordered too; moreover, this set is not discrete (consider sequences  
${\rm vol}(\phi\circ \tau^n)$ where $\tau$ is a Dehn twist).  This parallels Theorem~C. Moreover, as shown in Section~\ref{par:Approx-Salem},
accumulation points in $\Lambda(\P^2_\k)$ are obtained by replacing orbits of base points with an infinite length by orbits with finite length. For volumes of hyperbolic manifolds, one obtains accumulation points by Dehn fillings of cusps. Thus, cusps correspond to base points of infinite length in this dictionary. 
 
\subsection{}

It may also be interesting to compare our results to the description obtained by Thurston of the possible topological entropies of multimodal continuous maps of the interval $[0,1]$ into itself which are postcritically finite (see \cite{Thurston:entropies}). Those entropies are logarithms of Perron numbers, and all Perron numbers $\lambda>1$ are realized. Thus, in this setting, there is no gap phenomenon similar to the gaps in the dynamical spectrum $\Lambda(\P^2_\k)$.  

 






%
%

\vspace{8mm}

\bibliographystyle{plain}
\bibliography{bibliofinal}
\nocite{}

\end{document}